\documentclass[twoside,11pt,a4paper]{article}
\usepackage{jmlr}
\usepackage[utf8]{inputenc}\usepackage{macros}

\definecolor{darkblue}{rgb}{0.0,0.0,0.7}
\RequirePackage[linktocpage=true,
hyperfootnotes=false,
colorlinks = true,citecolor = darkblue,urlcolor = darkblue, linkcolor=magenta, backref=page,
]{hyperref}
\usepackage[square, numbers, comma, sort&compress]{natbib}

\usepackage[backref=page]{hyperref}

\renewcommand*{\backrefalt}[4]{    \ifcase #1     \or (Cited on page:~#2.)     \else (Cited on pages:~#2.)    \fi    }

\definecolor{linkequation}{rgb}{0.0, 0.0, 1.0}
\usepackage{xcolor}
\newcommand*{\SavedEqref}{}
\let\SavedEqref\eqref
\renewcommand*{\eqref}[1]{  \begingroup
    \hypersetup{
      linkcolor=linkequation,
      linkbordercolor=linkequation,
    }    \SavedEqref{#1}  \endgroup
}

\newcommand*{\SavedRef}{}
\let\SavedEqref\ref
\renewcommand*{\ref}[1]{  \begingroup
    \hypersetup{
      linkcolor=linkequation,
      linkbordercolor=linkequation,
    }    \SavedRef{#1}  \endgroup
}

\usepackage[colorinlistoftodos]{todonotes}
\usepackage[symbol*]{footmisc}
\DefineFNsymbolsTM{myfnsymbols}{              }\usepackage{multicol}
\setfnsymbol{myfnsymbols}
\usepackage{tabularx}
\allowdisplaybreaks

\usepackage{authblk}

\begin{document}

\title{Bounds in Wasserstein distance for locally stationary processes}

\author[1,2]{Jan Nino G. Tinio}
\author[1]{Mokhtar Z. Alaya}
\author[1]{Salim Bouzebda}
\affil[1]{\footnotesize Université de Technologie de Compiègne,  \authorcr
LMAC (Laboratoire de Mathématiques Appliquées de Compiègne), CS 60 319 - 60 203 Compiègne Cedex
}
\affil[2]{\footnotesize Department of Mathematics, Caraga State University, Butuan City, Philippines\linebreak\linebreak\texttt{e-mails}: \texttt{jan-nino.tinio@utc.fr}, \quad \texttt{alayaelm@utc.fr}, \quad \texttt{salim.bouzebda@utc.fr}}

\maketitle

\begin{abstract}
Locally stationary   (LSPs) constitute an essential modeling paradigm for capturing the nuanced dynamics inherent in time series data whose statistical characteristics, including mean and variance, evolve smoothly across time. In this paper, we introduce a novel conditional probability distribution estimator specifically tailored for LSPs, employing the Nadaraya–Watson (NW) kernel smoothing methodology. The NW estimator, a prominent local averaging technique, leverages kernel smoothing to approximate the conditional distribution of a response variable given its covariates. We rigorously establish convergence rates for the NW-based conditional probability estimator in the univariate setting under the Wasserstein metric, providing explicit bounds and conditions that guarantee optimal performance. Extending this theoretical framework, we subsequently generalize our analysis to the multivariate scenario using the sliced Wasserstein distance, an approach particularly advantageous in circumventing the computational and analytical challenges typically associated with high-dimensional settings. To corroborate our theoretical contributions, we conduct extensive numerical simulations on synthetic datasets and provide empirical validations using real-world data, highlighting the estimator’s practical relevance and effectiveness in capturing intricate temporal dependencies and underscoring its relevance for analyzing complex nonstationary phenomena.

                   \end{abstract}


\section{Introduction} \label{sec:introduction}

Time series analysis (TSA) aims to study the historical and current behavior of certain variables to predict future patterns. Such analysis is pivotal to forecast and control potential future scenarios. For instance, in predicting economic conditions, one would analyze historical behaviors of key indicators like Gross Domestic Product (GDP), inflation rates, stock prices, unemployment rates, among many others \citep{Wengetal2018, Guerardetal2020, Dadashovaetal2021, Jingetal2023}. Similarly, a health expert observing a correlation between the rise in the number of pulmonary diseases and air quality might delve into time series data on air pollutants (PM2.5, PM10, CO), ground-level ozone (O3), and meteorological factors such as temperature and humidity \citep{Jiangetal2020, Kolluruetal2021}.
   
While classical TSA operates under the assumption of stationarity, it is important to note that many time series, including those mentioned above, display nonstationarity \citep{Bugnietal2009, Aueetal2015, Chenetal2016, AuevanDelft2020, Amatoetal2020, Miyamaetal2020}. One approach to model this nonstationarity is through LSPs~\cite{Dahlhaus1996}, where these processes are locally approximated by strictly stationary processes in a finer-grid time interval \citep{Dahlhaus1996, DahlhausSubbaRao2006, DAHLHAUS2012351}. 
Most of the statistical theoretical guarantees on LSPs in the literature are proposed for both the conditional mean and the variance functions. 
In the parametric framework, \cite{Dahlhaus1996} obtained estimates by minimizing the generalized Whittle function using local periodograms. Nonparametric approaches rely on NW \cite{MR166874,MR185765} estimation procedure, which is a widely used local averaging method for estimating the conditional mean function~\cite{Kristensen2009, vogt2012, ZangWu2015, Truquet2019, Kurisuetal2022, Daisuke2022}.

To deviate from conditional mean function estimation, various works dealing with conditional distribution estimation have already been proposed. In \cite{Halletal1999}, the authors considered strictly stationary processes and proposed two estimation methods: a local logistic distribution method and an adjusted NW estimation procedure. Both methods produced distribution function estimators that lie between $0$ and $1$. Using a simulation study, they observed that the adjusted NW estimator is superior to locally fitting a logistic model since the latter produced arbitrarily high-order distribution estimators.  In \cite{MR4783436}, a local polynomial estimator for the conditional cumulative distribution function (CDF) of a scalar $Y_t$ given a functional $X_t$ was proposed. In their work, $\{X_t, Y_t\}_{1\leq t\leq T},$ is assumed to be a stationary strongly mixing process. They applied local polynomial smoother to reduce the large bias at the boundary region of kernel estimation and derived confidence intervals based on the asymptotic normality of the local linear estimator. 
Additionally, \cite{Ahmedetal2020} introduced an adaptive NW estimator for strictly stationary processes using varying bandwidth and proved the asymptotic normality of the proposed estimator and, through a simulation study, they have shown that the adaptive NW estimator performed better than the weighted NW estimator with fixed bandwidth. 
In the framework of distributional regression, \cite{Dombry2023} extended Stone's theorem using Wasserstein distance and showed that the conditional CDF estimator with local probability weights is a universally consistent estimator of the true conditional CDF. 

When we are interested in conditional distribution estimation, we have to carefully choose a metric measuring the distance between probability distributions. In this work, we consider an optimal transport (OT) metric that has been recognized as an effective tool in comparing probability distributions. OT solves problems centered around the shortest path principle \citep{PeyreCuturi2020}. One of the prominent metrics in OT is Wasserstein distance~\cite{Villani2009OToldnew}. Due to the topological structure induced by Wasserstein distance, it is used as a tool in asymptotic theory and a goodness-of-fit test in statistical inference \cite{PanaretosZemel2019}. 
It has gained many applications compared to Total Variation, Hellinger, and Kullback-Leibler divergence since it can be optimally estimated from samples under mild assumptions \cite{Manoleetal2022}.

\paragraph{Contributions.} 
The contributions of the present paper are three-fold: we consider estimating the conditional probability distribution of LSPs rather than the conditional mean or variance functions, as it was largely proposed in the literature. Under mixing conditions \citep{Doukhan1994, Rio2017, AhsenVidyasagar2014}, we provide the convergence rate of NW conditional distribution estimator with respect to Wasserstein distance for a scalar target $Y_{t, T}$ and a $d$-dimensional locally stationary covariates $\boldsymbol{X}_{t, T}$. We next extend the results to the multivariate setting, i.e., $Y_{t, T} \in \R^q  (q\geq 1)$, where we give the convergence rate of NW conditional distribution estimator through sliced Wasserstein distance.  
To the best of our knowledge, this is the first work that establishes OT bounds for conditional probability distribution in LSPs.
We then illustrate our theoretical findings through numerical experiments on synthetic and real-world datasets.

\paragraph{Layout of the paper.}
The structure of this paper is as follows. In Section \ref{sec:preliminaries}, we present the regression estimation problem, a brief background of local stationarity, and Wasserstein distance. We derive the main results in Section \ref{sec:theoretical_guarantees}:  we first define the NW kernel estimator, and then provide the rates of convergence of the first and second moments of Wasserstein distance between estimated and true conditional distribution. We extend our result to the multivariate case in Section \ref{sec: multivariate case}. Section \ref{sec:numerical_experiments} shows the results of numerical experiments. All the proofs are postponed to the appendices.

\paragraph{Notation.}
Throughout the paper, we consistently use the following notations. 
We denote by  $\delta_y$ the Dirac mass at point $y.$  For any real random variable $X$, we denote $\norm{X}_{L_q}$ as the $L_q$-norm of $X$, for $q\geq1$, i.e., $\norm{X}_{L_q}=(\E[|X|^q])^{\frac{1}{q}}$.
We say $a_T \lesssim b_T$ if there exists a constant $C$ independent of $T$ such that $a_T \leq Cb_T$. We write $a_T \sim b_T$ if $a_T \lesssim b_T$ and $b_T \lesssim a_T$. For any positive $a_T$ and $b_T$, we write $a_T = \bigO(b_T)$ if $\lim_{T\rightarrow \infty}{\frac{a_T}{b_T}}\leq C$ for some $C>0$.
To indicate that $a_T$ is bounded, we write $a_T=\bigO(1)$. On the other hand, we write $a_T = o(b_T)$ if $\lim_{T\rightarrow \infty}{\frac{a_T}{b_T}}=0$. If $a_T \rightarrow 0$, we write $a_T=o(1)$. For a given $a_T$ and a sequence of random variables $X_T$, we write $X_T = \bigO_\P(a_T)$ if for any $\epsilon>0$, there exists $C_\epsilon>0$ and $T_\epsilon \in \mathbb{N}$ such that, for all $T\geq T_\epsilon$, $\P\big[\frac{|X_T|}{a_T}>C_\epsilon\big]<\epsilon$. We write $X_T=o_\P(a_T)$ if $\lim_{T\rightarrow \infty}\P\big[\frac{|X_T|}{a_T}>\epsilon\big]=0$, for any $\epsilon>0$. If $X_T \xrightarrow{\P} 0$, we write $X_T=o_\P(1).$ We write $a\vee b = \max\{a,b\}$ and $a\wedge b = \min\{a,b\}$, for any $a,b \in \R$.

% section introduction (end)

\section{Preliminaries} \label{sec:preliminaries}

We start introducing a background of LSPs and optimal transport through Wasserstein distance. We then present the mixing coefficient employed to assess weak dependency.

\subsection{Locally stationary process} 
\label{ssub:background_of_locally_stationary_processes_and_wasserstein_distance}

Let $T \in \mathbb{N}$ and suppose that we have access to $T$ random variables $\{Y_{t,T}, \boldsymbol{X}_{t,T}\}_{t=1, \ldots, T}$, where $Y_{t,T}$ is real-valued and $\boldsymbol{X}_{t,T} =(X_{t,T}^1, \ldots, X_{t,T}^d)^\top \in\R^d$.
We consider the following regression estimation problem
\begin{align}
\label{eq:major_estimation_problem}
Y_{t,T} = m^\star\big(\frac tT, \boldsymbol{X}_{t,T}\big) + \varepsilon_{t,T}, \text{ for all }t=1, \ldots, T,
\end{align}
where  $\{\varepsilon_{t,T}\}_{t\in \mathbb{Z}}$ is a sequence of independent and identically distributed (i.i.d.) random variables independent of $\{\boldsymbol{X}_{t,T}\}_{t=1, \ldots, T}$, that is $\E[\varepsilon_{t,T}|\boldsymbol{X}_{t,T}] = 0.$ We assume that the covariate $\boldsymbol{X}_{t, T}$ is locally stationary and $Y_{t, T}$ is integrable. Note that $m^\star\big(\frac tT, \boldsymbol{X}_{t,T}\big) = \E[Y_{t,T}|\boldsymbol{X}_{t,T}]$ is the \textit{oracle} conditional mean function in model~(\eqref{eq:major_estimation_problem}), which does not depend on real-time $t$ but rather on the rescaled time $u=\frac tT.$ 
These $u$-points form a dense subset of the unit interval $[0,1]$ as the sample size $T$ goes to infinity. Hence, $m^\star$ is identified almost surely (a.s.) at all rescaled $u$-points if it is continuous in the time direction. In LSPs, this rescaled time refers to the transformation of the original time scale. A wide range of interesting nonlinear process models fit into the general
framework (\ref{eq:major_estimation_problem}). An important example is the nonparametric time-varying autoregressive (tvAR) model:
\begin{align*}
Y_{t,T} = m^\star\big(\frac tT, Y_{t-1,T}, \ldots, Y_{t-d,T}\big) + \varepsilon_{t,T},
\end{align*}
where $\boldsymbol{X}_{t,T} = (Y_{t-1,T},\ldots, Y_{t-d,T})^\top$ is the $d$-lag of $Y_{t,T}$; for instance, see \citep{vogt2012, DAHLHAUS2012351, Dahlhausetal2019, RichterDahlhaus2019}. Let us now formally define the notion of LSP. We adopt the definition given in \cite{vogt2012}.
\begin{definition} \label{definition:locallystatpr}
A process $\{\boldsymbol{X}_{t,T}\}_{t=1,\ldots,T}$ is locally stationary if for each rescaled time point $u\in[0,1]$, there exists an associated strictly stationary process $\{\boldsymbol{X}_t(u)\}_{t=1,\ldots,T}$  verifying
\begin{align*}
\norm{\boldsymbol{X}_{t,T} - {\boldsymbol{X}}_t(u)} \leq \big(\big|\frac tT - u\big| + \frac 1T\big) U_{t,T}(u) \quad \text{a.s.,}
\end{align*}
where $\{U_{t,T}(u)\}_{t=1,\ldots,T}$ is a positive process such that $\E\big[(U_{t,T}(u))^\rho\big]< C_U$ for
some $\rho > 0$ and $C_U < \infty$ independent of $u, t,$ and $T$. The norm 
$\norm{\cdot}$ denotes an arbitrary norm on $\R^d$.
\end{definition}

Definition \ref{definition:locallystatpr} states that around each rescaled time $u$, any $d$-dimensional LSP $\{\boldsymbol{X}_{t,T}\}_{t=1,\ldots,T}$ can be approximated by $\{\boldsymbol{X}_t(u)\}_{t=1,\ldots,T}$, which is a strictly stationary process at each fixed $u$.  This approximation results in a negligible difference between $\boldsymbol{X}_{t, T}$ and $\boldsymbol{X}_t(u)$. Due to this negligibility, we can presume that a nonstationary process is stationary at local time. According to \cite{vogt2012}, $U_{t,T}(u) = \bigO_{\mathds{P}}(1)$ since the $\rho$-th moments of $U_{t,T}(u)$ are uniformly bounded. This gives
\begin{align*}
    \norm{\boldsymbol{X}_{t,T}-\boldsymbol{X}_t(u)}  = \bigO_{\mathds{P}}\big(\big|\frac{t}{T}-u\big|+\frac{1}{T}\big).
\end{align*}
For $u=\frac{t}{T}$, we have $\norm{\boldsymbol{X}_{t,T} - \boldsymbol{X}_t\big(\frac{t}{T}\big)}\leq \frac{C_U}{T}$.  Note that the exponent $\rho$ can be considered as an indicator of how well this approximation is being done. Choosing larger $\rho$ gives a better approximation of $\boldsymbol{X}_{t, T}$ by $\boldsymbol{X}_t(u)$ and gives moderate bounds for their absolute difference.

\subsection{Optimal transport: Wasserstein distance}

Let $\mathcal{P}_r(\R)$ be the set of Borel probability measures in $\R$ having finite $r$-th moment $(r\geq 1)$, i.e., $\mathcal{P}_r(\R)=\{\mu\in\mathcal{P}(\R): \int_\R |x|^r \mu(\mathrm{d}x) < \infty \}$. 
We quantify the distance between probability measures $\mu, \nu \in \mathcal{P}_r(\R)$ through the $r$th-Wasserstein distance, denoted by $W_r(\mu,\nu)$ and defined as
\begin{align}\label{def:W1_general}
    W_r(\mu,\nu) &= \Big(\inf_{\pi\in \Pi(\mu,\nu)} \iint_{\R \times \R} {|u- v|}^r\pi(\diff u, \diff v)\Big)^{1/r},
\end{align}
where $\Pi(\mu,\nu)$ stands the set of probability measures on $\R \times \R$ with marginals $\mu$ and $\nu$. Since $\R$ is a complete and separable metric space where the infimum is indeed a minimum, optimal couplings always exist \citep{Villani2009OToldnew}. Equation (\ref{def:W1_general}) states that $W_r(\mu, \nu)$ is the infimum of the expectation of distance between two random variables over all possible couplings, i.e., $W_r(\mu,\nu) = \big( \inf_{U\sim \mu, \text{ } V\sim \nu}  \E [|U-V|^r]\big)^{1/r}$, where $\mu$ and $\nu$ are the laws of $U$ and $V$, respectively. Note that $W_r$ metrizes the space $\mathcal{P}_r(\R)$, for details see \citep{Villani2009OToldnew, HallinSegers2021, Manoleetal2022},  and often defined in higher dimensional setting that makes it difficult to compute \citep{BayraktarGuo2021, Dombry2023}. 

A simple optimal coupling can be represented by a probability inverse transform: given $\mu, \nu \in \mathcal{P}_r(\R)$, let $F_\mu(\cdot)$ and $F_\nu(\cdot)$ be the cumulative distribution functions (CDF) and $F_\mu^{-1}(\cdot)$ and $F_\nu^{-1}(\cdot)$ be the respective generalized inverse or quantile functions defined as $F_\mu^{-1}(z):= \inf\{v\in\R: \mu ((-\infty,v])\geq z\}$ for all $z\in[0,1]$ (similarly for $F_\nu^{-1}(z)$). Then, for a uniformly distributed random variable $Z$ on $(0,1)$, we can construct an optimal coupling $(U,V)=(F_\mu^{-1}(Z),F_\nu^{-1}(Z))$, see \citep{DedeckerMerleved2017, Dombry2023}. Hence, in univariate setting, the minimization problem (\ref{def:W1_general}) boils down to
\begin{align*}
    W_r(\mu,\nu) &= \Big(\int_0^1 {\left|F_\mu^{-1}(z) - F_\nu^{-1}(z)\right|}^r \diff z \Big)^{1/r}.
\end{align*}
For $r=1$ and using a change of variable, the $1$-Wasserstein distance writes as
\begin{align}\label{def:W1_cdf}
    W_1(\mu,\nu) &= \int_{\R} |F_\mu(v) - F_\nu(v)| \diff v.
\end{align}
Clearly, $W_1(\mu,\nu)$ is the $L_1$-distance between the CDF $F_\mu(\cdot)$ and $F_\nu(\cdot)$. 

Now, since we are dealing with sequences exhibiting weak dependency, let us define the mixing coefficient being considered in this paper.

\subsection{Mixing condition}

The convergence rates of LSPs estimation are given under weakly dependent conditions, often termed mixing conditions. 
These latter are used to measure the dependency degree between observation sets of a stochastic process when they get far apart in time.
In a nutshell, the farthest time distance between observations, the lower dependency. 
Mixing conditions are originally defined to prove the law of large numbers for non-i.i.d. processes~\cite  {Doukhan1994, Rio2017, AhsenVidyasagar2014}. 
Choosing the right mixing condition is essential for efficient modeling and inference \citep{Peligrad2002, DedeckerPrieur2005, Rio2017}. 
One of the prominent mixing conditions is $\beta$-mixing, it has been utilized to prove central limit theorems and moment inequalities \citep{Dedeckeretal2007, Bosq2012, Poinas2020}.

\begin{definition} \label{def: mixing}
    Let $(\Omega, \mathcal{A}, \mathds{P})$ be a probability space, $\mathcal{B}$ and $\mathcal{C}$ be subfields of $\mathcal{A}$, and set      $\beta(\mathcal{B},\mathcal{C}) = \E [\sup_{C\in\mathcal{C}}|\P(C) - \P(C|\mathcal{B})|]$.  
    For any array $\{Z_{t,T}:1\leq t \leq T\}$, define the coefficient
                    \begin{align*}  
        \beta(k) &= \sup_{1\leq t\leq T-k} \beta\big( \sigma(Z_{s,T}, 1\leq s\leq t), \sigma(Z_{s,T}, t+k\leq s \leq T)\big ),
    \end{align*}
    where $\sigma(Z)$ denotes the $\sigma$-algebra generated by $Z$. The array $\{Z_{t,T}\}$ is said to be $\beta$-mixing or absolutely regular mixing if $\beta(k)\rightarrow 0$ as $k \rightarrow \infty.$
\end{definition}

If a process is weakly dependent, particularly $\beta$-mixing or regular mixing, this definition entails asymptotic independence as $k \rightarrow \infty$. As argued in \cite{Vidyasagar1997}, $\beta$-mixing is a ``just right" assumption in analyzing weakly dependent sequences. In fact, regular mixing implies strong mixing or $\alpha$-mixing, making it a stronger form of weak dependency condition \cite{Mokkadem1998, DedeckerPrieur2005, Rio2017}. Various types of $\beta$-mixing include exponentially $\beta$-mixing where $\beta(k) = \bigO \big(e^{-\gamma k}\big)$ for $\gamma>0$ \citep{Masuda2007, Lee2012}. It can also be arithmetically $\beta$-mixing, i.e., $\beta(k) = \bigO \big(k^{-\gamma}\big)$ \citep{Ferraty&Vieu2006, vogt2012, SoukariehBouzebda2023,MR4737023}. Regular mixing is highly desirable in practice, as many commonly used time series models exhibit this property \cite{McDonaldetal2011}. Examples include autoregressive moving average (ARMA) models \cite{Mokkadem1998}, generalized autoregressive conditional heteroscedastic (GARCH) models \cite{CarrascoChen2002}, and some Markov processes \cite{Doukhan1994}.

% Note that $\beta$-mixing implies $\alpha$-mixing (see Proposition 2.1 in \cite{DedeckerPrieur2005}).

\section{Wasserstein bounds for NW estimation procedure} \label{sec:theoretical_guarantees}

For a fixed $t\in \{1, \ldots, T\}$ and $\boldsymbol{x}\in\R^d$, we denote the conditional probability distribution of $Y_{t,T}|\boldsymbol{X}_{t,T}=\boldsymbol{x}$ by $\pi_t^\star(\cdot|\boldsymbol{x})$ and its conditional CDF by $F_t^\star(\cdot|\boldsymbol{x})$. The mean conditional regression function  is then given by
\begin{equation*}
m^\star(\frac{t}{T}, \boldsymbol{x}) = \E_{\pi_t^\star(\cdot|\boldsymbol{x})}[Y_{t,T}|\boldsymbol{X}_{t,T}=\boldsymbol{x}] = \int_{-\infty}^\infty y \,\diff\pi_t^\star(y|\boldsymbol{x}).
\end{equation*}
Let $K_1, K_2$ be two $1$-dimensional based kernel functions and $h$ be a $T$-dependent bandwidth, i.e., $h=h(T)$ satisfying $h(T)\rightarrow 0$ as $T\rightarrow\infty$. 
Setting the scaled kernels $K_{h,i}(\cdot) = K_i(\frac{\cdot}{h}),$ for $ i=1,2,$ we define: 

\begin{definition}
\label{definition: pi_hat}
The NW estimator of $\pi_t^\star(\cdot|\boldsymbol{x})$ reads as 
\begin{equation*}
    \hat{\pi}_t(\cdot|\boldsymbol{x}) = \sum_{a=1}^T\omega_{a}(\frac t T,\boldsymbol{x})\delta_{Y_{a,T}},
\end{equation*}
 where 
\begin{align}\label{def: weights}
  \omega_{a}(\frac t T, \boldsymbol{x})=\frac{\displaystyle  K_{h,1}(\frac{t}{T} - \frac{a}{T})\prod_{j=1}^dK_{h,2}(x^j- X_{a,T}^j)}{\displaystyle \sum_{a=1}^TK_{h,1}(\frac{t}{T} - \frac{a}{T})\prod_{j=1}^dK_{h,2}(x^j - X_{a,T}^j)}.
 \end{align}
The associated conditional CDF to $\hat{\pi}_t(\cdot|\boldsymbol{x})$ is defined as, for all $y \in \R,$  
\begin{equation*}\label{def:CDF of pi-hat}
  \hat{F}_{t}(y|\boldsymbol{x})=\sum_{a=1}^T\omega_{a}(\frac tT,\boldsymbol{x})\mathds{1}_{Y_{a,T}\leq y}. \end{equation*}
\end{definition}

Hereafter, we assume that the weights $\{\omega_a(u,\boldsymbol{x})\}_{a= 1, \ldots, T}$ are measurable functions of $\boldsymbol{x}$, $\boldsymbol{X}_{a,T}$, and $u$ but do not depend on $Y_{a,T}$. Note that NW estimator of $m^\star(u,\boldsymbol{x})$ is given by  
\begin{equation}\label{eqn: m_hat univariate}
    \hat m(u,\boldsymbol{x}) 
        = \sum_{a=1}^T   \omega_{a}(u, \boldsymbol{x}) Y_{a, T}
\end{equation}
and involves two kernel functions: one is in the direction of the $d$-dimensional $\boldsymbol{X}_{t, T}$ and the other is with respect to the rescaled time $u=\frac{t}{T}$. This means that we do not only smooth in the space-direction of the covariates $\boldsymbol{X}_{t, T}$ but also in the time-direction \citep{vogt2012}, allowing us to properly assign weights $\omega_a(\frac{t}{T},\boldsymbol{x})$ and then consider local behavior of the data in the rescaled time $\frac{t}{T}$. The scaled kernel $K_{h, i}(\cdot)$ uses single bandwidth $h$ and can differ for time and space directions. This implies that, in both directions, weights placed on each data point are scaled equally to avoid over-fitting \citep{Silverman1998}.

Next, we present the assumptions about the underlying process in model~(\eqref{eq:major_estimation_problem}) and NW estimator given in Definition~\ref{definition: pi_hat}.

\subsection{Assumptions} 

Our main results are based on the following assumptions that are classical in LSPs~\cite{FanMasry1992, Masry2005, Hansen2008, Kristensen2009, vogt2012, MR4737023} and conditional density function estimation~\cite{Owen1986, Halletal1999, Veraverbekeetal2014, OtneimTjøstheim2016, Ahmedetal2020}.

\begin{assumption}[Local stationarity]
\label{Assumption: X is lsp}
Assume that $\{\boldsymbol{X}_{t,T}\}_{t=1, \ldots, T}$ has compact support $\mathcal{X}$ and is a locally stationary process approximated by $\{\boldsymbol{X}_t(u)\}$ for each time point $u\in [0,1]$. 
The density $f(u, \boldsymbol{x})$ of $\boldsymbol{X}_t(u)$ has continuous partial derivative, $\partial_j f(u, \boldsymbol{x}) := \frac{\partial}{\partial x^j}f(u, \boldsymbol{x})$, with respect to $\boldsymbol{x}$ for each $u\in[0,1]$.
\end{assumption}

Assumption~\ref{Assumption: X is lsp} establishes the smoothness of the density $f(u, \boldsymbol{x})$ wrt $\boldsymbol{x}$, allowing to use its Taylor expansion in the proofs of main results.

\begin{assumption}[Kernel functions] \label{Assumption: kernel functions}
The based kernel $K_i(\cdot)$, $i=1,2$, is symmetric about zero, bounded, and has compact support, that is, $K_i(z)=0$ for all $|z|>C_i$ for some $C_i<\infty$.
Additionally, it fulfills a Lipschitz condition with  a positive constant  $L_i<\infty$, such that 
$|K_i(z)-K_i(z')|\leq L_i|z-z'|$, for  all $z,z'\in \R$, and 
         \begin{equation}\label{eqn: some properties of K}
        \int K_{i}(z)\diff z = 1,  
        \int zK_{i}(z)\diff z = 0, \text{ and }
         \int z^2K_{i}(z)\diff z = \kappa<\infty.
    \end{equation}
\end{assumption}

Assumption~\ref{Assumption: kernel functions} signifies that the kernel function has a bounded rate of change. By assuming that $K_{i}$ is symmetric about zero, we allow either or both kernel functions to be box, triangle, quadratic, or Gaussian kernels. From (\ref{eqn: some properties of K}), we further assume that the based kernels can be interpreted as probability density functions. The second integral shows that each kernel does not introduce first-order linear bias when applied to the data. The last conveys bounded second-moment regularity, leading each kernel to have finite variance and limiting influence of outliers. 

\begin{assumption}[Regularity condition on the bandwidth]
\label{Assumption: bandwidth}
    The bandwidth $h$ satisfies
    \begin{align}\label{eqn: frac involving h is o1} 
                        \frac{1}{T^{\nu \wedge \frac{1}{2}} h^{d+1}} = o(1),
    \end{align}
    $\nu=\rho\wedge 1$, for $\rho>0$ as introduced in Definition~\ref{definition:locallystatpr}. 
\end{assumption}

Assumption \ref{Assumption: bandwidth} indicates that $h$ converges slower to zero, for instance at a polynomial rate, i.e., $h= \bigO(T^{-\xi})$, for small $\xi>0$. It is worth noting that the choice of bandwidth is crucial for the bias-variance trade-off \citep{Silverman1998}: small $h$ leads to over-fitting, producing an estimator with high variance and low bias, while large $h$ may cause under-fitting. The given condition gives balance for both variance and bias to have appropriate asymptotic properties. It may use a vector of smoothing parameters or varying bandwidths in certain situations, however, in our setting, we opt to use a single bandwidth. Condition (\ref{eqn: frac involving h is o1}) is a strengthening of the usual condition $Th^{d+1}\rightarrow \infty$, needed to guarantee convergence to zero of our resulting bounds. 
\begin{assumption}[Conditional CDF] \label{assumption: CDF}
    The conditional CDF $F_{\cdot}^\star(\cdot|\cdot)$ is Lipschitzian, i.e., $\big| F_a^\star(\cdot|\boldsymbol{x}) - F_t^\star (\cdot| \boldsymbol{x'}) \big| \leq L_{F^\star} \big( \norm{\boldsymbol{x} - \boldsymbol{x'}} + \big|\frac{a}{T} - \frac{t}{T}\big|\big)$, for some constant $L_{F^\star}<\infty$, and for all $a, t \in \{1, \ldots, T\},$ $\boldsymbol{x}, \boldsymbol{x'}\in \R^d$. 
\end{assumption}

Assumption~\ref{assumption: CDF} entails $F_{\cdot}^\star(\cdot|\cdot)$ to behave in a smooth manner, and it does not change rapidly as the observation changes. This differs from the assumption used in ~\cite{Halletal1999, Veraverbekeetal2014, OtneimTjøstheim2016, Ahmedetal2020} where the conditional CDF is assumed to be twice differentiable.

\begin{assumption}[Mixing condition] \label{Assumption: mixing}
    The process $\{(\boldsymbol{X}_{t,T}, \varepsilon_{t,T})\}_{t=1, \ldots, T}$ is arithmetically $\beta$-mixing, that is, $\beta(k) \leq Ak^{-\gamma}$ for some $A>0$ and $\gamma>2$. We further assume that for some $p>2$ and $\zeta >1-\frac{2}{p}$,
    \begin{align}\label{eqn: infinite sum of betas is finite}
        \sum_{k=1}^{\infty} k^\zeta\beta(k)^{1-\frac{2}{p}} <\infty. 
    \end{align}
\end{assumption}

\begin{assumption}[Blocking condition]
\label{assumption: blocking}
    There exists a sequence of positive integers $\{q_T\}$ satisfying $q_T \rightarrow \infty$ and $q_T = o\big(\sqrt{Th^{d+1}}\big)$, as $T \rightarrow \infty$.
            \end{assumption}

Assumptions \ref{Assumption: mixing} and \ref{assumption: blocking} are useful for dependent sequence estimation procedures. The $\beta$-mixing is a stronger form of independence between distant observations in a process \citep{Bradley2005, Rio2017, Poinas2020}. Condition (\ref{eqn: infinite sum of betas is finite}) highlights the decay of $\beta$-mixing coefficient $\beta(k)$. In the proof of Theorem \ref{Theorem: convergence of EW1}, Bernstein's blocking technique was used to create independent blocks \citep{Bernstein1927}. We define the size of big blocks to be proportional to $q_T$ in Assumption \ref{assumption: blocking}. 
$ $ 
\subsection{Convergence rate in Wasserstein distance}

We investigate the error between NW estimator $\hat{\pi}_t(\cdot|\boldsymbol{x})$ and true conditional distribution $\pi^\star_t(\cdot|\boldsymbol{x})$ by establishing the rate of convergence wrt Wasserstein distance.  

\begin{theorem}\label{Theorem: convergence of EW1}
Let Assumptions \ref{Assumption: X is lsp} - \ref{assumption: blocking} hold and define  $I_h = [C_1h, 1 - C_1h]$. Then,
\begin{align*}
    \sup_{\boldsymbol{x}\in \mathcal{X}, \frac{t}{T}\in I_h} \E\big[W_1\big(\hat{\pi}_t(\cdot|\boldsymbol{x}), \pi_t^\star(\cdot|\boldsymbol{x})\big)\big]
    &= \bigO\Big(\frac{1}{T^{\frac{1}{2}} h^{d + 1 - \frac{1}{p}(1-\nu)}} + \frac{1}{T^\nu h^{d + \nu - 1}} + h\Big).
\end{align*}
\end{theorem}

Theorem \ref{Theorem: convergence of EW1} ensures that the expectation of Wasserstein distance between the underlying conditional probability distributions converges to zero with nonstandard components of orders $\bigO\Big(\frac{1}{T^{\frac{1}{2}} h^{d + 1 - \frac{1}{p}(1-\nu)}}\Big)$ and $\bigO\big(\frac{1}{T^\nu h^{d + \nu - 1}}\big)$, and a standard component of order $\bigO(h)$. Generally, this convergence is affected by the bandwidth $h$; as discussed in Assumption \ref{Assumption: bandwidth}, it should slowly approach zero for this result to hold. The first and second components, which depend on $\nu$ and $p$, are results of approximating $\{\boldsymbol{X}_{t,T}\}_{t=1,\ldots,T}$ by a locally stationary $\{\boldsymbol{X}_t(\frac tT)\}_{t=1,\ldots,T}$ and by assuming that $\{\boldsymbol{X}_{t,T}\}_{t=1,\ldots,T}$ is $\beta$-mixing. Recall that $\nu$ measures how well $\{\boldsymbol{X}_t(\frac tT)\}_{t=1,\ldots,T}$ is locally approximating $\{\boldsymbol{X}_{t,T}\}_{t=1,\ldots,T}$, a larger $\nu$ makes faster convergence to zero. These rates are also affected by the dimension of the covariate. While the last component is obtained by assuming Lipschitz continuity on the conditional CDF $F_\cdot^\star(\cdot|\cdot)$. If $\nu = 1$, this convergence becomes $\bigO \big(\frac{1}{T^\frac{1}{2} h^{d+1}} + h\big)$. 

\textit{Sketch of proof.} The proof of Theorem \ref{Theorem: convergence of EW1} is postponed to Appendix \ref{appendix: proof of convergence of EW1}, where we use the definition of $W_1$ as the expected $L_1$ error between the conditional CDFs $\hat{F}_t(y|\boldsymbol{x})$ and $F_t^\star(y|\boldsymbol{x})$ for any $y\in \R$, given in (\ref{def:W1_cdf}), and Fubini's theorem to deal with the expectation. By applying Cauchy-Schwarz inequality, the expectation of the absolute difference of $\hat{F}_t(y|\boldsymbol{x})$ and $F^\star_t(y|\boldsymbol{x})$ is broken down into two parts: one involving the density estimator and the other involving the square of sums of the underlying terms. The latter term can be handled by employing Bernstein's blocking procedure: we decompose it as a sum of independent blocks: big blocks, small blocks, and a remainder block. 
For a strictly stationary stochastic process $\{Y_t, X_t\}$, where $Y_t$ and $X_t$ are scalar, \cite{Halletal1999} (Theorem 1.\textit{ii}) had shown the pointwise convergence of their proposed adjusted NW conditional distribution function estimator to be $\bigO\big(\frac{1}{\sqrt{Th}} + h^2\big)$.

\begin{corollary}\label{Remark: bound EW_s-s}
    Let Assumptions \ref{Assumption: X is lsp} - \ref{assumption: blocking} hold and assume that $Y_{t, T}$ is uniformly bounded by $M>0$. Then, for $r\geq 1$,
    \begin{align*}        \sup_{\boldsymbol{x}\in \mathcal{X}, \frac{t}{T}\in I_h} \E[W_r^r(\hat{\pi}_t(\cdot|\boldsymbol{x}),\pi_t^\star(\cdot|\boldsymbol{x}))]
        &= \bigO\Big(\frac{1}{T^{\frac{1}{2}} h^{d + 1 - \frac{1}{p}(1-\nu)}} + \frac{1}{T^\nu h^{d + \nu - 1}} + h\Big).
    \end{align*}
\end{corollary}

Proof of Corollary \ref{Remark: bound EW_s-s} is detailed in Appendix \ref{appendix: proof of expectation of W_r}. Let us examine the convergence rate of the second moment of the 1-Wasserstein distance between the considered NW estimator and true conditional distribution.

\begin{corollary}\label{corollary: convergence of the 2nd moment}
Let Assumptions \ref{Assumption: X is lsp} - \ref{assumption: blocking} hold. Then 
\begin{align*}
    \sup_{\boldsymbol{x}\in \mathcal{X}, \frac{t}{T}\in I_h} \norm{W_1\big(\hat{\pi}_t(\cdot|\boldsymbol{x}), \pi_t^\star(\cdot|\boldsymbol{x})\big)}_{L_2}
    = \bigO\Big(\frac{1}{T^{\frac{1}{2}} h^{d + 1 - \frac{1}{p}(1-\nu)}} + \frac{1}{T^\nu h^{d + \nu - 1}} + h \Big).
\end{align*}
\end{corollary}
The proof of Corollary \ref{corollary: convergence of the 2nd moment} is in Appendix \ref{appendix: proof of convergence of the 2nd moment} and is based on Minkowski's integral inequality.
\newline
\newline
\indent The NW conditional mean estimator $\hat{m}$ of $m^\star$, given in (\ref{eqn: m_hat univariate}), verifies 

\begin{proposition}\label{prop: |mhat-m| leq W1}
    Let $\hat{m}(\frac{t}{T},\boldsymbol{x}) = \sum_{a=1}^T   \omega_{a}(\frac t T, \boldsymbol{x}) Y_{a,T}$, then
    \begin{align*}
        \sup_{\boldsymbol{x}\in \mathcal{X}, \frac{t}{T}\in I_h} \E\big[ |\hat{m}(\frac{t}{T},\boldsymbol{x}) - m^\star(\frac{t}{T}, \boldsymbol{x})|\big] 
                        &= \bigO\Big(\frac{1}{T^{\frac{1}{2}} h^{d + 1 - \frac{1}{p}(1-\nu)}} + \frac{1}{T^\nu h^{d + \nu - 1}} + h\Big).
    \end{align*}
\end{proposition}

Proposition \ref{prop: |mhat-m| leq W1} signifies that convergence rate of NW regression function estimator $\hat{m}(u, \boldsymbol{x})$ can also be obtained through Wasserstein distance. This latter is comparable with the rate in \cite{vogt2012} (Theorem 4.2), of order $\bigO_\P\big(\sqrt{\frac{\log T}{Th^{d+1}}} + \frac{1}{T^\nu h^d} + h^2\big)$. Refer to Appendix \ref{appendix: proof of mhat leq W1} for the details of the proof.

If we assume that $F_\cdot^\star(\cdot|\cdot)$ is twice differentiable, then we get a similar convergence rate for the bias component. The bound of $W_1$ is slower than that of $\hat{m}(u,\boldsymbol{x})$ given in \cite{vogt2012} since we are measuring the disparity between underlying distributions, taking into account all aspects of distributional differences, not just discrepancies between conditional means.

\begin{proposition}\label{prop: convergence of EW1 chosen h}
Assume Assumptions \ref{Assumption: X is lsp} - \ref{assumption: blocking} hold and let $h = \bigO(T^{-\xi})$, where $0 < \xi < \frac{\frac{1}{2} \wedge \nu}{d+1}$.
Then,
\begin{align*}
   \sup_{\boldsymbol{x}\in \mathcal{X}, \frac{t}{T}\in I_h} \E\big[W_1\big(\hat{\pi}_t(\cdot|\boldsymbol{x}), \pi_t^\star(\cdot|\boldsymbol{x})\big)\big]
    &= \bigO \Big(  \frac{1}{T^{\frac{1}{2} -\xi(d + 1 - \frac{1}{p}(1-\nu))}} + \frac{1}{T^{\nu -\xi(d + \nu - 1)}} + \frac{1}{T^\xi} \Big).
\end{align*}
\end{proposition}
\newline
Proof of Proposition \ref{prop: convergence of EW1 chosen h} follows the same line of Theorem \ref{Theorem: convergence of EW1}'s proof, by setting $h = \bigO(T^{-\xi})$.

% section theoretical_guarantees (end)
\section{Extension to multivariate case}
\label{sec: multivariate case}

We suppose access to $T$ samples $(\boldsymbol{Y}_{t,T}, \boldsymbol{X}_{t,T}) \in \R^q\times\R^d$, where $\boldsymbol{Y}_{t,T} = (Y_{t,T}^1, \ldots, Y_{t,T}^q)^\top \in\R^q$ and $\boldsymbol{X}_{t,T}  \in\R^d$. We consider the multivariate regression model:
\begin{align*}
  \boldsymbol{Y}_{t,T} = \boldsymbol{m}^{\star}\big(\frac tT, \boldsymbol{X}_{t,T}\big) + \boldsymbol{\varepsilon}_{t,T},    
\end{align*}
where $\boldsymbol{m}^{\star}\big(\frac tT, \boldsymbol{X}_{t,T}\big) = \big(m^{\star 1}\big(\frac tT, \boldsymbol{X}_{t,T}\big), \ldots, m^{\star q}\big(\frac tT, \boldsymbol{X}_{t,T}\big)\big)^\top$ and $\boldsymbol{\varepsilon}_{t,T} = \big(\varepsilon_{t,T}^{1} , \ldots, \varepsilon_{t,T}^{q}\big)^\top$,
for all $t=1, \ldots, T$. The variables  $\{\varepsilon_{t,T}^l\}_{t\in \mathbb{Z}}$, for $l\in \{1, \ldots, q\}$, are i.i.d random variables independent of $\{\boldsymbol{X}_{t,T}\}_{t=1, \ldots, T}$. We denote the conditional distribution of $\boldsymbol{Y}_{t,T}|\boldsymbol{X}_{t,T}=\boldsymbol{x}$ by  $\boldsymbol{\pi}_t^\star(\cdot|\boldsymbol{x}) \in \mathcal{P}(\R^q)$. One example that fits this framework is the time-varying vector autoregressive (tvVAR) model \cite{lubik2015time, Haslbecketal2020, LI20241123}:
\begin{align*}
\boldsymbol{Y}_{t,T} = \boldsymbol{m}^\star\big(\frac tT, \boldsymbol{Y}_{t-1,T}, \ldots, \boldsymbol{Y}_{t-d,T}\big) + \boldsymbol{\varepsilon}_{t,T},
\end{align*}
where $\boldsymbol{X}_{t,T} = (\boldsymbol{Y}_{t-1,T},\ldots, \boldsymbol{Y}_{t-d,T})^\top$ is the $d$-lag of the $q$-dimensional vector $\boldsymbol{Y}_{t,T}$. The time-varying parameters of the mean function $\boldsymbol{m}^\star(\frac{t}{T}, \cdot)$ may involve linear or sigmoid smooth functions of the rescaled time $\frac{t}{T}$ \cite{Haslbecketal2020}.

\begin{definition}
\label{definition: pi_hat for k-dimensional Y}
The  NW estimator of $\boldsymbol{\pi}_t^\star(\cdot|\boldsymbol{x})$ is defined as $\boldsymbol{\hat{\pi}}_t(\cdot|\boldsymbol{x}) = \sum_{a=1}^T\omega_a(\frac tT,\boldsymbol{x})\delta_{\boldsymbol{Y}_{a,T}}$, where $\omega_{a}(\frac{t}{T}, \boldsymbol{x})$ is given in (\ref{def: weights}) and $\delta_{\boldsymbol{Y}_{a,T}}$ represents a point mass at $\boldsymbol{Y}_{a,T}\in\R^q$. The associated conditional CDF to $\boldsymbol{\hat{\pi}}_t(\cdot|\boldsymbol{x})$ writes as, for all $\boldsymbol{y} = (y^1, \ldots, y^q)^\top\in\R^q$, 
\begin{equation*}        \hat{F}_{t}(\boldsymbol{y}|\boldsymbol{x})=\sum_{a=1}^T\omega_{a}(\frac tT,\boldsymbol{x})\mathds{1}_{Y_{a,T}^{1}\leq y^{1},\ldots, Y_{a,T}^{q}\leq y^{q}}.
\end{equation*}
\end{definition}

\begin{remark}
The NW estimator of $\boldsymbol{m}^\star$ is given by $\hat{\boldsymbol{m}}(u,\boldsymbol{x}) = \sum_{a=1}^T \omega_a(u, \boldsymbol{x}) \boldsymbol{Y}_{a,T}.$
\end{remark}

When $\boldsymbol{Y}_{t, T}\in\R^q$, estimating the Wasserstein distance is often affected by the curse of dimensionality due to high computational complexity \citep{BayraktarGuo2021, Dombry2023}. To address this complexity, the metric sliced Wasserstein distance was introduced \citep{BayraktarGuo2021, Nadjahietal2021, XuHuang2022, Manoleetal2022}. It only requires estimating the distance of the projected unidimensional distributions.

\paragraph{Sliced Wasserstein distance.} Let $\mathbb{S}^{q-1} = \{{\boldsymbol{\theta}}\in\R^q: \norm{{\boldsymbol{\theta}}}_2 =1\}$ be the unit sphere in $\R^q$. Let ${\boldsymbol{\theta}}_\#: \R^q \rightarrow \R$ be the map defined by $\boldsymbol{\theta}_\#(\boldsymbol{v}) = \langle {\boldsymbol{\theta}}, \boldsymbol{v}\rangle = {\boldsymbol{\theta}}^\top \boldsymbol{v}$. For any $\boldsymbol{\mu}\in \mathcal{P}_1(\R^q)$ and ${\boldsymbol{\theta}}\in\mathbb{S}^{q-1}$, we define the push-forward measure 
\begin{align*}
    \boldsymbol{\theta}_\# \boldsymbol{\mu} (I) = \boldsymbol{\mu}(\{\boldsymbol{v}\in \R^q:  {\boldsymbol{\theta}}^\top \boldsymbol{v} \in I \}),
\end{align*}
for any $I$ Borelian in $\R$. For all $\boldsymbol{\mu}\in\mathcal{P}_1(\R^q)$ and ${\boldsymbol{\theta}}\in\mathbb{S}^{q-1}$, 
$\boldsymbol{\theta}_\# \boldsymbol{\mu} \in\mathcal{P}_1(\R)$ since it has a finite first moment in $\R$ \citep{BayraktarGuo2021}, i.e., \begin{align*}
    \int_\R |v|\boldsymbol{\theta}_\# \boldsymbol{\mu}(\diff v) = \int_{\R^q} |{\boldsymbol{\theta}}^\top \boldsymbol{v}| \boldsymbol{\mu}(\diff \boldsymbol{v}) \leq \int_{\R^q} \norm{\boldsymbol{v}} \boldsymbol{\mu}(\diff \boldsymbol{v}) < \infty.
\end{align*}
We next define the sliced Wasserstein distance of order one between $\boldsymbol{\mu}, \boldsymbol{\eta} \in \mathcal{P}_1(\R^q)$ denoted by $SW_1$ as follows.

\begin{definition}\label{defn: sliced and max-sliced Wasserstein distances}
    For $\boldsymbol{\mu}, \boldsymbol{\eta} \in \mathcal{P}_1(\R^q)$, the sliced Wasserstein distance of order one is defined as
    \begin{align}\label{def: sliced W}
        SW_1(\boldsymbol{\mu},\boldsymbol{\eta}) &= \int_{\mathbb{S}^{q-1}} W_1(\boldsymbol{\theta}_\# \boldsymbol{\mu}, \boldsymbol{\theta}_\# \boldsymbol{\eta}) \sigma_{q-1}(\diff {\boldsymbol{\theta}}),
    \end{align}
                    where $\sigma_{q-1}$ stands for the uniform measure on $\mathbb{S}^{q-1}$.
\end{definition}

Sliced Wasserstein distance can be determined by averaging the Wasserstein distance between random 1-dimensional projections of distributions. Generally, this metric is weaker than Wasserstein distance, but it still preserves similar properties, making it an alternative application computation \citep{Bonnotte2013, Manoleetal2022}.

Let ${\boldsymbol{\theta}}\in\mathbb{S}^{q-1}$, $\boldsymbol{\theta}_\# \boldsymbol{\pi}^\star_t(\cdot|\boldsymbol{x})$ is the pushforward measure of $\boldsymbol{\pi}^\star_t(\cdot|\boldsymbol{x})$ in the direction ${\boldsymbol{\theta}}$ with conditional CDF $F_{t, \boldsymbol{\theta}}^\star(\cdot|\boldsymbol{x})$. We estimate this pushforward measure by $\boldsymbol{\theta}_\# \hat{\boldsymbol{\pi}}_t(\cdot|\boldsymbol{x})$ with conditional CDF $\hat{F}_{t, \boldsymbol{\theta}}(\cdot|\boldsymbol{x})$ defined, for all $y\in\R$,
\begin{align}\label{eqn: CDF of projected pi-hat}
    \hat{F}_{t, \boldsymbol{\theta}}(y|\boldsymbol{x}) = \sum_{a=1}^T \omega_a(\frac{t}{T}, \boldsymbol{x}) \mathds{1}_{{\boldsymbol{\theta}}^\top \boldsymbol{Y}_{a,T} \leq y}.
\end{align}

\begin{assumption}[Conditional CDF for multivariate case]
\label{assumption: CDF multivariate case}
    For any $\boldsymbol{\theta}\in \mathbb{S}^{q-1}$, the projected conditional CDF $F_{\cdot, \boldsymbol{\theta}}^\star(\cdot|\cdot)$ is Lipschitzian, i.e., $\big| F_{a, {\boldsymbol{\theta}}}^\star (\cdot|\boldsymbol{x}) - F_{t, {\boldsymbol{\theta}}}^\star (\cdot| \boldsymbol{x'}) \big| \leq L_{F^{\star }_{\boldsymbol{\theta}}} \big( \norm{\boldsymbol{x} - \boldsymbol{x'}} + \big|\frac{a}{T} - \frac{t}{T}\big|\big)$, for some constant $L_{F_{\boldsymbol{\theta}}^\star}<\infty$, and for all $a, t \in \{1, \ldots, T\},$ $\boldsymbol{x}, \boldsymbol{x'}\in \R^d$.
\end{assumption}

Similar to the univariate case, we assume that the projected cumulative CDF $F_{\cdot, \boldsymbol{\theta}}^\star(\cdot|\cdot)$ likewise exhibits smooth behavior, changing slowly as observations change.

\begin{theorem}\label{Theorem: convergence of ESW1_multivariate Y}
 Let Assumptions \ref{Assumption: X is lsp} - \ref{Assumption: bandwidth} and \ref{Assumption: mixing} - \ref{assumption: CDF multivariate case} hold. Then, \begin{align*}
    \sup_{\boldsymbol{x}\in \mathcal{X}, \frac{t}{T}\in I_h}\E\big[SW_1\big(\hat{\boldsymbol{\pi}}_t(\cdot|\boldsymbol{x}), \boldsymbol{\pi}_t^\star(\cdot|\boldsymbol{x})\big)\big]
    &= \bigO\Big(\frac{1}{T^{\frac{1}{2}} h^{d + 1 - \frac{1}{p}(1 - \nu)}} + \frac{1}{T^\nu h^{d + \nu - 1}} + h\Big).
\end{align*}
\end{theorem}

Theorem \ref{Theorem: convergence of ESW1_multivariate Y} is an extension of Theorem \ref{Theorem: convergence of EW1} to the multivariate response $\boldsymbol{Y}_{t, T}\in\R^q$. We use sliced Wasserstein distance that allows the convergence of measures on $\R^q$ to be reduced to the convergence of their unidimensional projections with respect to direction ${\boldsymbol{\theta}}\in\mathbb{S}^{q-1}$. As a by-product, at a direction ${\boldsymbol{\theta}}$, the convergence of the multidimensional measure $\hat{\boldsymbol{\pi}}_t(\cdot|\boldsymbol{x})$ is identical to that of the univariate case. The proof directly follows the lines of Theorem \ref{Theorem: convergence of EW1}'s proof and is postponed to Appendix \ref{appendix: proof of convergence of ESW1_multivariate Y}. %By Theorem 2.3.\textit{i} in \cite{BayraktarGuo2021} (Lemma \ref{lemma: equivalence W SW maxSW}), we could say that since $\E\big[SW_1\big(\hat{\boldsymbol{\pi}}_t(\cdot|\boldsymbol{x}), \boldsymbol{\pi}_t^\star(\cdot|\boldsymbol{x})\big)\big]$ converges to zero, then $\E\big[W_1\big(\hat{\boldsymbol{\pi}}_t(\cdot|\boldsymbol{x}), \boldsymbol{\pi}_t^\star(\cdot|\boldsymbol{x})\big)\big]$ will also approach zero. For an i.i.d sequence of $\mu$-distributed random variables $(X_t)_{t\geq 1}$, $\mu\in\R^q$, it was shown in \cite{FournierGuillin2015} that, for $q\geq 1$ and $r\neq q/(q-1)$, $\E[W_1(\mu_T, \mu)] \leq C M_r^{1/r}(\mu)(T^{-\frac{1}{q}} + T^{-\frac{r-1}{r}})$, provided that the moment $M_r(\mu)<\infty$. 
\section{Numerical experiments}

\label{sec:numerical_experiments}
We conduct numerical experiments on synthetic and real-world datasets to calculate the empirical Wasserstein distance between NW estimator and true conditional CDF. We have made the implementation code of the experiments in Python using Pytorch and Scikit-learnpackages. The code
that generates all figures is available from \texttt{\url{https://github.com/mzalaya/wasslsp}} in
the form of annotated programs, together with notebook tutorials.

\subsection{Synthetic data}
We consider univariate response case $Y_{t, T}\in\R$ and illustrate the convergence of NW estimator wrt Wasserstein distance for each of the following processes:

\begin{figure*}[!ht]
    \centering
    \begin{subfigure}[t]{\linewidth}
    \begin{subfigure}[t]{0.5\linewidth}        \centering
        \includegraphics[width=.3\textwidth]{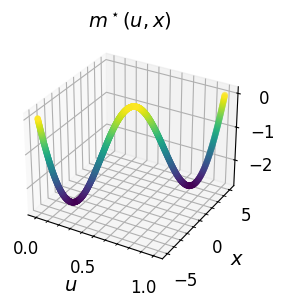}
                    \end{subfigure}    ~ 
    \begin{subfigure}[t]{0.5\linewidth}        \centering
        \includegraphics[width=1.\textwidth]{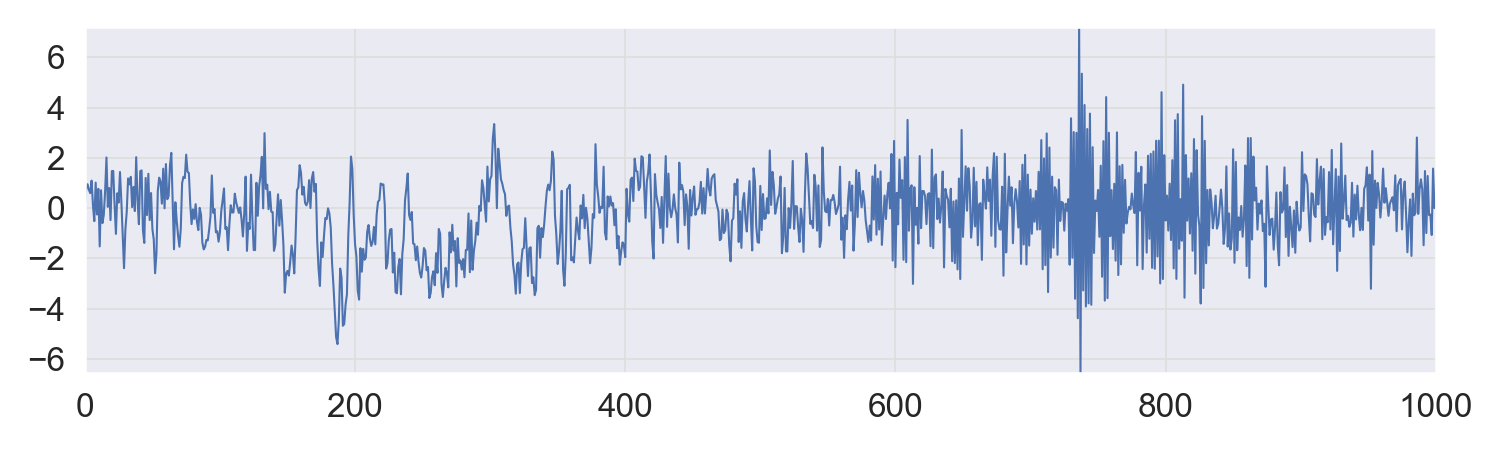}
                    \end{subfigure}
        \label{fig: gaussian_tvAR1}
    \end{subfigure}
        \begin{subfigure}[t]{\linewidth}
    \begin{subfigure}[t]{0.5\linewidth}        \centering
        \includegraphics[width=.3\textwidth]{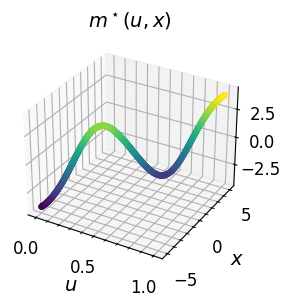}
                    \end{subfigure}    ~ 
    \begin{subfigure}[t]{0.5\linewidth}        \centering
        \includegraphics[width=1.\textwidth]{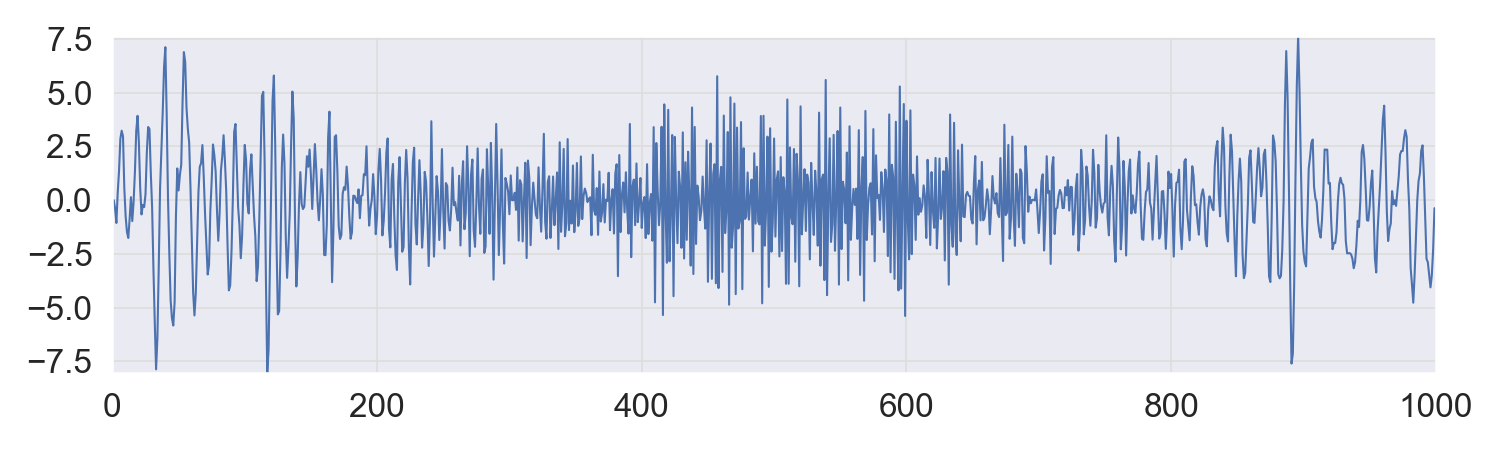}
                    \end{subfigure}
        \label{fig: gaussian_tvAR2}
    \end{subfigure}
        \begin{subfigure}[t]{\linewidth}
    \begin{subfigure}[t]{0.5\linewidth}        \centering
        \includegraphics[width=.3\textwidth]{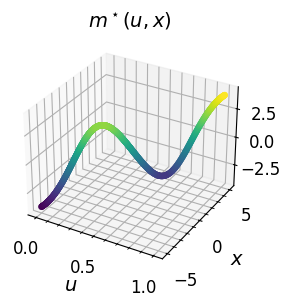}
                    \end{subfigure}    ~ 
    \begin{subfigure}[t]{0.5\linewidth}        \centering
        \includegraphics[width=1.\textwidth]{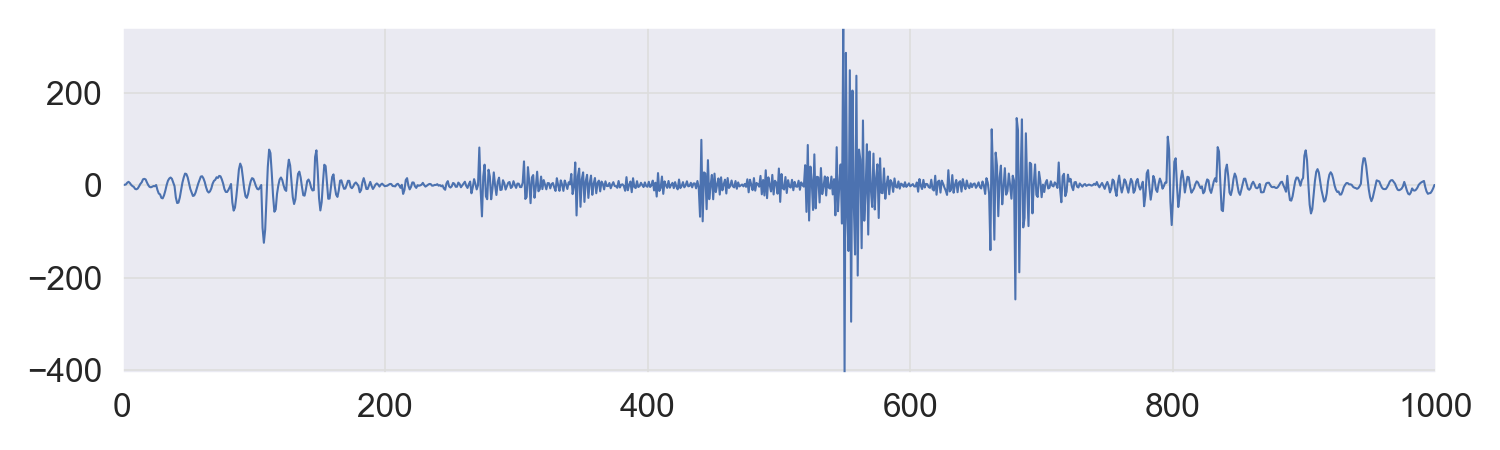}
                    \end{subfigure}
        \label{fig: cauchy_tvAR2}
    \end{subfigure}
        \begin{subfigure}[t]{\linewidth}
    \begin{subfigure}[t]{0.5\linewidth}        \centering
        \includegraphics[width=.3\textwidth]{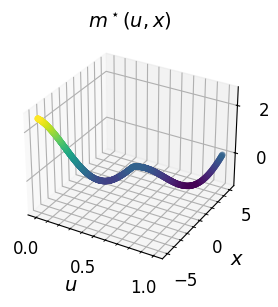}
                    \end{subfigure}    ~ 
    \begin{subfigure}[t]{0.5\linewidth}        \centering
        \includegraphics[width=1.\textwidth]{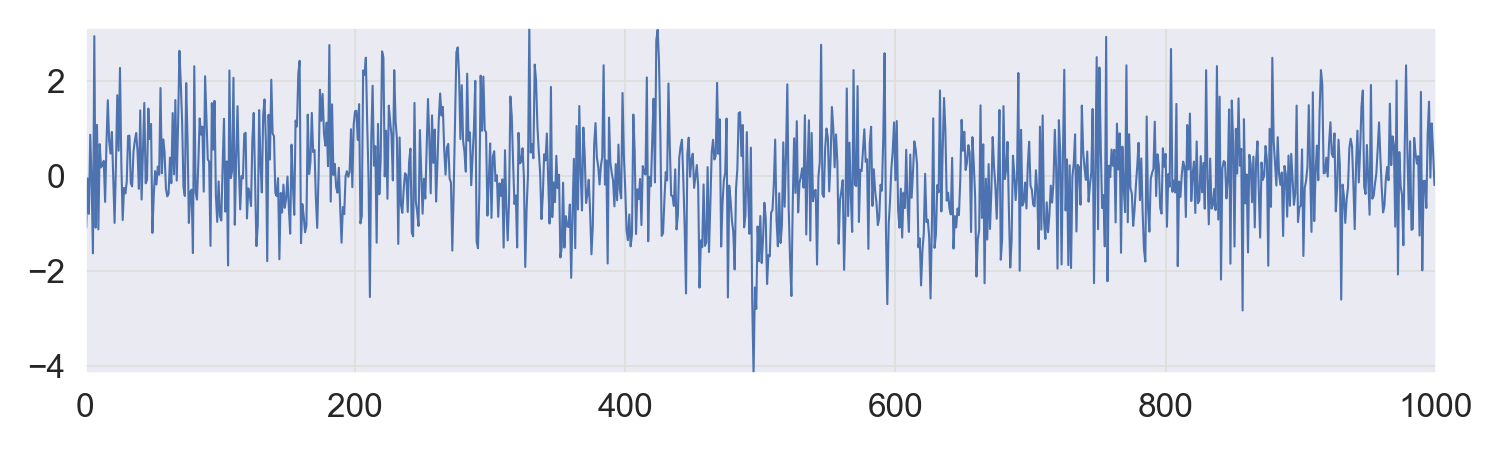}
                    \end{subfigure}
        \label{fig: gaussian_tvTAR1}
    \end{subfigure}
         \begin{subfigure}[t]{\linewidth}
    \begin{subfigure}[t]{0.5\linewidth}        \centering
        \includegraphics[width=.3\textwidth]{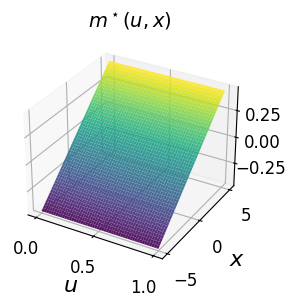}
    \end{subfigure}
    ~ 
    \begin{subfigure}[t]{0.5\linewidth}        \centering
        \includegraphics[width=1.\textwidth]{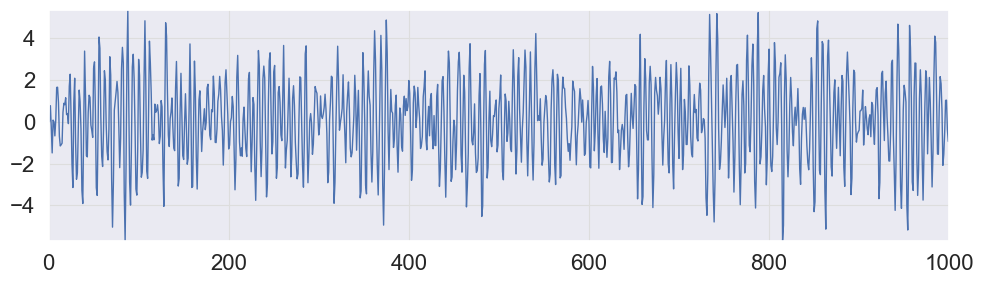}
    \end{subfigure}
    \label{fig: gaussian_AR2}
    \end{subfigure}
        \vspace{5pt}
    \caption{Time plots of the simulated processes (right portion) with their corresponding true conditional mean function $m^\star(u, \boldsymbol{x})$ (left portion) for sample size $T=1000$; from top to bottom: Gaussian tvAR(1), Gaussian tvAR(2), Cauchy tvAR(2), Gaussian tvTAR(1), and Gaussian AR(2).}
    \label{fig: simulated LSP}
\end{figure*}

\noindent{\it{Gaussian tvAR(1).}} The time-varying autoregressive model for $p = 1$, tvAR(1) \cite{RichterDahlhaus2019}, with Gaussian noise is defined by

\begin{equation*}
    Y_{t, T} = \alpha \big(\frac{t}{T}\big) Y_{t-1, T} +  \varepsilon_{t},
\end{equation*}
where $\alpha(u) = 0.9 \sin (2\pi u)$ and $\varepsilon_{t} \sim \mathcal{N}(0,1)$. Its strictly stationary approximation at rescaled time $u$, \cite{DAHLHAUS2012351}, is
\begin{equation*}
    Y_{t}(u) = \alpha (u) Y_{t-1}(u) + \zeta_{t},
\end{equation*}
where $\zeta_{t} \sim \mathcal{N}(0,1)$. The topmost time plot of Figure \ref{fig: simulated LSP} shows the resulting process $Y_{t, T}$ for $T = 1000$. There are gradual downward and upward trends between time points $t=100$ and $t=400$; however, these trends are smooth over time, that is, the values remain tight at finer time intervals. The mean of the whole series is roughly constant.
\\

\noindent{\it{Gaussian tvAR(2).}} We simulate the time-varying autoregressive model for $p=2$, tvAR(2) \citep{DAHLHAUS2012351}, with Gaussian noise:   
\begin{equation*}
    Y_{t,T} = 1.8 \cos\big(1.5 - \cos(2\pi \frac{t}{T})\big)Y_{t-1, T} - 0.81 Y_{t-2, T} + \varepsilon_{t},
\end{equation*}
where $\varepsilon_{t} \sim \mathcal{N}(0,1)$. The strictly stationary approximation of $Y_{t, T}$ at rescaled time $u$, \cite{DAHLHAUS2012351}, is
\begin{equation*}
    Y_{t}(u) = 1.8 \cos\big(1.5 - \cos(2\pi u)\big) Y_{t-1}(u) - 0.81 Y_{t-2}(u) + \zeta_{t},
\end{equation*}
where $\zeta_{t} \sim \mathcal{N}(0,1)$. For $T=1000$, the resulting process $Y_{t, T}$ exhibits nonstationarity through fluctuations as depicted in the second time plot of Figure \ref{fig: simulated LSP}. Particularly, it can be observed that the process has a constant mean and in the middle time points of the series, the oscillations are relatively rapid, indicating the process is quickly reverting to the mean. \\

\noindent{\it{Cauchy tvAR(2).}}
The third synthetic process is time-varying autoregressive model for $p=2$, tvAR(2) \citep{Birretal2017}, with Cauchy noise:   
\begin{equation*}
    Y_{t,T} = 1.8 \cos\big(1.5 - \cos(2\pi \frac{t}{T})\big)Y_{t-1, T} - 0.81 Y_{t-2, T} + \varepsilon_{t},
\end{equation*}
with i.i.d. Cauchy noise $\varepsilon_{t}$. For a rescaled time $u$, the strictly stationary approximation reads as
\begin{equation*}
    Y_{t}(u) = 1.8 \cos\big(1.5 - \cos(2\pi u)\big)Y_{t-1}(u) - 0.81 Y_{t-2}(u) + \zeta_{t},
\end{equation*}
with i.i.d. Cauchy noise $\zeta_{t}$. The process $Y_{t, T}$, for $T=1000$, in this example is depicted in the third time plot of Figure \ref{fig: simulated LSP}. Most observations in the series are centered around zero with relatively low-valued fluctuations. However, the stationarity of the process is affected by the intermittent high-valued spikes at some time points of the series, which are due to the heavy-tailed property of Cauchy distributed error term $\varepsilon_{t}$ \cite{rojo2013heavy, jaber2024analysis}.\\

\noindent{\it Gaussian tvTAR(1).} We next consider the time-varying threshold autoregressive model for $p=1$, tvTAR(1) \cite{RichterDahlhaus2019}, with Gaussian noise:
\begin{equation*}
    Y_{t,T} = \alpha_1\big(\frac{t}{T}\big) Y_{t-1, T}^{+} + \alpha_2\big(\frac{t}{T}\big) Y_{t-1, T}^{-} + \varepsilon_t,
\end{equation*}
where $\alpha_1(u) = 0.4 \sin (2 \pi u)$, $\alpha_2(u) = 0.5 \cos (2 \pi u)$, $y^{+} = \max\{y, 0\}$, $y^{-} = \max\{-y, 0\}$, and $\varepsilon_t \sim \mathcal{N}(0,1)$. This can be approximated at rescaled time $u$ by a strictly stationary process given by
\begin{equation*}
    Y_{t}(u) = \alpha_1(u) Y_{t-1}^{+}(u) + \alpha_2(u) Y_{t-1}^{-}(u) + \zeta_t,
\end{equation*}
where $\zeta_{t} \sim \mathcal{N}(0,1)$. As shown in the fourth time plot of Figure \ref{fig: simulated LSP}, the series practically has a constant mean. Though there are trends in the series, the values still remain tight.\\

\noindent{\it Gaussian AR(2).} Finally, we also give an example of a stationary autoregressive process of order $p=2$, AR(2), with Gaussian noise:   
\begin{equation*}
    Y_{t} = 0.9 Y_{t-1} - 0.81 Y_{t-2} + \varepsilon_{t},
\end{equation*}
where $\varepsilon_{t} \sim \mathcal{N}(0,1)$. The characteristic equation of this process is given by $1 - 0.9 z + 0.81 z^2 = 0$, whose roots lie outside the unit circle; hence, the given process is stationary. This process is plotted at the bottom of Figure \ref{fig: simulated LSP}, which behaves stationarily compared to the plot of Gaussian tvAR(2).

The conditional mean functions $m^\star(u,x)$ of these example processes are also plotted in Figure \ref{fig: simulated LSP}, correspondingly placed beside each time plot. As shown, only the conditional mean function of the Gaussian AR(2) process is stationary for different values of $u\in[0, 1]$.

\paragraph{Monte Carlo simulations.}

Note that true conditional probability distribution and NW estimator are calculated for a fixed time $t \in \{1, \ldots, T\}$. Hence, obtaining these quantities from a \textit{single one-shot sampling} is impossible. We replicate each process $L = 1000$ and calculate NW conditional CDF at specified time $t$, for each $l\in\{1, \ldots, L\}$. Using these $L$ replications, we calculate the average NW and the empirical conditional CDFs. We then measure the corresponding Wasserstein distance. The replicated data-generating procedure is given in Algorithm \ref{alg: simulated data}.

\LinesNotNumbered
\begin{algorithm}[htbp]
\DontPrintSemicolon
\SetNlSty{textbf}{}{.}
\SetKwInOut{Input}{input}
\SetKwInOut{Return}{return}
\caption{Data generating and NW estimation for synthetic data} 
\label{alg: simulated data}
\nl \Input{ sample size $T$, time point $t \in \{1, \ldots, T\} $, number of replications $L$, based kernels $K_1(\cdot), K_2(\cdot)$, bandwidth $h;$}
\nl \For{$l = 1, \ldots, L$}{
    \# \texttt{Generate $l$-th replication process} $\{Y_{a, T}^{(l)}\}_{a=1, \ldots, T}$\\
    \For{$a=1,\ldots,T$}{
   $Y_{a,T}^{(l)}  \gets  m^\star\big(\frac aT, \boldsymbol{X}^{(l)}_{a,T}\big) + \varepsilon_{a,T}^{(l)};$\\
    }
     \# \texttt{Calculate $l$-th NW conditional CDF estimator}\\
     $\displaystyle \hat{F}_{t}^{(l)}(y|\boldsymbol{x}) \gets \sum_{a=1}^T \omega_{a}(\frac tT,\boldsymbol{x})  \mathds{1}_{Y_{a,T}^{(l)}\leq y};$
}
\# \texttt{Calculate average NW estimator }\\
\nl $\displaystyle \hat{F}_{t}^L(y|\boldsymbol{x}) \gets  \frac{1}{L} \sum_{l=1}^L \hat{F}_{t}^{(l)}(y|\boldsymbol{x});$\\
 \# \texttt{Calculate empirical conditional CDF}\\
 \nl $ \displaystyle F_t^L(y|\boldsymbol{x}) \gets \frac{1}{L} \sum_{l=1}^L \mathds{1}_{Y_{t,T}^{(l)}\leq y};$\\
\nl \Return{$W_1(\hat{F}_t^L(y|\boldsymbol{x}),  F_t^L(y|\boldsymbol{x}) );$}
\end{algorithm}

\begin{figure}[!ht]
\centering
\begin{subfigure}[b]{0.48\textwidth}
\centering
\includegraphics[width=\textwidth]{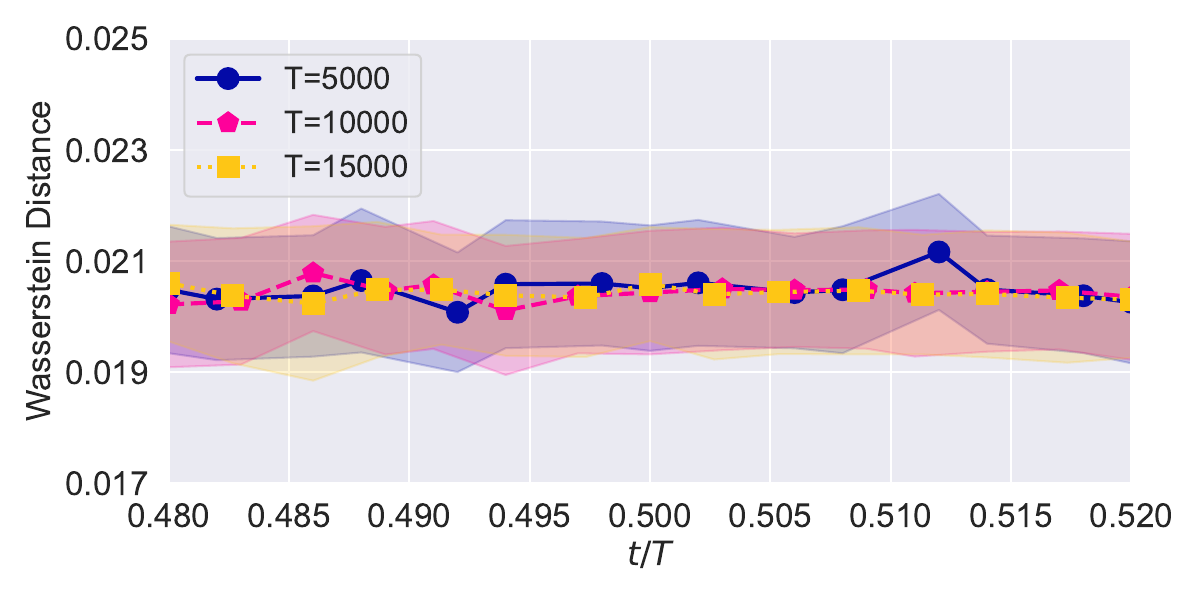} 
\caption{\centering Gaussian tvAR(1), $K_1= \texttt{Uniform}, K_2 = \texttt{Silverman}$ }
\label{fig:convergence_gaussian_tvAR1}
\end{subfigure}
\begin{subfigure}[b]{0.48\textwidth}
\centering
\includegraphics[width=\textwidth]{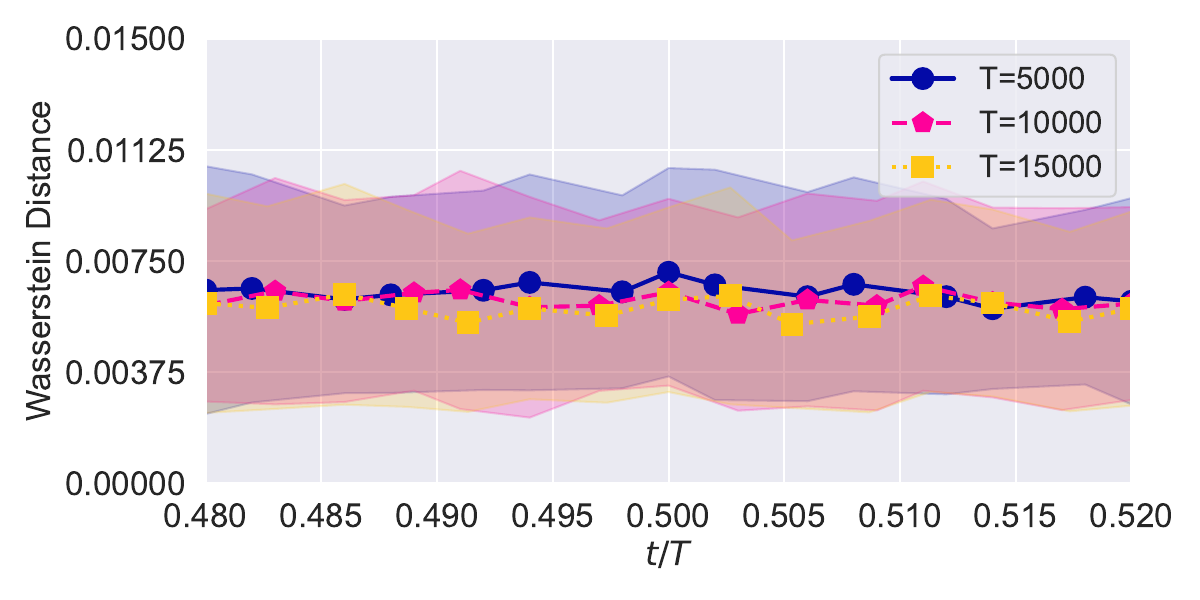} 
\caption{\centering Gaussian tvAR(2), $K_1= \texttt{Rectangle}, K_2 = \texttt{Silverman}$ }
\label{fig:convergence_gaussian_tvAR2}
\end{subfigure}
\hfill
\begin{subfigure}[b]{0.48\textwidth}
\centering
\includegraphics[width=\textwidth]{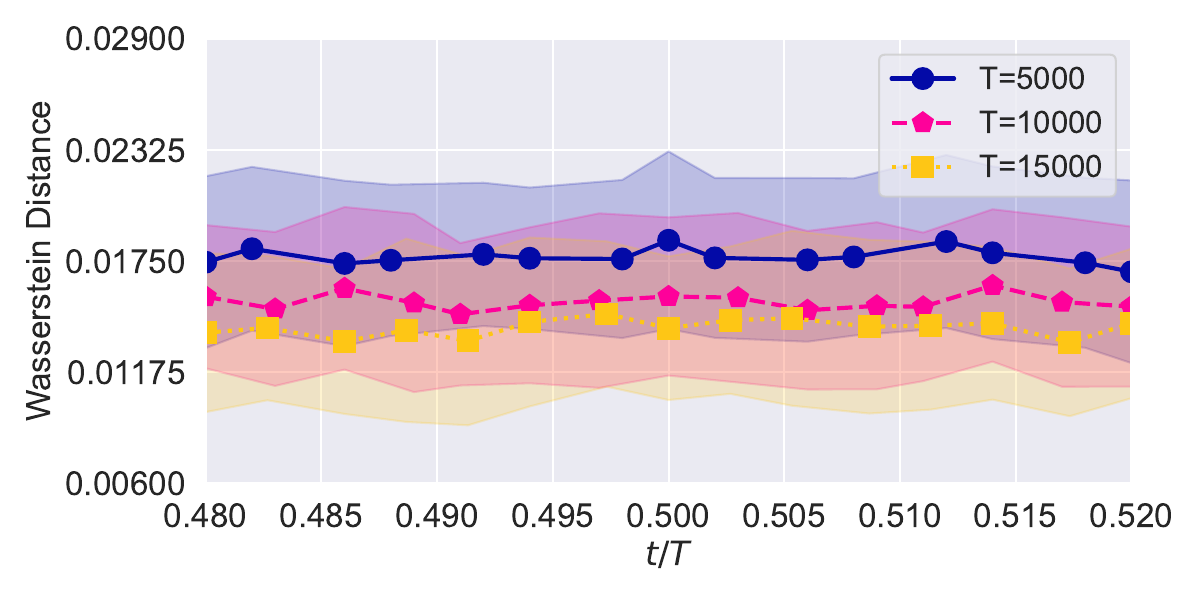} 
\caption{\centering Cauchy tvAR(2), $K_1= \texttt{Triangle}, K_2 = \texttt{Gaussian}$ }
\label{fig:convergence_cauchy_tvAR12}
\end{subfigure}
\begin{subfigure}[b]{0.48\textwidth}
\centering
\includegraphics[width=\textwidth]{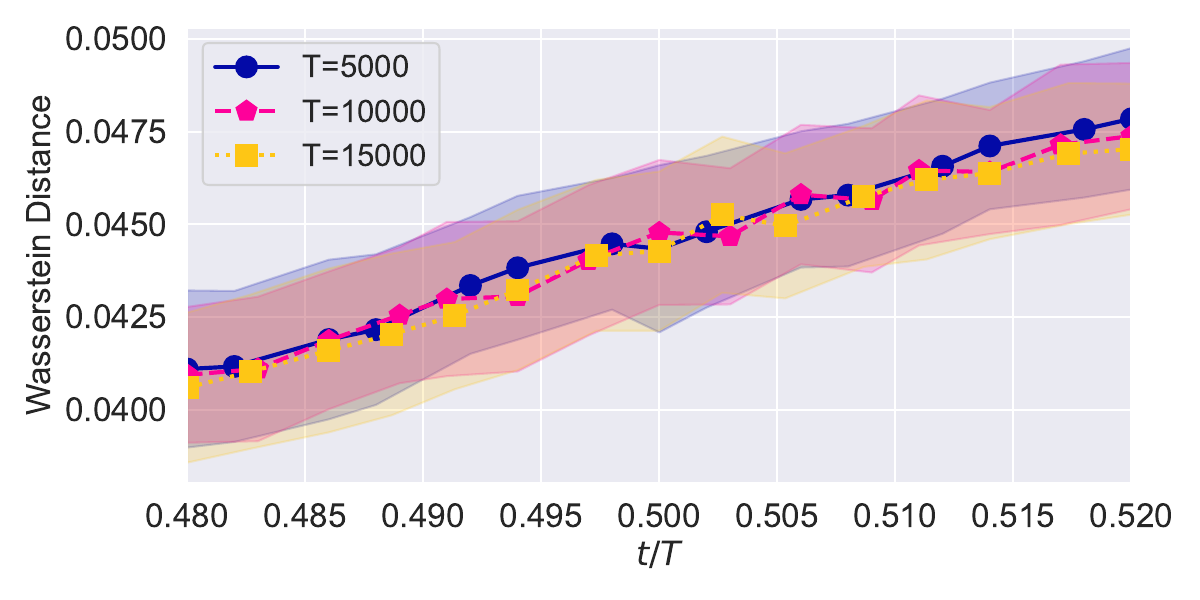} 
\caption{\centering Gaussian tvTAR(1); $K_1 = \texttt{Tricube},  K_2 = \texttt{Gaussian} $}
\label{fig:convergence_gaussian_tvTAR1}
\end{subfigure}
\begin{subfigure}[b]{0.48\textwidth}
\centering
\includegraphics[width=\textwidth]{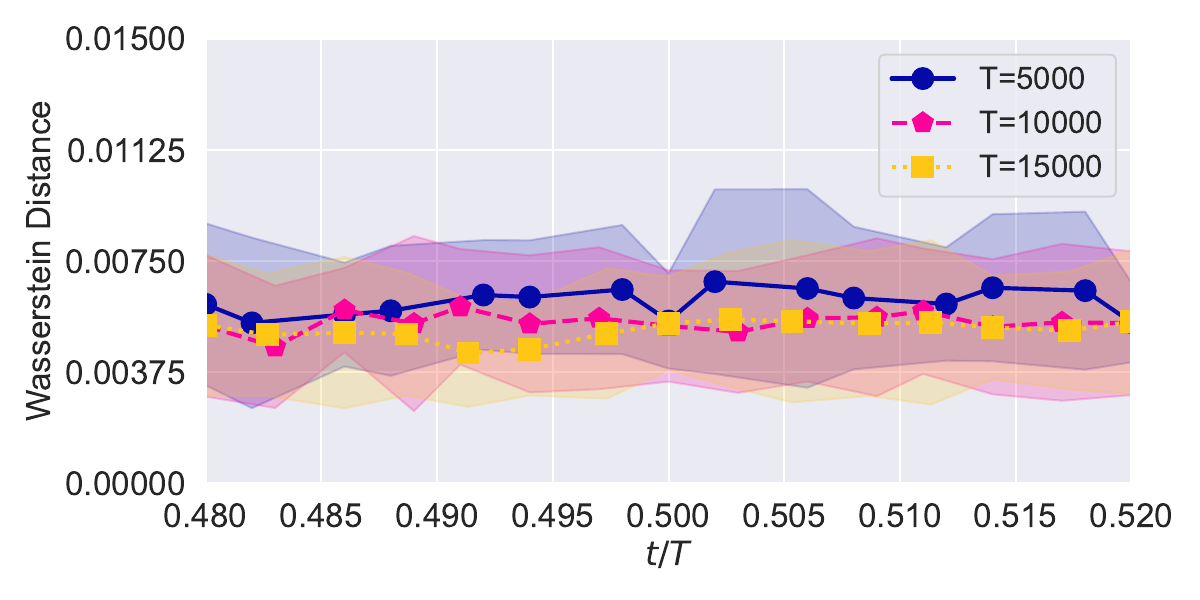} 
\caption{\centering Gaussian AR(2), $K_1= \texttt{Rectangle}, K_2 = \texttt{Silverman}$ }
\label{fig:convergence_gaussian_AR2}
\end{subfigure}
\vspace{5pt}
\caption{Wasserstein distances $\pm$ standard deviation at different $u = \frac{t}{T}$ for $T = 5000, 10000, 15000$ using various kernels for $K_1$ and $K_2$; $L=1000$ replications and 100 Monte Carlo runs. }
\label{fig: convergence at various t increasing T}
\end{figure}

To illustrate theoretical results in Section \ref{sec:theoretical_guarantees}, we provide 100 Monte Carlo runs of Algorithm \ref{alg: simulated data} to get the expected $W1$ distance between the underlying conditional distributions. We consider various kernels $K_1(\cdot)$ and $K_2(\cdot)$ for the chosen processes. We set increasing sample sizes $T = 5000, 10000, 15000$. We select $h = T^{-\xi}$, where $\xi = \frac{0.2}{d+1}$ for Gaussian tvAR(1) and Gaussian tvTAR(1), and $\xi = \frac{0.3}{d+1}$ for Gaussian tvAR(2), Cauchy tvAR(2), and Gaussian AR(2). Recall that our theoretical results are valid when $\frac{t}{T} \in I_h$. For based kernel $K_1$ belonging to \texttt{Uniform}, \texttt{Rectangle}, \texttt{Triangle}, and \texttt{tricube}, the constant $C_1 = 1$ and $I_h = [h, 1-h]$.

Figure \ref{fig: convergence at various t increasing T} conveys the expected Wasserstein distances along with the corresponding standard deviations. Note that our theoretical results are expressed using the supremum. For each considered process, it is shown that the maximum expected Wasserstein distances are smaller for larger sample sizes $T$, which validates our theoretical result. Specifically, we list the maximum expected Wasserstein distance for each sample size $T=5000, 10000, 15000$ as follows: Gaussian tvAR(1), $(0.0212, 0.0208, \textbf{0.0206})$; Gaussian tvAR(2), $(0.0071, 0.0067, \textbf{0.0063})$; Cauchy tvAR(2), $(0.0186, 0.0162, \textbf{0.0148})$; Gaussian tvTAR(1), $(0.0478, 0.0474, \textbf{0.0470})$; and Gaussian AR(2), $(0.0068, 0.0060, \textbf{0.0055})$. Observe that, for each process, the minimum among the maximum expected Wasserstein distances is consistently attained at the largest $T=15000$. It is worth noticing that the convergence rate depends on local stationarity approximation, in particular for Gaussian tvAR(2) and Cauchy tvAR(2). Wasserstein distances of Gaussian tvAR(2) are relatively smaller than Cauchy tvAR(2). This could be explained by the local stationarity of the process that can be affected by the extremely large fluctuations in the case of Cauchy tvAR(2). In addition, Wasserstein distances of Gaussian AR(2) are relatively smaller than Gaussian tvAR(2), indicating that the proposed estimation method is even more accurate when applied to stationary data. In general, the produced Wasserstein distances for all considered processes are small, signifying that the NW conditional distribution estimator is robust in dealing with nonstationarity and extreme values.

\subsection{Real-world data}

We use BabyECG ($T = 2048$), SP500 ($T = 8372$), and HRV ($T = 17178$) datasets. The BabyECG dataset contains a record of the heart rate (in beats per minute) of a 66-day-old infant. It has $T = 2048$ observations sampled every 16 seconds. The Standard \& Poors' SP500 index dataset contains $T = 8372$ observations from 1971 to 2018. These values are the differences of the logarithms of daily opening and closing prices. Lastly, the HRV dataset records $T = 17178$ observations of instantaneous noninterpolated heart rate (niHR) frequency measured in beats per minute (bpm). This is calculated directly from the time intervals between consecutive heartbeats without any form of interpolation.  

\LinesNotNumbered
\begin{algorithm}[ht]
\DontPrintSemicolon
\SetNlSty{textbf}{}{.}
\SetKwInOut{Input}{input}
\SetKwInOut{Return}{return}
\caption{Gaussian smoothed procedure and NW estimation for real datasets}
\label{alg: real data}
\nl \Input{ real dataset $\{Y_{a, T}\}_{a=1, \ldots, T}$,  $\sigma >0$, time point $t \in \{1, \ldots, T\} $, number of replications $L$, based kernels $K_1(\cdot), K_2(\cdot)$, bandwidth $h;$}
\nl \For{$l = 1, \ldots, L$}{
    \# \texttt{Generate $l$-th replication} $\{Y_{a, T}^{(l)}\}_{a=1, \ldots, T}$\\
    \For{$a=1,\ldots,T$}{
    $Y_{a,T}^{(l)} \gets  Y_{a,T} + Z_{a,T}^{(l)} $, where $Z_{a,T}^{(l)} \sim \mathcal{N}(0, \sigma^2);$\\
    }
     \# \texttt{Calculate $l$-th NW conditional CDF estimator}\\
     $\displaystyle \hat{F}_{t}^{(l)}(y|\boldsymbol{x}) \gets \sum_{a=1}^T \omega_{a}(\frac tT,\boldsymbol{x})  \mathds{1}_{Y_{a,T}^{(l)}\leq y};$
}
                      \# \texttt{Calculate average NW estimator }\\
\nl $\displaystyle \hat{F}_{t}^L(y|\boldsymbol{x}) \gets  \frac{1}{L} \sum_{l=1}^L \hat{F}_{t}^{(l)}(y|\boldsymbol{x});$\\
 \# \texttt{Calculate empirical conditional CDF}\\
 \nl $ \displaystyle F_t^L(y|\boldsymbol{x}) \gets \frac{1}{L} \sum_{l=1}^L \mathds{1}_{Y_{t,T}^{(l)}\leq y};$\\
\nl \Return{$W_1(\hat{F}_t^L(y|\boldsymbol{x}),  F_t^L(y|\boldsymbol{x}) );$}
\end{algorithm}

In order to produce Wasserstein distances, we create copies of these datasets through replication, as was done for synthetic experiments. The replication scheme relies on Gaussian smoothed procedure \cite{Nietertetal2021}. Since Corollary 1 in \cite{Nietertetal2021} ensures that $\lim_{\sigma \rightarrow 0} W(\mu, \nu + \mathcal{N}(0, \sigma^2)) = W(\mu, \nu)$, for $\mu, \nu \in \mathcal{P}_1(\R)$, we can add $Z_{t, T}^{(l)} \sim \mathcal{N}(0,\sigma^2)$ with $\sigma > 0$ to each data observation $Y_{t,T}$ at each replication $l$,  for all $t \in \{1, \ldots, T\}$:
$$Y_{t,T}^{(l)} =  Y_{t,T} + Z_{t,T}^{(l)}. $$ 
We replicate these Gaussian-smoothed datasets $L$ times and calculate NW conditional CDF at a specific time point $t$. We calculate the average NW and the empirical conditional CDFs and measure the corresponding Wasserstein distance. Algorithm \ref{alg: real data} details the replicated Gaussian smoothness of the data.

Figure~\ref{fig: noised datasets} presents example time plots of Gaussian-smoothed datasets $Y_{t, T}^{(l)}$ with $Z_{t,T}^{(l)}\sim \mathcal{N}(0, 1)$ and $l=1,2,3$. We simply show the plots for $L=3$ replications to clearly exhibit their behavior. Looking at each replication, it is notable that SP500 has a constant mean and is considered a white noise process \cite{Birretal2017}. Meanwhile, the mean of BabyECG and HRV changes gradually. 

\begin{figure}[htbp]
\centering
\begin{subfigure}[b]{.8\textwidth}
\centering
\includegraphics[height=1.1in, width=1.\textwidth]{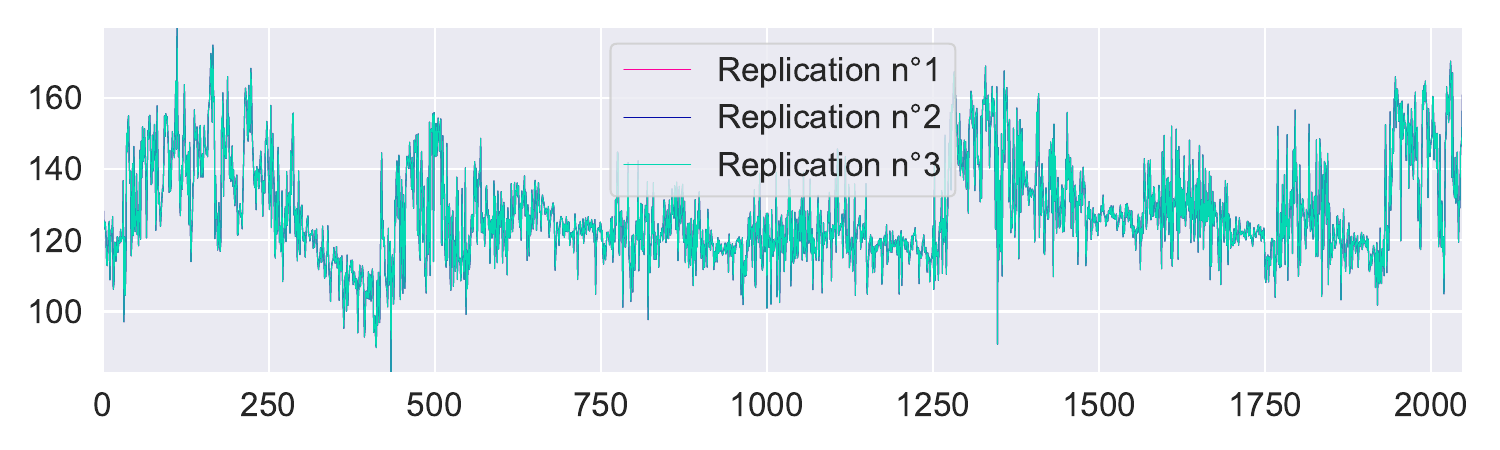} 
\caption{Gaussian-smoothed BabyECG ($T = 2048$)}
\label{fig: Noised BabyECG}
\end{subfigure}
\begin{subfigure}[b]{.8\textwidth}
\centering
\includegraphics[height=1.1in, width=1.\textwidth]{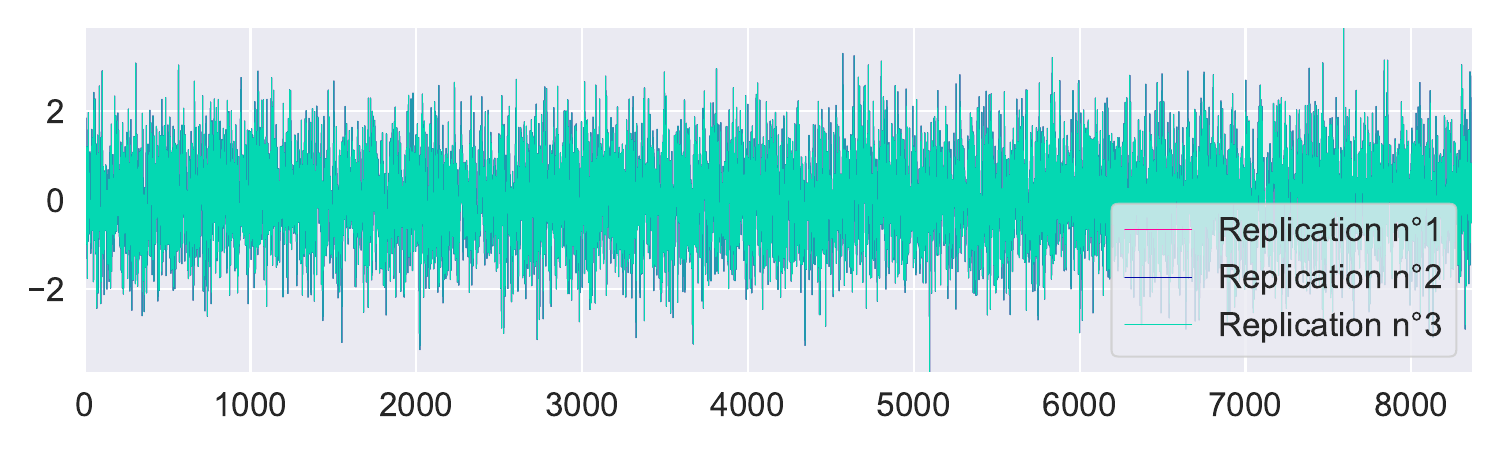} 
\caption{Gaussian-smoothed SP500 ($T = 8372$)}
\label{fig: Noised SP500}
\end{subfigure}
\begin{subfigure}[b]{.8\textwidth}
\centering
\includegraphics[height=1.1in, width=1.\textwidth]{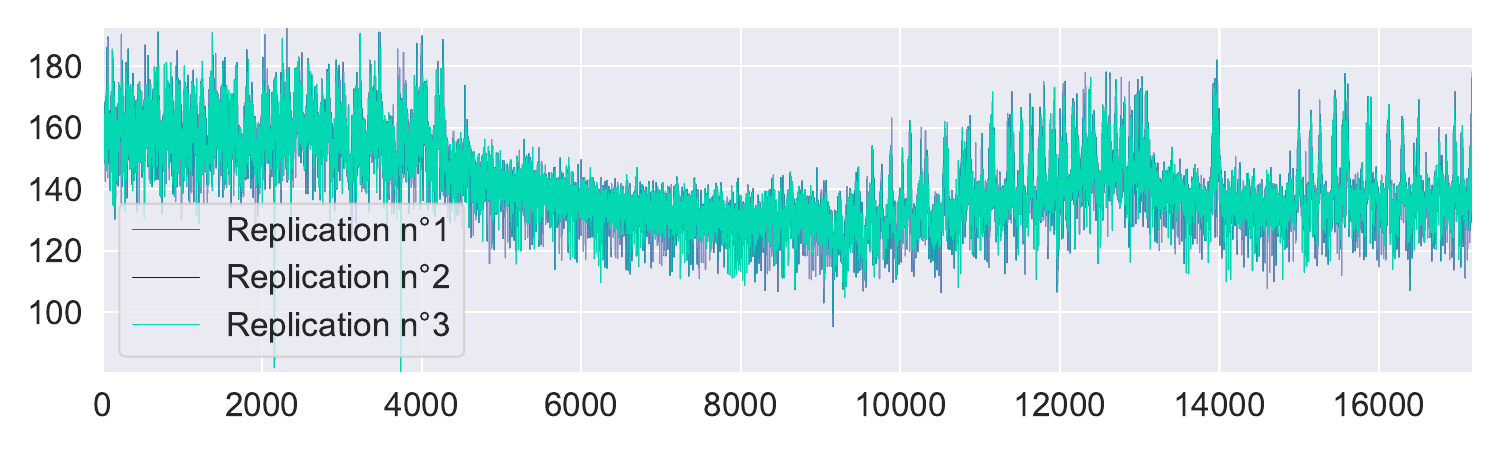} 
\caption{Gaussian-smoothed HRV ($T = 17178$)}
\label{fig: Noised logHRV}
\end{subfigure}
\vspace{5pt}
\caption{Example plots of $Y_{t,T}^{l}$: real datasets added with Gaussian noise $\mathcal{N}(0, 1)$ for replications $l=1, 2, 3$ and $t=1,\ldots,T$.}
\label{fig: noised datasets}
\end{figure}

\begin{figure}[htpb]
\centering
\begin{subfigure}[b]{0.32\textwidth}
\centering
\includegraphics[width=\textwidth]{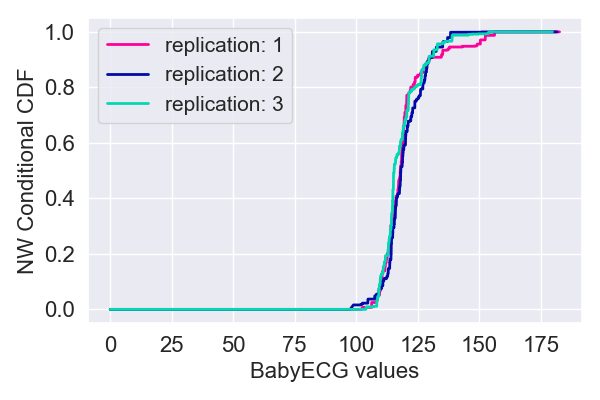} 
\caption{\centering Gaussian-smoothed BabyECG; $t=970$}
\label{fig: BabyECG NW conditional CDFs}
\end{subfigure}
\begin{subfigure}[b]{0.32\textwidth}
\centering
\includegraphics[width=\textwidth]{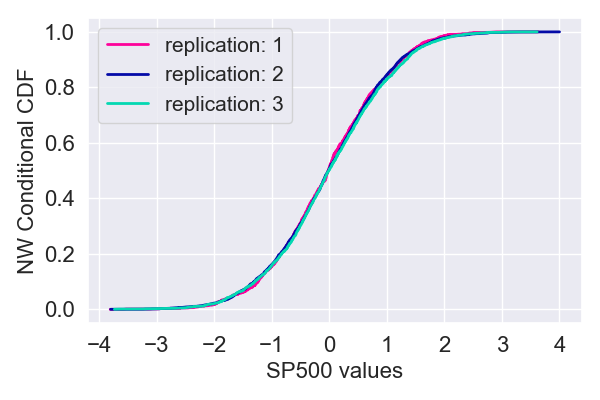} 
\caption{\centering Gaussian-smoothed SP500; $t=4480$}
\label{fig: SP500 NW conditional CDFs}
\end{subfigure}
\begin{subfigure}[b]{0.32\textwidth}
\centering
\includegraphics[width=\textwidth]{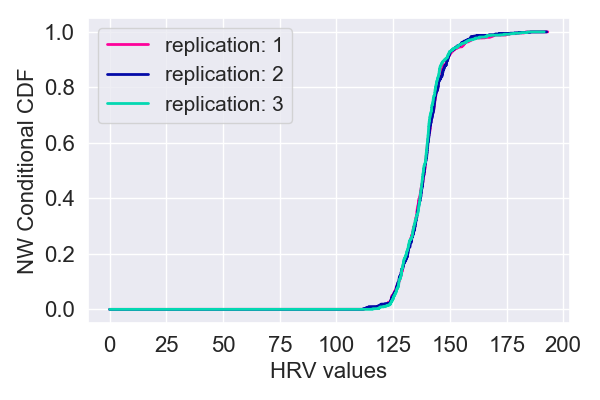} 
\caption{\centering Gaussian-smoothed HRV; $t = 7950$}
\label{fig: HRV NW conditional CDFs}
\end{subfigure}
\vspace{5pt}
\caption{Plots of estimated NW conditional CDFs for each Gaussian smoothed-dataset shown in Figure \ref{fig: noised datasets} at specified $t$ using $K_1= \texttt{Uniform} \text{ and } K_2 = \texttt{Gaussian}$. } \label{fig: sample NW conditional CDFs real datasets}
\end{figure}

Hereafter, we quantify NW conditional CDF using uniform and Gaussian kernels for $K_1$ and $K_2$, respectively. Similarly, we select $h = T^{-\xi}$ for $\xi = \frac{0.2}{d+1}$, and $d=1$. We initially illustrate the result of this estimation procedure using the Gaussian-smoothed datasets for only $L=3$ replications depicted in Figure \ref{fig: noised datasets}. 

\begin{figure}[ht]
    \centering
    \begin{subfigure}[t]{\linewidth}
    \begin{subfigure}{0.32\textwidth}
        \centering
        \subcaption*{$S = \frac{1}{3}T$}
        \includegraphics[width=\linewidth]{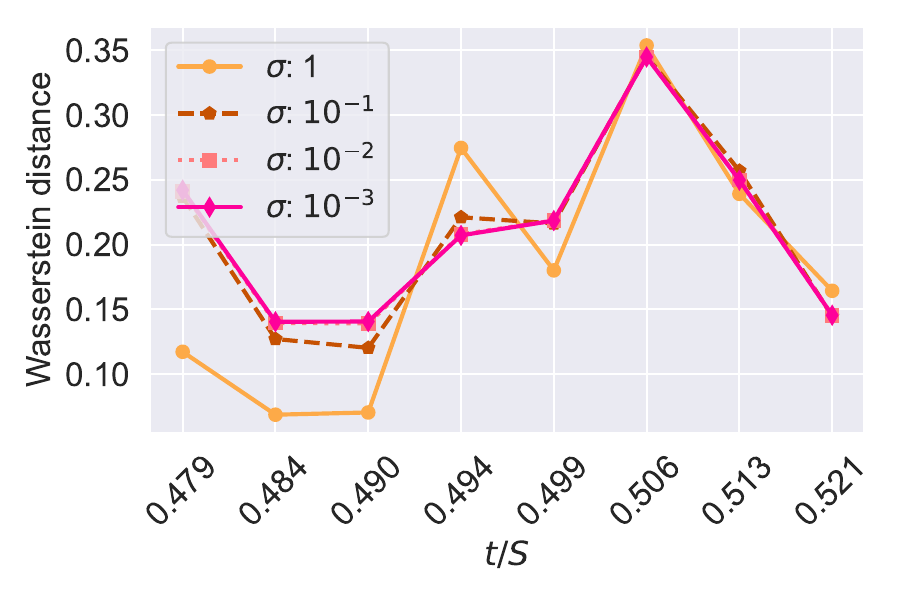}
            \end{subfigure}
    \begin{subfigure}{0.32\textwidth}
        \centering
        \subcaption*{$S = \frac{2}{3}T$}
        \includegraphics[width=\linewidth]{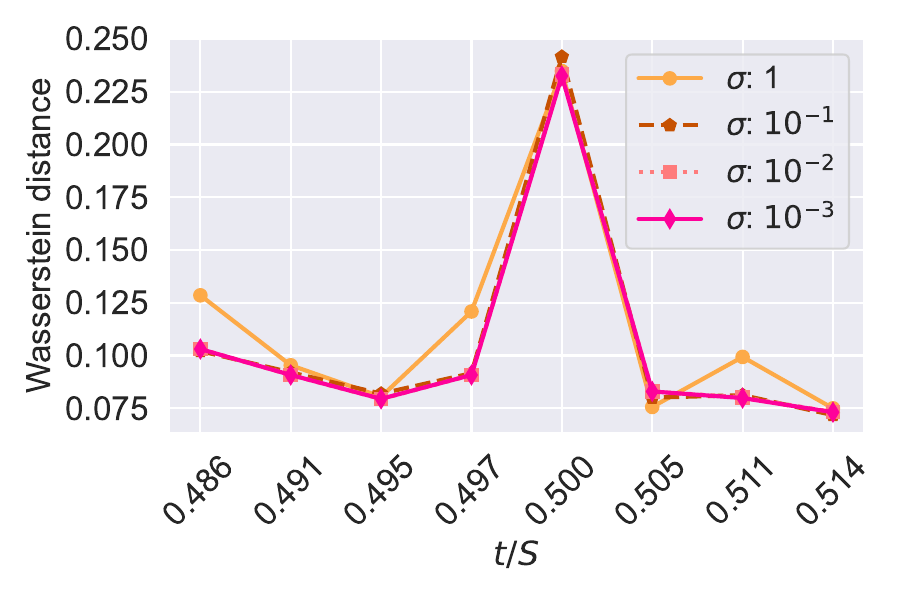}
            \end{subfigure}
    \begin{subfigure}{0.32\textwidth}
        \centering
        \subcaption*{$S = T$}
        \includegraphics[width=\linewidth]{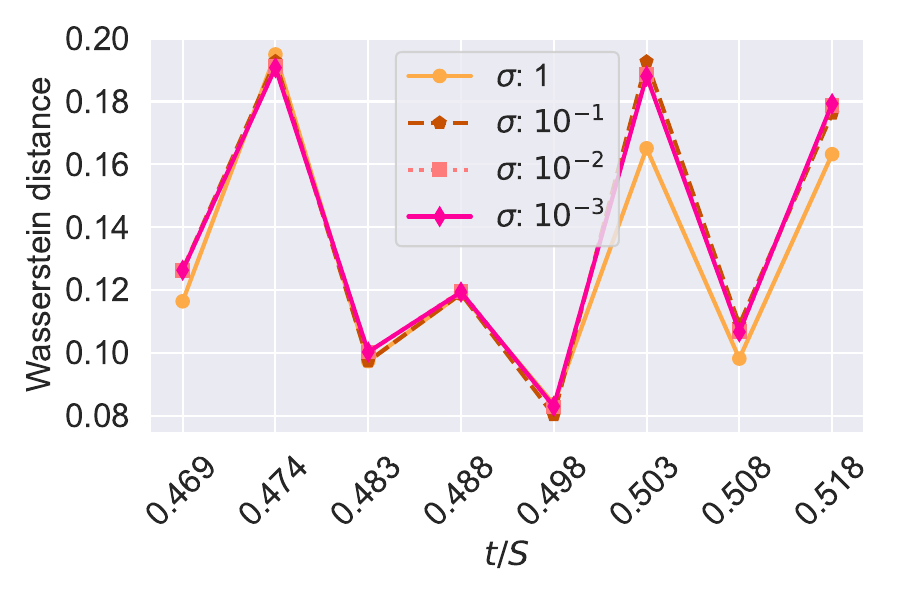}
            \end{subfigure}
    \subcaption{Gaussian-smoothed BabyECG}
    \end{subfigure}
        \begin{subfigure}[t]{\linewidth}
    \begin{subfigure}{0.32\textwidth}
        \centering
        \includegraphics[width=\linewidth]{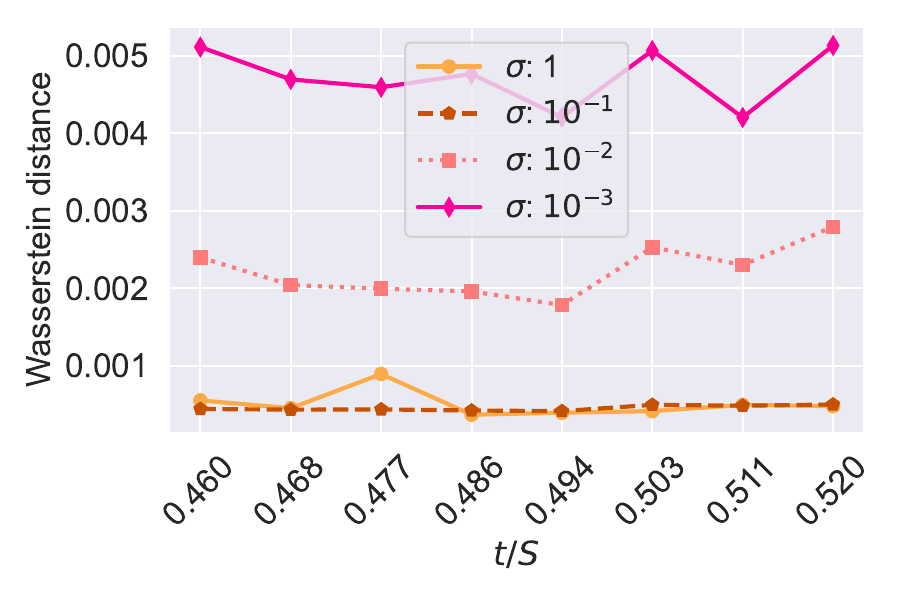}
            \end{subfigure}
    \begin{subfigure}{0.32\textwidth}
        \centering
        \includegraphics[width=\linewidth]{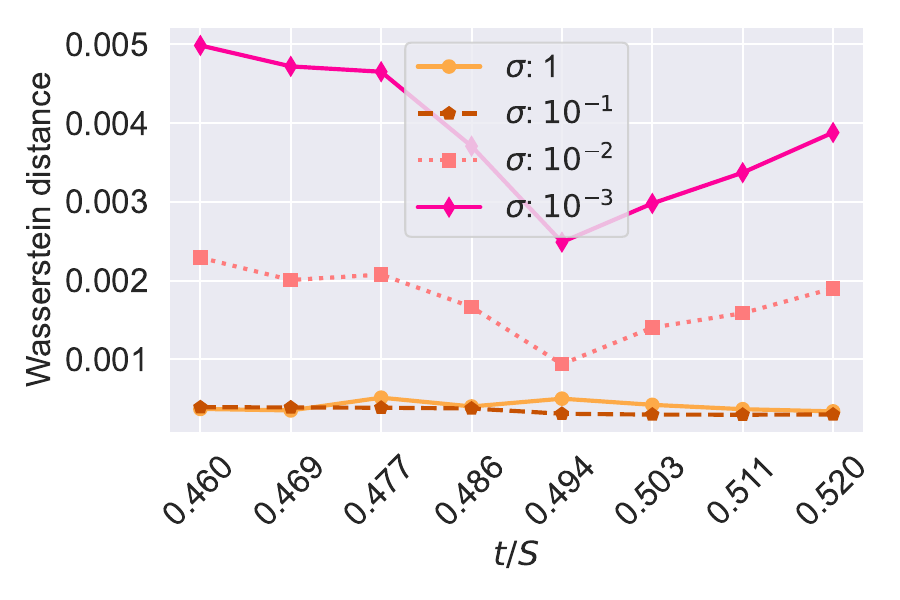}
            \end{subfigure}
    \begin{subfigure}{0.32\textwidth}
        \centering
        \includegraphics[width=\linewidth]{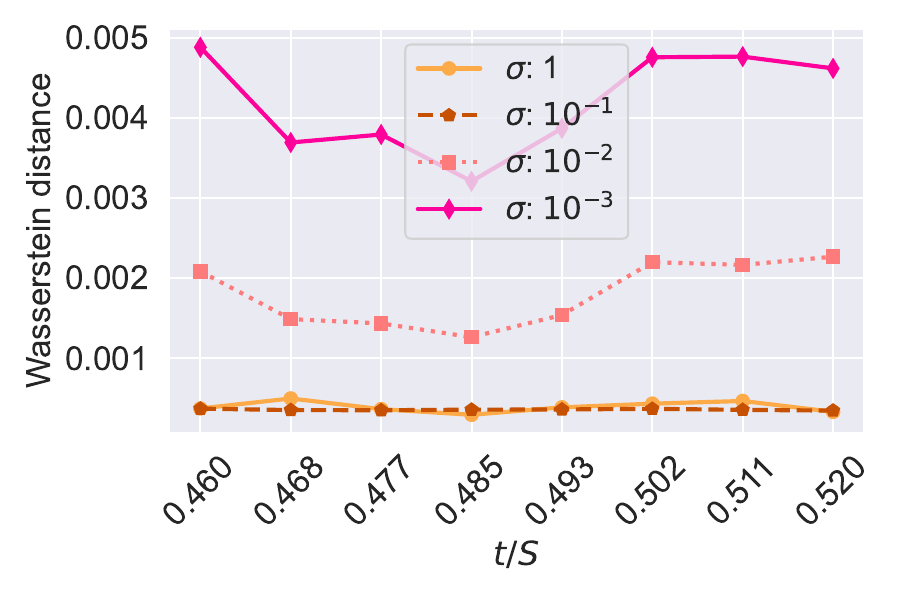}
            \end{subfigure}
    \subcaption{Gaussian-smoothed SP500}
    \end{subfigure}
        \begin{subfigure}[t]{\linewidth}
    \begin{subfigure}{0.32\textwidth}
        \centering
        \includegraphics[width=\linewidth]{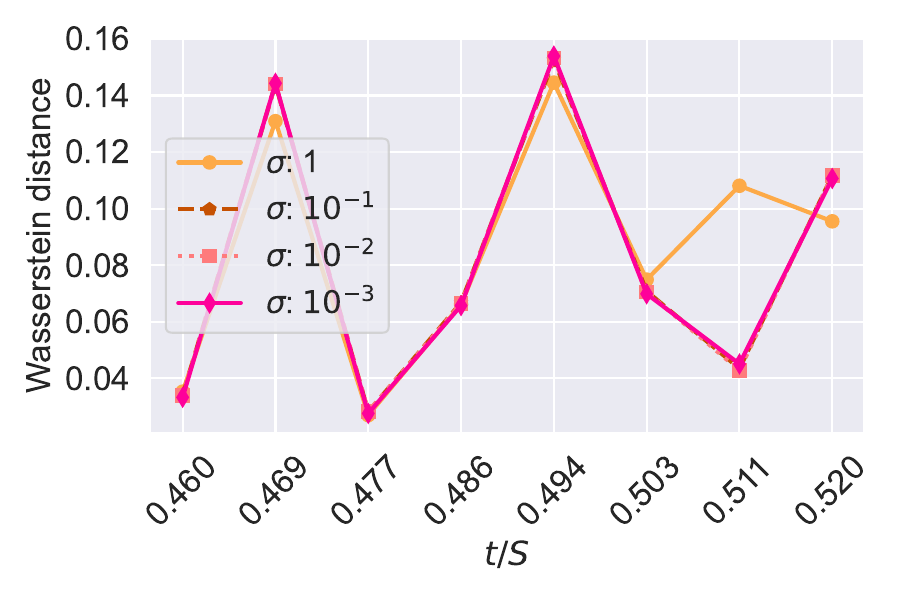}
            \end{subfigure}
    \begin{subfigure}{0.32\textwidth}
        \centering
        \includegraphics[width=\linewidth]{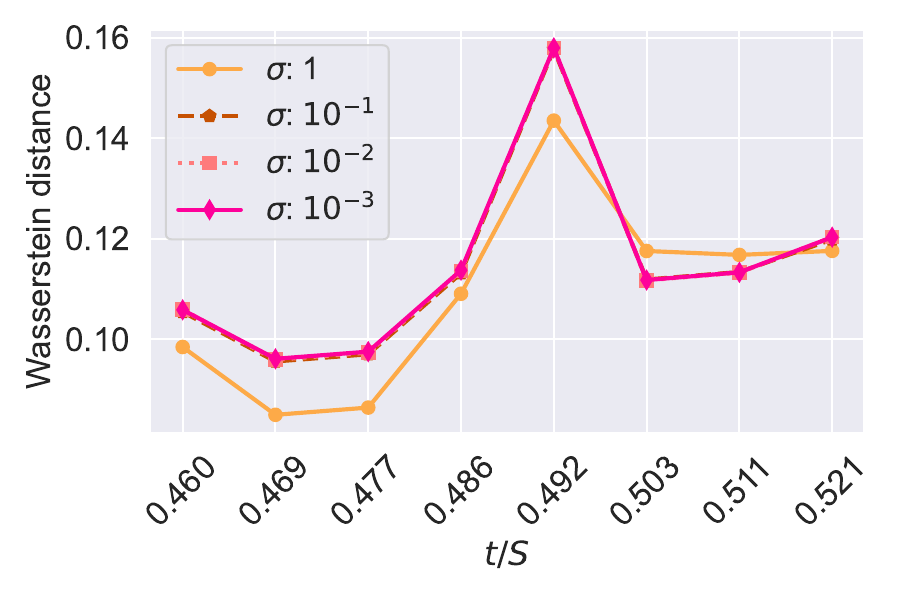}
            \end{subfigure}
    \begin{subfigure}{0.32\textwidth}
        \centering
        \includegraphics[width=\linewidth]{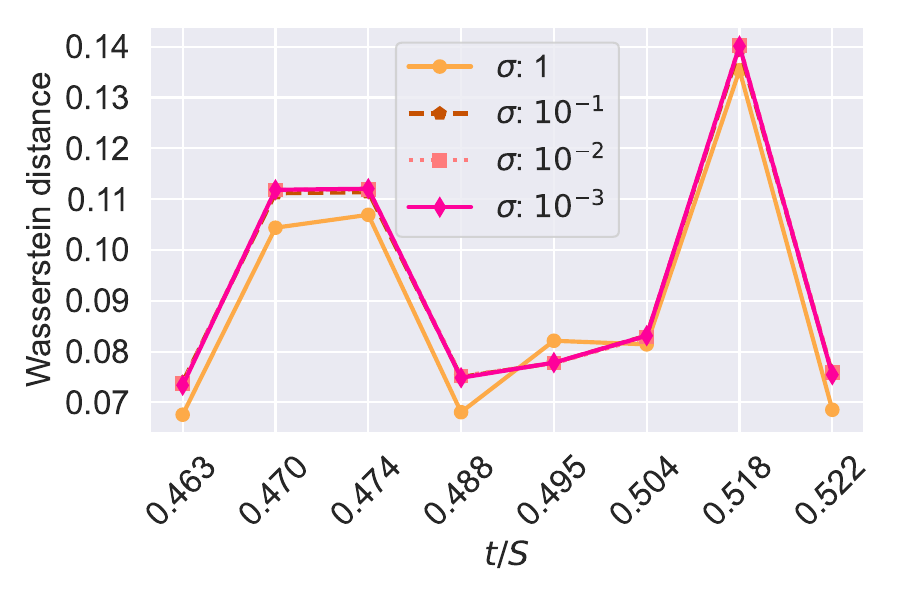}
            \end{subfigure}
    \subcaption{Gaussian-smoothed HRV}
    \end{subfigure}
    \vspace{5pt}
    \caption{Wasserstein distance between true conditional CDFs of Gaussian-smoothed datasets, for different smoothness level $\sigma$, and corresponding NW conditional CDF estimators using $K_1= \texttt{Uniform} \text{ and } K_2 = \texttt{Gaussian}$ and sample size partitions $S = \frac{T}{3}, \frac{2T}{3}, T$. }     \label{fig: real datasets W1 various sigmas}
\end{figure}

Figure \ref{fig: sample NW conditional CDFs real datasets} shows plots of NW conditional CDFs of Gaussian-smoothed BabyECG at $t=970$, Gaussian-smoothed SP500 at $t=4480$, and Gaussian-smoothed HRV at $t=7950$. We can observe that NW conditional CDFs of Gaussian-smoothed SP500 and HRV, having more data points, tend to be smoother. 

Recall that this real data experiment relies on Gaussian smoothing with parameter $\sigma > 0$. We next conduct an experiment to check the behavior of Wasserstein distance for various $\sigma > 0$ and increasing sample size. 
Towards this end, we cut the observations at $S\in\{\frac{T}{3}, \frac{2T}{3}, T\}$.\textbf{}
We set $L = 1000$ and    
$\sigma \in  \{1, 10^{-1}, 10^{-2}, 10^{-3}\}$.
Similarly, since we use a uniform kernel for $K_1$, we fix $t$ such that $\frac{t}{S} \in [h, 1-h]$. The next steps are then executed using Algorithm~\ref{alg: real data}. Figure \ref{fig: real datasets W1 various sigmas} shows the resulting Wasserstein distances that are smaller for datasets with larger sample sizes. Plots in the first column represent Wasserstein distances for the partition $S=\frac{1}{3}T$, the second column for $S=\frac{2}{3}T$, and the last column for $S=T$. For each Gaussian-smoothed dataset, the Wasserstein distance tends to be smaller for larger partitions $S=\frac{1}{3}T, T$, validating our theoretical results. Due to the stationarity of SP500 \cite{Birretal2017}, its corresponding distances are smaller than those of the other datasets, a similar observation from our synthetic experiment. It can also be observed that Wasserstein distance for Gaussian-smoothed SP500 increases as $\sigma$ gets smaller since as $\sigma\rightarrow 0$, the Gaussian-smoothed SP500 tends to behave as the original SP500.

\section{Conclusion} \label{sec:conclusion}

We investigated Nadaraya-Watson (NW) conditional probability estimation for LSP. Convergence rates were established wrt the Wasserstein distance in the univariate setting and the sliced Wasserstein distance in the multivariate case. These rates are determined by the degree of deviation from the local stationarity approximation and the weak dependence structure of the process. Additionally, we provided an explicit convergence rate when the bandwidth is selected as \( h = \bigO(T^{-\xi}) \), where \( 0 < \xi < \frac{\frac{1}{2} \wedge \nu}{d+1} \).   We conducted numerical experiments using both synthetic and real-world datasets. We proposed a data-generating procedure for the synthetic data to compute the NW estimator, while for the real-world data, we used a Gaussian kernel. 

One aspect that remains unexplored in this article is the best selection of the smoothing parameters to minimize  Wasserstein distance. The subject at hand holds significant importance and warrants dedicated research effort.  We defer this matter to a forthcoming investigation. 
Additionally, this work opens avenues for future research, including:  
\textit{(i)} replacing the basic indicator function with an integrated kernel \( H_g(y - Y_{t, T}) \), where \( H \) represents a smooth cumulative distribution function (CDF) and \( H_g(y - Y_{t, T}) \) serves as a local weighting function with bandwidth \( g \);  
\textit{(ii)} employing a kernel estimator based on an additive model, as developed in \cite{vogt2012}, to mitigate the curse of dimensionality;  
\textit{(iii)} adapting the NW estimator in Definition \ref{definition: pi_hat} to accommodate missing data; \textit{(iv)} similar to \cite{MR4783436}, considering a local polynomial approach to minimize the large bias at the boundary region of kernel estimation.

\paragraph{Acknowledgements.} The work of Jan Nino G. Tinio is supported by the Department of Science and Technology - Science Education Institute (DOST-SEI) in partnership with Campus France through a PhilFrance-DOST scholarship grant.

\clearpage
\appendix 
\section{Proofs of main results}\label{Appendix: Proofs of the main results}

Before providing the proofs of the main results, we begin with the following propositions that will be useful in the succeeding proofs.

\begin{proposition}\label{Lemma: E of K2}
    Let Assumptions \ref{Assumption: X is lsp} to \ref{assumption: CDF} hold. 
    Then, for $a,t \in \{1, \ldots, T\}$, the following inequalities hold:
    \begin{enumerate}[label=(\roman*)]
    \item $ \E \Big[\Big| \prod_{j=1}^d K_{h,2}(x^j - X_{a,T}^j) - \prod_{j=1}^d K_{h,2}\big(x^j - X^j_a\big(\frac{a}{T}\big) \big)\Big|\Big] 
        \leq \frac{L_2 C_U C_2^{d-\nu} d^\frac{3}{2} }{T^\nu h^\nu}. $
                    \item $\E \Big[\Big| \prod_{j=1}^d K_{h,2}(x^j - X_{a,T}^j)\Big|\Big] 
        \leq \frac{L_2 C_U C_2^{d-\nu} d^\frac{3}{2} }{T^\nu h^\nu}  + h^d f(\frac{t}{T}, \boldsymbol{x}) + h^{d+2} \frac{M}{2} \kappa d.$
                    \item  ${K_{h,1}\big( \frac{t}{T} - \frac{a}{T} \big) \E \Big[ \prod_{j=1}^d K_{h,2}(x^j - X_{a,T}^j) [\mathds{1}_{Y_{a,T}\leq y} - F_t^\star(\cdot|\boldsymbol{x})] \Big]}\\
       \hspace{3ex}\qquad \leq (\sqrt{d}C_2 + C_1) L_{F^\star} K_{h,1}\big( \frac{t}{T} - \frac{a}{T} \big) \Big\{ \frac{L_2 C_U C_2^{d-\nu} d^\frac{3}{2} }{T^\nu h^{\nu-1}}  + h^{d+1} f(\frac{t}{T}, \boldsymbol{x}) + h^{d+3} \frac{M}{2} \kappa d \Big\},$
                    \end{enumerate}
    where $\nu=\rho \wedge 1$, $\kappa = \int z^2 K_2(z) \diff z$, and $\sum_{j=1}^d \big| \partial_j f(\frac{t}{T}, \boldsymbol{x}) \big| \leq M$.
\end{proposition}

\begin{proof}
\noindent{\textit{(i)}} Using Lemma \ref{lemma: Prod of K2 leq Sum}.(\textit{i}), we get 
$$ \Big| \prod_{j=1}^dK_{h,2}(x^j-X_{a,T}^j) - \prod_{j=1}^dK_{h,2}\big(x^j-X_a^j\big(\frac{a}{T}\big)\big)\Big| \leq C_2^{d-1} \sqrt{d} \sum_{j=1}^d \big| K_{h,2}(x^j-X_{a,T}^j) - K_{h,2}\big(x^j-X_a^j\big(\frac{a}{T}\big)\big)\big|.$$ 
In addition, by Assumption~\ref{Assumption: kernel functions}, $K_2$ is bounded by $C_2$. Also, for any bounded function $|f(x)|\leq \iota$, we have $|f(x)|^{1-\nu}\leq \iota^{1-\nu}$, which implies that $|f(x)|\leq \iota^{1-\nu}|f(x)|^\nu$, for $1-\nu \geq 0$. This means that
\begin{align*}
    \big| K_{h,2}(x^j-X_{a,T}^j) - K_{h,2}\big(x^j-X_a^j\big(\frac{a}{T}\big)\big)\big| \leq C_2^{1-\nu} \big| K_{h,2}(x^j-X_{a,T}^j) - K_{h,2}\big(x^j-X_a^j\big(\frac{a}{T}\big)\big)\big|^\nu.
\end{align*}
Accordingly,
\begin{align*}
    \E \Big[\Big| &\prod_{j=1}^d K_{h,2}(x^j - X_{a,T}^j) - \prod_{j=1}^d K_{h,2}\big(x^j - X_a^j\big(\frac{a}{T}\big) \big) \Big|\Big]\\
    &\leq C_2^{d-1} \sqrt{d}\, \E \big[\sum_{j=1}^d \big| K_{h,2}(x^j-X_{a,T}^j) - K_{h,2}\big(x^j-X_a^j\big(\frac{a}{T}\big)\big)\big| \big]\\
    &\leq C_2^{d-\nu} \sqrt{d}\, \E \big[\sum_{j=1}^d \big| K_{h,2}(x^j-X_{a,T}^j) - K_{h,2}\big(x^j-X_a^j\big(\frac{a}{T}\big)\big)\big|^\nu \big].
\end{align*}
Additionally, again by Assumption~\ref{Assumption: kernel functions}, $K_2$ is Lipschitz, so we get
\begin{align*}
    \E \Big[\Big| &\prod_{j=1}^d K_{h,2}(x^j - X_{a,T}^j) - \prod_{j=1}^d K_{h,2}\big(x^j - X_a^j\big(\frac{a}{T}\big) \big) \Big|\Big]\\
    &\leq \E \big[L_2 C_2^{d-\nu} \sqrt{d} \sum_{j=1}^d \big| \big(\frac{x^j-X_{a,T}^j}{h}\big) - \big(\frac{x^j-X_a^j(\frac{a}{T})}{h}\big)\big|^\nu \big]\\
    &\leq \E \big[L_2 C_2^{d-\nu} \sqrt{d} \sum_{j=1}^d \big| \frac{1}{h} \big( X_{a,T}^j - X_a^j\big(\frac{a}{T}\big) \big) \big|^\nu \big]\\
    &= \frac{L_2 C_2^{d-\nu} \sqrt{d}}{h^\nu} \sum_{j=1}^d \E \big[\big| \big( X_{a,T}^j - X_a^j\big(\frac{a}{T}\big) \big) \big|^\nu \big].
\end{align*}
Note that $\big|X_{a,T}^j -X_{a,T}^j\big(\frac{a}{T}\big)\big|\leq \norm{X_{a,T}^j -X_{a,T}^j\big(\frac{a}{T}\big)}_1$ and by Assumption \ref{Assumption: X is lsp}, $\norm{X_{a,T}^j -X_{a,T}^j\big(\frac{a}{T}\big)}_1 \leq \frac{1}{T}U_{a,T}\big(\frac{a}{T}\big)$, where $\E\big[\big(U_{a,T}\big(\frac{a}{T}\big)\big)^\nu\big]< C_U$, so we get
\begin{align*}
    \E \Big[\Big| &\prod_{j=1}^d K_{h,2}(x^j - X_{a,T}^j) - \prod_{j=1}^d K_{h,2}\big(x^j - X_a^j\big(\frac{a}{T}\big) \big) \Big|\Big]\\
    &\leq \frac{L_2 C_2^{d-\nu} \sqrt{d}}{T^\nu h^\nu} \sum_{j=1}^d \E \big[\big| U_{a,T}\big(\frac{a}{T}\big) \big|^\nu \big] \\
    &\leq \frac{L_2 C_U C_2^{d-\nu} d^\frac{3}{2} }{T^\nu h^\nu},
\end{align*}
which approaches to zero using Assumption \ref{Assumption: bandwidth}.
\noindent{\textit{(ii)}} Using Assumption~\ref{Assumption: X is lsp}, $\boldsymbol{X}_{a,T}$ is locally stationary, so
\begin{align*}
    \E \Big[\Big|&\prod_{j=1}^d K_{h,2}(x^j - X_{a,T}^j)\Big|\Big]\\
        &\leq \E \Big[\Big|\prod_{j=1}^d K_{h,2}(x^j - X_{a,T}^j) - \prod_{j=1}^d K_{h,2}\big(x^j - X_a^j\big(\frac{a}{T}\big) \big)\Big|\Big] + \E\Big[\Big|\prod_{j=1}^d K_{h,2}\big(x^j - X_a^j\big(\frac{a}{T}\big) \big)\Big|\Big]\\
    &\leq \frac{L_2 C_U C_2^{d-\nu} d^\frac{3}{2} }{T^\nu h^\nu}  + \E\Big[\Big|\prod_{j=1}^d K_{h,2}\big(x^j - X_a^j\big(\frac{a}{T}\big) \big)\Big|\Big],
\end{align*}
using \textit{(i)}. For the second term in the previous inequality, we have
\begin{eqnarray*}
   \E\Big[\Big|\prod_{j=1}^d K_{h,2}\big(x^j - X_a^j\big(\frac{a}{T}\big) \big)\Big|\Big]
    = \int \cdots \int K_{h,2}(x^1 - y^1) \cdots K_{h,2}(x^d - y^d)  f(\frac{t}{T} , y^1,\ldots,y^d) \diff y^1 \cdots \diff y^d.
\end{eqnarray*} 
Let $z^j = \frac{x^j - y^j}{h}$ implying that $y^j = x^j - hz^j$ and $\diff y^j = -h \diff z^j$. So,

\begin{align*}
    \lefteqn{\E\Big[\Big|\prod_{j=1}^d K_{h,2}\big(x^j - X_a^j\big(\frac{a}{T}\big) \big)\Big|\Big]}\\
    &= \int \cdots \int K_{2}(z^1) \cdots K_{2}(z^d) f(\frac{t}{T}, x^1 - hz^1,\ldots,x^d - hz^d) (-h) \diff z^1 \cdots (-h) \diff z^d.
\end{align*} 
Using Assumption~\ref{Assumption: X is lsp}, we can use the first order Taylor expansion of $f(\frac{t}{T}, x^1 - hz^1,\ldots,x^d - hz^d)$ wrt all $x^j$. Letting $f(\frac{t}{T}, x^1,\ldots,x^d) = f(\frac{t}{T}, \boldsymbol{x})$, we have
\begin{align*}    f(\frac{t}{T}, x^1 - hz^1,\ldots,x^d - hz^d)
    &= f(\frac{t}{T}, x^1,\ldots,x^d) + \sum_{j=1}^d \partial_j f(\frac{t}{T}, x^1,\ldots,x^d) (-h)z^j + R_1(h\boldsymbol{z}) \nonumber\\
    &= f(\frac{t}{T}, \boldsymbol{x}) + \sum_{j=1}^d \partial_j f(\frac{t}{T}, \boldsymbol{x}) (-h)z^j + R_1(h\boldsymbol{z}).
\end{align*}
The remainder part of this expansion $R_1(h\boldsymbol{z}) \leq \frac{M}{2} h^2\norm{\boldsymbol{z}}^2$ since $\partial_j f(\frac{t}{T}, \boldsymbol{x})$ are continuous for $\boldsymbol{x}\in S$, so $\sum_{j=1}^d \big| \partial_j f_{X_t(\frac{t}{T})} (\boldsymbol{x}) \big| \leq M < \infty$ for $\norm{\boldsymbol{x} - \boldsymbol{y}} \leq h\norm{\boldsymbol{z}}$, where $\boldsymbol{y} = (x^1-hz^1, \ldots, x^d - hz^d)$. That is, $R_1(h\boldsymbol{z})$ goes to zero as $h\rightarrow 0$. Also, using Assumption~\ref{Assumption: kernel functions}, $\int K_2(z^j) \diff z^j = 1$, $\int z^j K_2(z^j) \diff z^j = 0$, and $\int (z^j)^2 K_2(z^j) \diff z^j = \kappa$, so we have
\begin{align*}
    \lefteqn{\E\Big[\Big|\prod_{j=1}^d K_{h,2}\big(x^j - X_a^j\big(\frac{a}{T}\big) \big)\Big|\Big]}\\
    &= (-h)^d \int \cdots \int K_{2}(z^1) \cdots K_{2}(z^d) \Big\{ f(\frac{t}{T}, \boldsymbol{x}) + \sum_{j=1}^d \partial_j f(\frac{t}{T}, \boldsymbol{x}) (-h)z^j + R_1(h\boldsymbol{z}) \Big\}  \diff z^1 \cdots \diff z^d\\
    &\leq (-h)^d \int \cdots \int K_{2}(z^1) \cdots K_{2}(z^d) f(\frac{t}{T}, \boldsymbol{x}) \diff z^1 \cdots \diff z^d\\
    &\quad - (-1)^{d+1}h^{d+1} \int \cdots \int K_{2}(z^1) \cdots K_{2}(z^d) \sum_{j=1}^d \partial_j f(\frac{t}{T}, \boldsymbol{x}) z^j \diff z^1 \cdots \diff z^d\\
    &\quad + (-1)^d h^{d} \frac{M}{2} \int \cdots \int K_{2}(z^1) \cdots K_{2}(z^d) h^2 \norm{\boldsymbol{z}}^2 \diff z^1 \cdots \diff z^d\\
    &\leq (-h)^d f(\frac{t}{T}, \boldsymbol{x}) - (-1)^{d+1}h^{d+1} \Bigg\{ \partial_1 f(\frac{t}{T}, \boldsymbol{x}) \int \cdots \int K_{2}(z^2) \cdots K_{2}(z^d)\Big(  \int  z^1 K_2(z^1) \diff z^1 \Big) \diff z^2 \cdots \diff z^d\\
    &\quad + \cdots + \partial_d f(\frac{t}{T}, \boldsymbol{x}) \int \cdots \int K_{2}(z^1) \cdots K_{2}(z^{d-1}) \Big(  \int z^d K_2(z^d) \diff z^d \Big) \diff z^1 \cdots \diff z^{d-1} \Bigg\} \\
    &\quad + (-1)^d h^{d+2} \frac{M}{2} \Bigg\{ \int \cdots \int K_{2}(z^2) \cdots K_{2}(z^d) \Big(  \int  (z^1)^2 K_2(z^1) \diff z^1 \Big) \diff z^2 \cdots \diff z^d\\
    &\quad + \cdots + \int \cdots \int K_{2}(z^1) \cdots K_{2}(z^{d-1}) \Big(  \int (z^d)^2 K_2(z^d)  \diff z^d \Big) \diff z^1 \cdots \diff z^{d-1} \Bigg\}.
    \end{align*}
So
\begin{align}\label{eqn: E K2 Xu}
    \E\Big[\Big|\prod_{j=1}^d K_{h,2}\big(x^j - X_a^j\big(\frac{a}{T}\big) \big)\Big|\Big]
    &\leq (-h)^d f(\frac{t}{T}, \boldsymbol{x}) + (-1)^d h^{d+2} \frac{M}{2} \kappa d \nonumber\\
    &\leq h^d f(\frac{t}{T}, \boldsymbol{x}) + h^{d+2} \frac{M}{2} \kappa d.
\end{align}
Therefore,
\begin{align*}
    \E \Big[\Big|\prod_{j=1}^d K_{h,2}(x^j - X_{a,T}^j)\Big|\Big]
    &\leq \frac{L_2 C_U C_2^{d-\nu} d^\frac{3}{2} }{T^\nu h^\nu}  + h^d f(\frac{t}{T}, \boldsymbol{x}) + h^{d+2} \frac{M}{2} \kappa d.
\end{align*}

\noindent{\textit{(iii)}} Note that using Assumption \ref{assumption: CDF}, $\big| F_a^\star(y|\boldsymbol{X}_{a,T}) - F_t^\star(y|\boldsymbol{x}) \big| \leq L_{F^\star} \big( \norm{\boldsymbol{X}_{a,T} - \boldsymbol{x}} + \big| \frac{a}{T} - \frac{t}{T} \big|\big)$. Now see that
\begin{align*}
  K_{h,1}\big( &\frac{t}{T} - \frac{a}{T} \big) \E \Big[ \prod_{j=1}^d K_{h,2}(x^j - X_{a,T}^j) [\mathds{1}_{Y_{a,T}\leq y} - F_t^\star(y|\boldsymbol{x})] \Big]\\
    &\leq K_{h,1}\big( \frac{t}{T} - \frac{a}{T} \big) \E \Big[ \prod_{j=1}^d K_{h,2}(x^j - X_{a,T}^j) \E\Big[\big(\mathds{1}_{Y_{a,T}\leq y} - F_t^\star(y|\boldsymbol{x}) \big) \Big| \boldsymbol{X}_{a,T} \Big] \Big]\\
    &\leq K_{h,1}\big( \frac{t}{T} - \frac{a}{T} \big) \E \Big[ \prod_{j=1}^d K_{h,2}(x^j - X_{a,T}^j) \big| F_a^\star(y|\boldsymbol{X}_{a,T}) - F_t^\star(y|\boldsymbol{x}) \big| \Big]\\
    &\leq L_{F^\star} K_{h,1}\big( \frac{t}{T} - \frac{a}{T} \big)  \E \Big[ \prod_{j=1}^d K_{h,2}(x^j - X_{a,T}^j) \big( \norm{\boldsymbol{X}_{a,T} - \boldsymbol{x} } + \big| \frac{a}{T} - \frac{t}{T} \big|\big) \Big]\\
    &\leq L_{F^\star} K_{h,1}\big( \frac{t}{T} - \frac{a}{T} \big) \Big\{ \E \Big[ \prod_{j=1}^d K_{h,2}(x^j - X_{a,T}^j)  \norm{\boldsymbol{X}_{a,T} - \boldsymbol{x} } \Big]\\
    &\quad + \E \Big[ \prod_{j=1}^d K_{h,2}(x^j - X_{a,T}^j) \big| \frac{a}{T} - \frac{t}{T} \big| \Big] \Big\}.
\end{align*}
However,
\begin{align*}
    \prod_{j=1}^d K_{h,2}(x^j - X_{a,T}^j)  \norm{\boldsymbol{X}_{a,T} - \boldsymbol{x} }_2
    &= \prod_{j=1}^d K_{h,2}(x^j - X_{a,T}^j) \sqrt{\sum_{j=1}^d {|x^j - X_{a,T}^j|}^2}\\
    &\leq \prod_{j=1}^d K_{h,2}(x^j - X_{a,T}^j) \sqrt{d \max_{j} {|x^j - X_{a,T}^j|}^2}\\
    &\leq \sqrt{d} C_2 h \prod_{j=1}^d K_{h,2}(x^j - X_{a,T}^j),
\end{align*}
since using Assumption \ref{Assumption: kernel functions}, $|x^{j} - X_{a,T}^{j}| \leq C_2 h$ otherwise, $K_{h,2}(x^{j} - X_{a,T}^{j}) = 0$. Additionally, $\big| \frac{a}{T} - \frac{t}{T} \big| \leq C_1 h$ otherwise, $K_{h,1}\big(\big| \frac{a}{T} - \frac{t}{T} \big|\big) = 0$. Using \textit{(ii)}, we get 
\begin{align*}
    K_{h,1}\big(& \frac{t}{T} - \frac{a}{T} \big) \E \Big[ \prod_{j=1}^d K_{h,2}(x^j - X_{a,T}^j) [\mathds{1}_{Y_{a,T}\leq y} - F_t^\star(y|\boldsymbol{x})] \Big]\\
    &\leq L_{F^\star} K_{h,1}\big( \frac{t}{T} - \frac{a}{T} \big) \Big\{ \sqrt{d} C_2 h \E \Big[ \prod_{j=1}^d K_{h,2}(x^j - X_{a,T}^j) \Big] + C_1 h \E \Big[ \prod_{j=1}^d K_{h,2}(x^j - X_{a,T}^j) \Big] \Big\}\\
    &\leq (\sqrt{d}C_2 + C_1) L_{F^\star} h K_{h,1}\big( \frac{t}{T} - \frac{a}{T} \big) \E \Big[\prod_{j=1}^d K_{h,2}(x^j - X_{a,T}^j)\Big]\\
    &\leq (\sqrt{d}C_2 + C_1) L_{F^\star} K_{h,1}\big( \frac{t}{T} - \frac{a}{T} \big) \Big\{ \frac{L_2 C_U C_2^{d-\nu} d^\frac{3}{2} }{T^\nu h^{\nu-1}}  + h^{d+1} f(\frac{t}{T}, \boldsymbol{x}) + h^{d+3} \frac{M}{2} \kappa d \Big\}.
\end{align*}

\end{proof}

\begin{proposition}\label{lemma: J1 is Op(1)}
    Let Assumptions \ref{Assumption: X is lsp} - \ref{Assumption: bandwidth} hold, then
    \begin{align*}
        J_{t,T}^{-1}(\frac{t}{T}, \boldsymbol{x}) = \Big(\frac{1}{Th^{d+1}}\sum_{a=1}^T K_{h,1}\big(\frac{t}{T} - \frac{a}{T}\big) \prod_{j=1}^d K_{h,2}(x^j - X_{a,T}^j)\Big)^{-1} &= \bigO(1).
    \end{align*}
\end{proposition}

\begin{proof}

By applying Theorem 4.1 in \cite{vogt2012}, $$\Big|J_{t,T} (\frac{t}{T},\boldsymbol{x}) -\E\left[J_{t,T}(\frac{t}{T},\boldsymbol{x})\right]\Big| = \bigO_\P\Big(\sqrt{\frac{\log T}{Th^{d+1}}}\Big).$$ Additionally, using Assumption~\ref{Assumption: X is lsp}, $J_{t,T}(\frac{t}{T},\boldsymbol{x})$ can be decomposed as $$J_{t,T}(\frac{t}{T},\boldsymbol{x}) = \widetilde{J}_{t,T}(\frac{t}{T},\boldsymbol{x}) + \Bar{J}_{t,T}(\frac{t}{T},\boldsymbol{x}).$$ Then
\begin{align*}
    \Big|J_{t,T}(\frac{t}{T},\boldsymbol{x})\Big| &= \Big| J_{t,T}(\frac{t}{T},\boldsymbol{x}) - \E[J_{t,T}(\frac{t}{T},\boldsymbol{x})] + \E[J_{t,T}(\frac{t}{T},\boldsymbol{x})] \Big| \\
    &\leq \Big| J_{t,T}(\frac{t}{T},\boldsymbol{x}) - \E[J_{t,T}(\frac{t}{T},\boldsymbol{x})]\Big| + \Big| \E[J_{t,T}(\frac{t}{T},\boldsymbol{x})] \Big| \\
    &\leq \bigO_\P\Big(\sqrt{\frac{\log T}{Th^{d+1}}}\Big) + \Big|\E[J_{t,T}(\frac{t}{T},\boldsymbol{x})]\Big|\\
    &\leq \bigO_\P\Big(\sqrt{\frac{\log T}{Th^{d+1}}}\Big) + \Big|\E[\widetilde{J}_{t,T}(\frac{t}{T},\boldsymbol{x}) + \Bar{J}_{t,T}(\frac{t}{T},\boldsymbol{x})]\Big|\\
    &\leq \bigO_\P\Big(\sqrt{\frac{\log T}{Th^{d+1}}}\Big) + \Big|\E[\widetilde{J}_{t,T}(\frac{t}{T},\boldsymbol{x})]\Big| + \Big|\E[\Bar{J}_{t,T}(\frac{t}{T},\boldsymbol{x})]\Big|,
\end{align*}
where
\begin{align*}
    \widetilde{J}_{t,T}(\frac{t}{T},\boldsymbol{x}) 
    &= \frac{1}{Th^{d+1}}\sum_{a=1}^T K_{h,1}\big(\frac{t}{T} - \frac{a}{T}\big)\prod_{j=1}^dK_{h,2}\big(x^j-X_a^j\big(\frac{a}{T}\big)\big), 
\end{align*}
and
\begin{align*}
    \Bar{J}_{t,T}(\frac{t}{T},\boldsymbol{x}) 
    &= \frac{1}{Th^{d+1}}\sum_{a=1}^T K_{h,1}\big(\frac{t}{T} - \frac{a}{T}\big) \big\{ \prod_{j=1}^dK_{h,2}(x^j-X_{a,T}^j) - \prod_{j=1}^dK_{h,2}\big(x^j-X_a^j\big(\frac{a}{T}\big)\big)\big\}.
\end{align*}
Now, let us first observe $\Big|\E[\Bar{J}_{t,T}(\frac{t}{T},\boldsymbol{x})]\Big| $. Using Assumptions \ref{Assumption: X is lsp} and \ref{Assumption: kernel functions} together with Proposition \ref{Lemma: E of K2}.(\textit{i}), we have
\begin{align*}
    \Big| \E[\Bar{J}_{t,T}(\frac{t}{T},\boldsymbol{x})] \Big|
    &\leq \E\Big[\Big| \frac{1}{Th^{d+1}}\sum_{a=1}^T K_{h,1}\big(\frac{t}{T} - \frac{a}{T}\big) \Big\{ \prod_{j=1}^dK_{h,2}(x^j-X_{a,T}^j) - \prod_{j=1}^dK_{h,2}\big(x^j-X_a^j\big(\frac{a}{T}\big)\big)\Big\}\Big|\Big]\\
    &\leq \frac{1}{Th^{d+1}}\sum_{a=1}^T K_{h,1}\big(\frac{t}{T} - \frac{a}{T}\big) \E\Big[\Big\{ \Big|\prod_{j=1}^dK_{h,2}(x^j-X_{a,T}^j) - \prod_{j=1}^dK_{h,2}\big(x^j-X_a^j\big(\frac{a}{T}\big)\big)\Big\}\Big|\Big]\\
    &\leq \big(\frac{L_2 C_U C_2^{d-\nu} d^\frac{3}{2} }{T^\nu h^\nu}\big) \frac{1}{Th^{d+1}}\sum_{a=1}^T K_{h,1}\big(\frac{t}{T} - \frac{a}{T}\big).
\end{align*}
Using Lemma \ref{Lemma: sup K1 - g}, for $I_h = [C_1h, 1 - C_1h]$,
\begin{align}\label{eqn: O1 sum}
    \frac{1}{Th}\sum_{a=1}^T K_{h,1}\big(\frac{t}{T} - \frac{a}{T}\big) 
    &\leq \sup_{u\in I_h} \Big|\frac{1}{Th}\sum_{a=1}^T K_{h,1}\big(u - \frac{a}{T}\big) \Big|\nonumber\\
    &\leq \sup_{u\in I_h} \Big|\frac{1}{Th}\sum_{a=1}^T K_{h,1}\big(u - \frac{a}{T}\big) - 1 \Big| + 1 \nonumber\\
    &= \bigO\Big(\frac{1}{Th^2}\Big) + o(h) + 1
    = \bigO(1).
\end{align}
So
\begin{align*}
    \Big| \E[\Bar{J}_{t,T}(\frac{t}{T},\boldsymbol{x})] \Big|
    &\leq \big(\frac{L_2 C_U C_2^{d-\nu} d^\frac{3}{2} }{T^\nu h^{d+\nu}}\big) \underbrace{\frac{1}{Th}\sum_{a=1}^T K_{h,1}\big(\frac{t}{T} - \frac{a}{T}\big)}_{\bigO(1)}\\
    &\leq \frac{L_2 C C_U C_2^{d-\nu} d^\frac{3}{2} }{T^\nu h^{d+\nu}}
        \lesssim \frac{1}{T^\nu h^{d+\nu}},
\end{align*}
which converges to zero using Assumption \ref{Assumption: bandwidth}. 
On the other hand, using (\ref{eqn: E K2 Xu}), we get
\begin{align*}
    \Big|\E[\widetilde{J}_{t,T}(\frac{t}{T},\boldsymbol{x})]\Big|
    &\leq \E\Big[\Big|\frac{1}{Th^{d+1}}\sum_{a=1}^T K_{h,1}\big(\frac{t}{T} - \frac{a}{T}\big)\prod_{j=1}^dK_{h,2}\big(x^j-X_a^j\big(\frac{a}{T}\big)\big)\Big|\Big]\\
    &\leq \frac{1}{Th^{d+1}}\sum_{a=1}^T K_{h,1}\big(\frac{t}{T} - \frac{a}{T}\big)\E\Big[\Big|\prod_{j=1}^dK_{h,2}\big(x^j-X_a^j\big(\frac{a}{T}\big)\big)\Big|\Big]\\
    &\leq \frac{1}{Th^{d+1}} \sum_{a=1}^T K_{h,1}\big(\frac{t}{T} - \frac{a}{T}\big) \Big(h^d f(\frac{t}{T}, \boldsymbol{x}) + h^{d+2} \frac{M}{2} \kappa d \Big)\\
    &\leq \Big(f(\frac{t}{T}, \boldsymbol{x}) + h^{2} \frac{M}{2} \kappa d \Big) \underbrace{\frac{1}{Th} \sum_{a=1}^T K_{h,1}\big(\frac{t}{T} - \frac{a}{T}\big)}_{\bigO(1)}\\
    &\lesssim f(\frac{t}{T}, \boldsymbol{x}) + h^2,
\end{align*}
using (\ref{eqn: O1 sum}). Now, observe that $\big|\E[\widetilde{J}_{t,T}(\frac{t}{T},\boldsymbol{x})]\big| > 0$, since $\displaystyle f(\frac{t}{T}, \boldsymbol{x}) \geq \inf_{u\in [0,1], \boldsymbol{x}\in S} f(\frac{t}{T}, \boldsymbol{x})> 0$. Additionally, using Theorem 4.1 in \cite{vogt2012}, 
\begin{align*}
    J_{t,T}(\frac{t}{T},\boldsymbol{x})
    &\leq \big|J_{t,T}(\frac{t}{T},\boldsymbol{x}) - f(\frac{t}{T}, \boldsymbol{x}) \big| + f(\frac{t}{T}, \boldsymbol{x})\\
    &\leq \sup_{u\in [0,1], \boldsymbol{x}\in S} \big|J_{t,T}(u,\boldsymbol{x}) - f(\frac{t}{T}, \boldsymbol{x}) \big| +  f(\frac{t}{T}, \boldsymbol{x})\\
    &\leq o (1) +  f(\frac{t}{T}, \boldsymbol{x}).
\end{align*}
Hence
\begin{align*}
    \inf_{u\in [0,1], \boldsymbol{x}\in S} J_{t,T}(u,\boldsymbol{x})
    &\leq o (1) + \inf_{u\in [0,1], \boldsymbol{x}\in S} f(\frac{t}{T}, \boldsymbol{x}) > 0.
\end{align*}
Therefore, we have 
\begin{align*}
    \frac{1}{J_{t,T}(\frac{t}{T},\boldsymbol{x})} 
    &\leq \sup_{u\in [0,1], \boldsymbol{x}\in S} \frac{1}{J_{t,T}(u,\boldsymbol{x})}
    = \frac{1}{\inf_{u\in [0,1], \boldsymbol{x}\in S} J_{t,T}(u,\boldsymbol{x}) }
            = \bigO(1).
\end{align*}

\end{proof}

\begin{proposition}\label{prop: control of square of sums}
    Let Assumptions \ref{Assumption: X is lsp} - \ref{assumption: blocking} be satisfied. For $\boldsymbol{x}, y \in \R^{d+1}$, define
    \begin{align*}
        Z_{t,T}(y, \boldsymbol{x}) &= \frac{1}{Th^{d+1}}\sum_{a=1}^T K_{h,1}\big(\frac{t}{T} - \frac{a}{T}\big) \prod_{j=1}^d K_{h,2}(x^j - X_{a,T}^j) \big[\mathds{1}_{Y_{a,T}\leq y} - F^\star_{t}(y|\boldsymbol{x})\big].
    \end{align*}
    Then
    \begin{align*}
    \E\big[Z_{t,T}^2 (y, \boldsymbol{x})\big]
    =  \bigO\Big(\frac{1}{Th^{2(d + 1) +\frac{2}{p}(\nu - 1) }} + \frac{1}{T^{2\nu} h^{2(d+\nu-1)}} + h^2\Big),
    \end{align*}
    where $\nu=\rho \wedge 1$ and $p>2$.
\end{proposition}

\begin{proof}

Let
\begin{align}\label{eqn: Z tT}
    Z_{t,T}(y, \boldsymbol{x}) := \frac{1}{Th^{d+1}}\sum_{a=1}^T K_{h,1}\big(\frac{t}{T} - \frac{a}{T}\big) Z_{a,t,T}(y, \boldsymbol{x}),
\end{align}
where $$Z_{a,t,T}(y, \boldsymbol{x}) = \prod_{j=1}^d K_{h,2}(x^j - X_{a,T}^j) \big[\mathds{1}_{Y_{a,T}\leq y} - F^\star_{t}(y|\boldsymbol{x})\big].$$ 
Applying Bernstein's big-block and small-block procedure on $Z_{t,T}(y, \boldsymbol{x})$, we partition the set $\{1,\ldots,T\}$ into $2v_{T}+1$ independent subsets: $v_T$ big blocks of size $r_T$, $v_T$ small blocks of size $s_T$, and a remainder block of size $T-v_T(r_T+s_T)$, where $v_T = \lfloor\frac{T}{r_T + s_T}\rfloor$. To establish independence between the blocks, we need to place the asymptotically negligible small blocks in between two consecutive big blocks. This procedure was also used in \citep{FanMasry1992, Masry2005, Daisuke2022,MR4737023}. So, we decompose $Z_{t,T} (y, \boldsymbol{x})$ as
\begin{align}\label{eqn: bernstein blocking}
    Z_{t,T} (y, \boldsymbol{x}) &= \Lambda_{t,T}(y, \boldsymbol{x}) + \Pi_{t,T}(y, \boldsymbol{x}) + \Xi_{t,T}(y, \boldsymbol{x})\nonumber\\
    &:= \sum_{l=0}^{v_T-1} \Lambda_{l, t,T}(y, \boldsymbol{x}) + \sum_{l=0}^{v_T-1} \Pi_{l, t,T}(y, \boldsymbol{x}) + \Xi_{t,T}(y, \boldsymbol{x}),
\end{align}
where
\begin{align*}
    \Lambda_{l, t,T}(y, \boldsymbol{x}) &= \frac{1}{Th^{d+1}} \sum_{a = l(r_T+s_T)+1}^{l(r_T+s_T)+r_T} K_{h,1}\big(\frac{t}{T} - \frac{a}{T}\big) Z_{a,t,T}(y, \boldsymbol{x}),\\
    \Pi_{l, t,T}(y, \boldsymbol{x}) &= \frac{1}{Th^{d+1}} \sum_{a = l(r_T+s_T)+r_T+1}^{(l+1)(r_T+s_T)} K_{h,1}\big(\frac{t}{T} - \frac{a}{T}\big) Z_{a,t,T}(y, \boldsymbol{x}),
\end{align*}
and
\begin{align*}
    \Xi_{t,T}(y, \boldsymbol{x}) &= \frac{1}{Th^{d+1}} \sum_{a = v_T(r_T+s_T)+1}^T K_{h,1}\big(\frac{t}{T} - \frac{a}{T}\big) Z_{a,t,T}(y, \boldsymbol{x}).
\end{align*}
Let us define the size of the big blocks as $r_T = \lfloor\sqrt{Th^{d+1}}/q_T\rfloor$, where $q_T$ satisfies Assumption \ref{assumption: blocking}, i.e., $q_T = o(\sqrt{Th^{d+1}})$. This further implies that there exists a sequence of positive integers $\{q_T\}$, $q_T \rightarrow \infty$, such that $q_Ts_T = o\big(\sqrt{Th^{d+1}}\big)$. Additionally, as $T\rightarrow \infty$,
\begin{align}\label{eqn: blocking asymptotics}
    \frac{s_T}{r_T}\rightarrow 0, \qquad \text{and} \qquad \frac{r_T}{T}\rightarrow 0.
\end{align}
Note that defining $r_T = \lfloor\sqrt{Th^{d+1}}/q_T\rfloor$ immediately implies that $r_T = o\big(\sqrt{Th^{d+1}}\big)$. Additionally, note that $s_T = o(r_T)$ and $v_T = o(q_T\sqrt{Th^{d+1}})$.
Now,  
\begin{align*}
    \E\big[ Z_{t,T}^2 &(y, \boldsymbol{x})\big] 
    = \E\big[\Lambda_{t,T}^2(y, \boldsymbol{x})\big] + \E\big[\Pi_{t,T}^2(y, \boldsymbol{x})\big] + \E\big[\Xi_{t,T}^2(y, \boldsymbol{x})\big]\\
    & + 2 \Big\{ \E\big[\Lambda_{t,T}(y, \boldsymbol{x}) \Pi_{t,T}(y, \boldsymbol{x}) \big] + \E\big[\Lambda_{t,T}(y, \boldsymbol{x}) \Xi_{t,T}(y, \boldsymbol{x}) \big] + \E\big[\Pi_{t,T}(y, \boldsymbol{x}) \Xi_{t,T}(y, \boldsymbol{x}) \big] \Big\}.
\end{align*}
However, the defined size of big blocks and the relation (\ref{eqn: blocking asymptotics}) ensure that the blocks are asymptotically independent and the sums of small blocks and the remainder block are asymptotically negligible. Consequently, we can neglect the last terms in the previous equation. Hence, we have
\begin{align*}
    \E\big[ Z_{t,T}^2 (y, \boldsymbol{x})\big] 
    \approx \E\big[\Lambda_{t,T}^2(y, \boldsymbol{x})\big] + \E\big[\Pi_{t,T}^2(y, \boldsymbol{x})\big] + \E\big[\Xi_{t,T}^2(y, \boldsymbol{x})\big].
\end{align*}
For convenience of notation, in the succeeding steps, the dependency on $y$ and $\boldsymbol{x}$ is implicit.\\

\noindent \textcolor{magenta}{\underline{\it Step 1. Control of the big blocks.}} First, let us start by dealing with $\E\big[\Lambda_{t,T}^2\big]$. One has

\begin{align*}
    \E\big[\Lambda_{t,T}^2\big]
    &= \sum_{l=0}^{v_T-1}\E\big[\Lambda_{l,t,T}^2\big] + \sum_{\substack{l=0\\ \quad \mathrlap{l\neq l'}}}^{v_T-1}\sum_{l'=0}^{v_T-1} \E[\Lambda_{l,t,T}]\E[\Lambda_{l',t,T}] \\
    &= \frac{1}{(Th^{d+1})^2} \sum_{l=0}^{v_T-1} \E\Big[\Big(\sum_{a=l(r_T+s_T)+1}^{l(r_T+s_T)+r_T} K_{h,1}\big(\frac{t}{T} - \frac{a}{T}\big)Z_{a,t,T}\Big)^2\Big]\\
    &\qquad + \frac{1}{(Th^{d+1})^2} \sum_{\substack{l=0\\ \quad \mathrlap{l\neq l'}}}^{v_T-1}\sum_{l'=0}^{v_T-1} \sum_{a=l(r_T+s_T)+1}^{l(r_T+s_T)+r_T} \sum_{b=l'(r_T+s_T)+1}^{l'(r_T+s_T)+r_T} K_{h,1}\big(\frac{t}{T} - \frac{a}{T}\big) K_{h,1}\big(\frac{t}{T} - \frac{b}{T}\big) \E\big[Z_{a,t,T}Z_{b,t,T}\big]\\
        &= \frac{1}{(Th^{d+1})^2}\sum_{l=0}^{v_T-1} \sum_{a=l(r_T+s_T)+1}^{l(r_T+s_T)+r_T} K_{h,1}^2\big(\frac{t}{T} - \frac{a}{T}\big)\E\big[Z_{a,t,T}^2\big]\\
    &\qquad + \frac{1}{(Th^{d+1})^2} \sum_{l=0}^{v_T-1} \sum_{\substack{a=l(r_T+s_T)+1\\ \qquad \quad \mathrlap{|a-b|>0}}}^{l(r_T+s_T)+r_T} \sum_{b=l(r_T+s_T)+1}^{l(r_T+s_T)+r_T} K_{h,1}\big(\frac{t}{T} - \frac{a}{T}\big) K_{h,1}\big(\frac{t}{T} - \frac{b}{T}\big) \E\big[Z_{a,t,T}Z_{b,t,T}\big]\\
        &\qquad \quad + \frac{1}{(Th^{d+1})^2} \sum_{\substack{l=0\\ \quad \mathrlap{l\neq l'}}}^{v_T-1}\sum_{l'=0}^{v_T-1} \sum_{a=l(r_T+s_T)+1}^{l(r_T+s_T)+r_T} \sum_{b=l'(r_T+s_T)+1}^{l'(r_T+s_T)+r_T} K_{h,1}\big(\frac{t}{T} - \frac{a}{T}\big) K_{h,1}\big(\frac{t}{T} - \frac{b}{T}\big) \E\big[Z_{a,t,T}Z_{b,t,T}\big]\\
        &=: \mathlarger{\mathsf  S}_1^{\Lambda} + \mathlarger{\mathsf S}_2^{\Lambda} + \mathlarger{\mathsf S}_3^{\Lambda}.
\end{align*}
\noindent \textcolor{magenta}{\underline{\it Step 1.1. Control of $\mathlarger{\mathsf S}_1^{\Lambda}$.}} Considering $\mathlarger{\mathsf  S}_1^{\Lambda}$, we have
\begin{align*}
    \mathlarger{\mathsf S}_1^{\Lambda} &= \frac{1}{(Th^{d+1})^2}\sum_{l=0}^{v_T-1} \sum_{a=l(r_T+s_T)+1}^{l(r_T+s_T)+r_T} K_{h,1}^2\big(\frac{t}{T} - \frac{a}{T}\big)\E\big[Z_{a,t,T}^2\big]\nonumber\\
    &= \frac{1}{(Th^{d+1})^2}\sum_{l=0}^{v_T-1} \sum_{a=l(r_T+s_T)+1}^{l(r_T+s_T)+r_T} K_{h,1}^2\big(\frac{t}{T} - \frac{a}{T}\big) \E\Big[  \prod_{j=1}^d K_{h,2}^2(x^j - X_{a,T}^j)(\mathds{1}_{Y_{a,T}\leq y} - F_t^\star (y|\boldsymbol{x}))^2 \Big].
\end{align*}
Now observe that
\begin{align*}
    \lefteqn{K_{h,1}\big(\frac{t}{T} - \frac{a}{T}\big)\E\Big[\prod_{j=1}^d K_{h,2}^2(x^j - X_{a,T}^j)(\mathds{1}_{Y_{a,T}\leq y} - F_t^\star (y|\boldsymbol{x}))^2 \Big]}\nonumber\\
    &\leq 2 C_2^d K_{h,1}\big(\frac{t}{T} - \frac{a}{T}\big)\E\Big[\prod_{j=1}^d K_{h,2} (x^j - X_{a,T}^j) \big| \mathds{1}_{Y_{a,T}\leq y} - F_t^\star (y|\boldsymbol{x})\big| \Big].
\end{align*}
By Proposition  \ref{Lemma: E of K2}.\textit{(iii)},
\begin{align*}    \lefteqn{K_{h,1}\big(\frac{t}{T} - \frac{a}{T}\big)\E\Big[\prod_{j=1}^d K_{h,2}^2(x^j - X_{a,T}^j)(\mathds{1}_{Y_{a,T}\leq y} - F_t^\star (y|\boldsymbol{x}))^2 \Big]}\nonumber\\
    &\leq 2 C_2^d K_{h,1}\big(\frac{t}{T} - \frac{a}{T}\big)\E\Big[\prod_{j=1}^d K_{h,2} (x^j - X_{a,T}^j) \big| \mathds{1}_{Y_{a,T}\leq y} - F_t^\star (y|\boldsymbol{x})\big| \Big]\\
    &\leq 2 C_2^d (\sqrt{d}C_2 + C_1) L_{F^\star} K_{h,1}\big( \frac{t}{T} - \frac{a}{T} \big) \Big( \frac{L_2 C_U C_2^{d-\nu} d^\frac{3}{2} }{T^\nu h^{\nu-1}}  + h^{d+1} f(\frac{t}{T}, \boldsymbol{x}) + h^{d+3} \frac{M}{2} \kappa d \Big)\nonumber\\
    &\lesssim K_{h,1}\big( \frac{t}{T} - \frac{a}{T} \big) \Big( \frac{1}{T^\nu h^{\nu-1}}   + h^{d+1} + h^{d+3}  \Big).
\end{align*}
Thus
\begin{align}\label{eqn: big block 1}
    \mathlarger{\mathsf S}_1^{\Lambda} 
    &\lesssim \frac{1}{T^2h^{2d+2}} \Big(\frac{1}{T^\nu h^{\nu-1}}   + h^{d+1} + h^{d+3}\Big) \sum_{l=0}^{v_T-1} \sum_{a=l(r_T+s_T)+1}^{l(r_T+s_T)+r_T} K_{h,1}^2\big(\frac{t}{T} - \frac{a}{T}\big) \nonumber\\
    &\leq \frac{C_1}{Th^{2d+1}} \Big(\frac{1}{T^\nu h^{\nu-1}}   + h^{d+1} + h^{d+3}\Big) \underbrace{\frac{1}{Th}\sum_{a=1}^T K_{h,1}\big(\frac{t}{T} - \frac{a}{T}\big)}_{\bigO(1)} \nonumber\\
    &\lesssim \frac{1}{Th^{2d+1}} \Big(\frac{1}{T^\nu h^{\nu-1}}   + h^{d+1} + h^{d+3}\Big)\nonumber\\
    &\lesssim \frac{1}{T^{1+\nu} h^{2d+\nu}} + \frac{1}{T h^{d}}\nonumber\\
    &\lesssim \frac{1}{T h^{2d+\nu}}.
\end{align}
\noindent\textcolor{magenta}{ \underline{\it Step 1.2. Control of $\mathlarger{\mathsf S}_2^{\Lambda}$.}}
On the other hand,
\begin{align*}
    \mathlarger{\mathsf S}_2^{\Lambda} &= \frac{1}{(Th^{d+1})^2} \sum_{l=0}^{v_T-1} \sum_{\substack{a=l(r_T+s_T)+1\\ \qquad \quad \mathrlap{|a-b|>0}}}^{l(r_T+s_T)+r_T} \sum_{b=l(r_T+s_T)+1}^{l(r_T+s_T)+r_T} K_{h,1}\big(\frac{t}{T} - \frac{a}{T}\big) K_{h,1}\big(\frac{t}{T} - \frac{b}{T}\big) \E\big[Z_{a,t,T} Z_{b,t,T}\big]\\
    &= \frac{1}{(Th^{d+1})^2} \sum_{l=0}^{v_T-1} \sum_{\substack{a=l(r_T+s_T)+1\\ \qquad \quad \mathrlap{|a-b|>0}}}^{l(r_T+s_T)+r_T} \sum_{b=l(r_T+s_T)+1}^{l(r_T+s_T)+r_T} K_{h,1}\big(\frac{t}{T} - \frac{a}{T}\big) K_{h,1}\big(\frac{t}{T} - \frac{b}{T}\big) \text{$\C$ov}\big(Z_{a,t,T}, Z_{b,t,T}\big)\\
    &\qquad \quad + \frac{1}{(Th^{d+1})^2} \sum_{l=0}^{v_T-1} \sum_{\substack{a=l(r_T+s_T)+1\\ \qquad \quad \mathrlap{|a-b|>0}}}^{l(r_T+s_T)+r_T} \sum_{b=l(r_T+s_T)+1}^{l(r_T+s_T)+r_T} K_{h,1}\big(\frac{t}{T} - \frac{a}{T}\big)\\
    &\hspace{8cm} \times K_{h,1}\big(\frac{t}{T} - \frac{b}{T}\big) \E\big[Z_{a,t,T}\big] \E\big[Z_{b,t,T}\big]\\
    &:= \mathlarger{\mathsf S}_{21}^{\Lambda} + \mathlarger{\mathsf S}_{22}^{\Lambda}.
\end{align*}
\noindent \textcolor{magenta}{\underline{\it Step 1.2.1. Control of $\mathlarger{\mathsf S}_{21}^{\Lambda}$.}} Looking at $\mathlarger{\mathsf S}_{21}^{\Lambda}$, we have
\begin{align*}
    \mathlarger{\mathsf S}_{21}^{\Lambda} &= \frac{1}{(Th^{d+1})^2} \sum_{l=0}^{v_T-1} \sum_{\substack{a=l(r_T+s_T)+1\\ \qquad \quad \mathrlap{|a-b|>0}}}^{l(r_T+s_T)+r_T} \sum_{b=l(r_T+s_T)+1}^{l(r_T+s_T)+r_T} K_{h,1}\big(\frac{t}{T} - \frac{a}{T}\big)  K_{h,1}\big(\frac{t}{T} - \frac{b}{T}\big) \text{$\C$ov}\big(Z_{a,t,T}, Z_{b,t,T}\big)\\
    &= \frac{1}{(Th^{d+1})^2} \sum_{l=0}^{v_T-1} \sum_{\substack{n_1=1 \\  \mathrlap{|n_1-n_2|>0}}}^{r_T} \sum_{n_2=1}^{r_T} K_{h,1}\big(\frac{t}{T} - \frac{\lambda + n_1}{T}\big)K_{h,1}\big(\frac{t}{T} - \frac{\lambda + n_2}{T}\big)  \text{$\C$ov}\big(Z_{\lambda + n_1,t,T}, Z_{\lambda + n_2,t,T}\big) \\
    &\leq \frac{1}{(Th^{d+1})^2} \sum_{l=0}^{v_T-1} \sum_{\substack{n_1=1 \\  \mathrlap{|n_1-n_2|>0}}}^{r_T} \sum_{n_2=1}^{r_T} K_{h,1}\big(\frac{t}{T} - \frac{\lambda + n_1}{T}\big)K_{h,1}\big(\frac{t}{T} - \frac{\lambda + n_2}{T}\big)  \big|\text{$\C$ov}\big(Z_{\lambda + n_1,t,T}, Z_{\lambda + n_2,t,T}\big)\big|, 
\end{align*}
where $\lambda = l(r_T+s_T)$. Note that $\{\boldsymbol{X}_{t,T},\varepsilon_{t,T}\}$ is regularly mixing (Assumption \ref{Assumption: mixing}), using Davydov's inequality (Lemma \ref{lemma: Davydovs}), for $p>2$ and by Lemma \ref{lemma: beta l_l'}, $\beta(\sigma(\boldsymbol{X}_{\lambda+n_1,t,T}), \sigma(\boldsymbol{X}_{\lambda+n_2,t,T})) \leq \beta(|n_1 - n_2|)$, we get
\begin{align*}    \lefteqn{K_{h,1}\big(\frac{t}{T} - \frac{\lambda+n_1}{T}\big) K_{h,1}\big(\frac{t}{T} - \frac{\lambda+n_2}{T}\big) \Big|\text{$\C$ov}\big(Z_{\lambda+n_1,t,T}, Z_{\lambda+n_2,t,T}\big)\Big|}\nonumber\\
    &\leq 8 K_{h,1}\big(\frac{t}{T} - \frac{\lambda+n_1}{T}\big) K_{h,1}\big(\frac{t}{T} - \frac{\lambda+n_2}{T}\big) \big\| Z_{\lambda+n_1,t,T}\big\|_{L_p} \big\| Z_{\lambda+n_2,t,T}\big\|_{L_p} \beta(\sigma(\boldsymbol{X}_{\lambda+n_1,t,T}), \sigma(\boldsymbol{X}_{\lambda+n_2,t,T}))^{1-\frac{2}{p}}\nonumber\\ 
        &\leq 8 K_{h,1}\big(\frac{t}{T} - \frac{\lambda+n_1}{T}\big) K_{h,1}\big(\frac{t}{T} - \frac{\lambda+n_2}{T}\big)\nonumber\\
    &\quad \times \Big(\E\Big[\Big| \prod_{j=1}^d K_{h,2}(x^j - X_{\lambda+n_1,T}^j)(\mathds{1}_{Y_{\lambda+n_1,T}\leq y} - F_t^\star (y|\boldsymbol{x}))\Big|^p\Big]\Big)^\frac{1}{p}\nonumber\\
    &\quad \times \Big(\E\Big[\Big|\prod_{j=1}^d K_{h,2}(x^j - X_{\lambda+n_2,T}^j)(\mathds{1}_{Y_{\lambda+n_2,T}\leq y} - F_t^\star (y|\boldsymbol{x}))\Big|^p\Big]\Big)^\frac{1}{p} \beta(|n_1 - n_2|)^{1-\frac{2}{p}}.
\end{align*}
Using Proposition \ref{Lemma: E of K2}.(\textit{iii}), 
\begin{align}\label{eqn: cov within blocks}
    \lefteqn{K_{h,1}\big(\frac{t}{T} - \frac{\lambda+n_1}{T}\big) K_{h,1}\big(\frac{t}{T} - \frac{\lambda+n_2}{T}\big) \Big|\text{$\C$ov}\big(Z_{\lambda+n_1,t,T}, Z_{\lambda+n_2,t,T}\big)\Big|}\nonumber\\
    &\lesssim K_{h,1}\big(\frac{t}{T} - \frac{\lambda+n_1}{T}\big) \Big(\frac{1}{T^\nu h^{\nu-1}}   + h^{d+1} + h^{d+3}\Big)^\frac{1}{p}\nonumber\\
    &\quad \times K_{h,1}\big(\frac{t}{T} - \frac{\lambda+n_2}{T}\big) \Big(\frac{1}{T^\nu h^{\nu-1}}   + h^{d+1} + h^{d+3}\Big)^\frac{1}{p} \beta(|n_1 - n_2|)^{1-\frac{2}{p}}\nonumber\\
    &\lesssim K_{h,1}\big(\frac{t}{T} - \frac{\lambda+n_1}{T}\big) K_{h,1}\big(\frac{t}{T} - \frac{\lambda+n_2}{T}\big) \Big(\frac{1}{T^\nu h^{\nu-1}}   + h^{d+1} + h^{d+3}\Big)^{\frac{2}{p}} \beta(|n_1 - n_2|)^{1-\frac{2}{p}}.\nonumber\\
\end{align}
In consequence,
\begin{align*}
    \mathlarger{\mathsf S}_{21}^{\Lambda} &\lesssim \frac{1}{T^2 h^{2d+2}} \Big(\frac{1}{T^\nu h^{\nu-1}}   + h^{d+1} + h^{d+3}\Big)^{\frac{2}{p}}  \sum_{l=0}^{v_T-1} \sum_{\substack{n_1=1 \\  \mathrlap{|n_1-n_2|>0}}}^{r_T} \sum_{n_2=1}^{r_T}  K_{h,1}\big(\frac{t}{T} - \frac{\lambda + n_1}{T}\big)\nonumber\\
    &\quad \times K_{h,1}\big(\frac{t}{T} - \frac{\lambda + n_2}{T}\big)  \beta(|n_1-n_2|)^{1-\frac{2}{p}}\nonumber\\
    &\leq \frac{C_1^2}{T^2h^{2d+2}} \Big(\frac{1}{T^\nu h^{\nu-1}}   + h^{d+1} + h^{d+3}\Big)^{\frac{2}{p}}  \sum_{l=0}^{v_T-1} \sum_{\substack{n_1=1 \\  \mathrlap{|n_1-n_2|>0}}}^{r_T} \sum_{n_2=1}^{r_T} \beta(|n_1-n_2|)^{1-\frac{2}{p}}.
\end{align*}
Using Assumption \ref{Assumption: mixing}, $\sum_{k=1}^{\infty} k^\zeta \beta(k)^{1-\frac{2}{p}}<\infty$, which can be expressed as $\sum_{k=1}^{r_T} k^\zeta \beta(k)^{1-\frac{2}{p}} + \sum_{k=r_T+1}^{\infty} k^\zeta \beta(k)^{1-\frac{2}{p}}$. Now, observe that letting $k = |n_1 -n_2|$ yields
\begin{align*}
    \sum_{\substack{n_1=1 \\  \mathrlap{|n_1-n_2|>0}}}^{r_T} \sum_{n_2=1}^{r_T} \beta(|n_1 - n_2|)^{1-\frac{2}{p}}
    &= \sum_{n_1=1}^{r_T} \Big( \sum_{n_2>n_1}^{r_T} \beta(n_2 - n_1)^{1-\frac{2}{p}} + \sum_{n_2<n_1}^{r_T} \beta(n_1 - n_2)^{1-\frac{2}{p}}\Big)\\
    &= \sum_{n_1=1}^{r_T}\sum_{k>0}^{r_T-n_1} \beta(k)^{1 - \frac{2}{p}} + \sum_{n_2=1}^{r_T}\sum_{k>0}^{r_T-n_2} \beta(k)^{1 - \frac{2}{p}}\\
    &= 2 \sum_{n=1}^{r_T}\sum_{k>0}^{r_T-n} \beta(k)^{1 - \frac{2}{p}}
    \leq 2r_T \sum_{k=1}^{r_T} \beta(k)^{1 - \frac{2}{p}}\\
    &\lesssim  r_T \sum_{k=1}^{r_T} k^\zeta \beta(k)^{1 - \frac{2}{p}}\\
    &\leq r_T \sum_{k=1}^{\infty} k^\zeta \beta(k)^{1 - \frac{2}{p}},
\end{align*}
since $k^\zeta \geq 1$ for $\zeta > 1 - \frac{2}{p}$, where $p>2$. Hence
\begin{align}\label{eqn: big block 2_1}
    \mathlarger{\mathsf S}_{21}^{\Lambda} 
    &\leq \frac{C_1^2 r_T}{T^2h^{2d+2}} \Big(\frac{1}{T^\nu h^{\nu-1}}   + h^{d+1} + h^{d+3}\Big)^{\frac{2}{p}}  \sum_{l=0}^{v_T-1}\sum_{k=1}^{\infty} k^\zeta \beta(k)^{1-\frac{2}{p}} \nonumber\\
    &\lesssim \frac{v_Tr_T}{T^2h^{2d+2}} \Big(\frac{1}{T^\nu h^{\nu-1}}   + h^{d+1} + h^{d+3}\Big)^{\frac{2}{p}} \sum_{k=1}^{\infty} k^\zeta \beta(k)^{1-\frac{2}{p}} \nonumber\\
    &\lesssim \frac{1}{Th^{2d+2}} \Big(\frac{1}{T^\nu h^{\nu-1}}   + h^{d+1} + h^{d+3}\Big)^{\frac{2}{p}}, \quad \text{since } v_Tr_T \leq \frac{T}{r_T}r_T =T, \nonumber\\
    &= \Big(\frac{1}{T^{p}h^{2(d+1)p}} \Big(\frac{1}{T^\nu h^{\nu-1}}   + h^{d+1} + h^{d+3}\Big)^2 \Big)^\frac{1}{p}\nonumber\\
        &\lesssim \Big(\frac{1}{T^{p+2\nu} h^{2(d + 1)p + 2(\nu - 1)}} + \frac{1}{T^p h^{2(d + 1)p - 2(d + 1)}}\Big)^\frac{1}{p}\nonumber\\
    &\lesssim \Big(\frac{1}{T^{p} h^{2(d + 1)p + 2(\nu - 1)}}\Big)^\frac{1}{p}\nonumber\\
    &\lesssim \frac{1}{T h^{2(d + 1) - \frac{2}{p}(1 - \nu) }}.
\end{align}
\noindent {\textcolor{magenta}{\underline{\it Step 1.2.2. Control of $\mathlarger{\mathsf S}_{22}^{\Lambda}$.}}}
Considering $\mathlarger{\mathsf S}_{22}^{\Lambda}$, see that
\begin{align*}
    \mathlarger{\mathsf S}_{22}^{\Lambda} &= \frac{1}{(Th^{d+1})^2} \sum_{l=0}^{v_T-1} \sum_{\substack{a=l(r_T+s_T)+1\\ \qquad \quad \mathrlap{|a-b|>0}}}^{l(r_T+s_T)+r_T} \sum_{b=l(r_T+s_T)+1}^{l(r_T+s_T)+r_T} K_{h,1}\big(\frac{t}{T} - \frac{a}{T}\big) K_{h,1}\big(\frac{t}{T} - \frac{b}{T}\big)  \E\big[Z_{a,t,T}\big] \E\big[Z_{b,t,T}\big]\nonumber\\
    &= \frac{1}{(Th^{d+1})^2} \sum_{l=0}^{k_T-1} \sum_{\substack{n_1=1 \\  \mathrlap{|n_1-n_2|>0}}}^{r_T} \sum_{n_2=1}^{r_T} K_{h,1}\big(\frac{t}{T} - \frac{\lambda + n_1}{T}\big)K_{h,1}\big(\frac{t}{T} - \frac{\lambda + n_2}{T}\big)  \E\big[Z_{\lambda + n_1,t,T}\big]\E\big[Z_{\lambda + n_2,t,T}\big]\nonumber\\
    &= \frac{1}{(Th^{d+1})^2} \sum_{l=0}^{v_T-1} \sum_{\substack{n_1=1 \\  \mathrlap{|n_1-n_2|>0}}}^{r_T} \sum_{n_2=1}^{r_T} K_{h,1}\big(\frac{t}{T} - \frac{\lambda + n_1}{T}\big)K_{h,1}\big(\frac{t}{T} - \frac{\lambda + n_2}{T}\big)\nonumber\\
    &\quad \times \E\Big[\prod_{j=1}^d K_{h,2}(x^j - X_{\lambda+n_1,T}^j)(\mathds{1}_{Y_{\lambda+n_1,T}\leq y} - F_t^\star (y|\boldsymbol{x}))\Big]\nonumber\\
    &\quad \times \E\Big[\prod_{j=1}^d K_{h,2}(x^j - X_{\lambda+n_2,T}^j)(\mathds{1}_{Y_{\lambda+n_2,T}\leq y} - F_t^\star (y|\boldsymbol{x}))\Big].
\end{align*}
By Proposition \ref{Lemma: E of K2}.(\textit{iii}), for $i=1,2$, 
\begin{eqnarray*}
    \lefteqn{K_{h,1}\big(\frac{t}{T} - \frac{\lambda + n_i}{T}\big)\E\big[\prod_{j=1}^d K_{h,2}(x^j - X_{\lambda+n_i,T}^j)(\mathds{1}_{Y_{\lambda + n_i,T}\leq y} - F_t^\star (y|\boldsymbol{x}))\big]}\\&& \lesssim K_{h,1}\big(\frac{t}{T} - \frac{\lambda + n_i}{T}\big) \big( \frac{1}{T^\nu h^{\nu-1}}  + h^{d+1} + h^{d+3} \big),
\end{eqnarray*}
then
\begin{align}\label{eqn: big block 2_2}
    \mathlarger{\mathsf S}_{22}^{\Lambda} 
    &\lesssim \frac{1}{T^2 h^{2d+2}} \Big(\frac{1}{T^\nu h^{\nu-1}}  + h^{d+1} + h^{d+3}\Big)^2  \sum_{l=0}^{v_T-1} \sum_{\substack{n_1=1 \\  \mathrlap{|n_1-n_2|>0}}}^{r_T} \sum_{n_2=1}^{r_T} K_{h,1}\big(\frac{t}{T} - \frac{\lambda + n_1}{T}\big) K_{h,1}\big(\frac{t}{T} - \frac{\lambda + n_2}{T}\big)\nonumber\\
    &\leq \frac{C_1}{T h^{2d+1}}  \Big(\frac{1}{T^\nu h^{\nu-1}}  + h^{d+1} + h^{d+3}\Big)^2 \underbrace{\frac{1}{Th} \sum_{a=1}^{T} K_{h,1}\big(\frac{t}{T} - \frac{a}{T}\big)}_{\bigO(1)} \nonumber\\
    &\lesssim  \frac{1}{Th^{2d+1}} \Big(\frac{1}{T^\nu h^{\nu-1}}  + h^{d+1} + h^{d+3}\Big)^2\nonumber\\
    &\lesssim \frac{1}{Th^{2d+1}} \Big( \frac{1}{T^{2\nu} h^{2(\nu - 1)}}  + h^{2(d+1)} \Big)\nonumber\\
    &\lesssim \frac{1}{T^{1+ 2\nu} h^{2(d+\nu)-1}} + \frac{h}{T}\nonumber\\ 
    &\lesssim  \frac{1}{T h^{2(d+\nu)-1}}.
\end{align}
\noindent {\textcolor{magenta}{\underline{\it Step 1.3 Control of $\mathlarger{\mathsf S}_{3}^{\Lambda}$.}}}
Now, let us examine $\mathlarger{\mathsf S}_{3}^{\Lambda}$. Observe that
\begin{align*}
    \mathlarger{\mathsf S}_{3}^{\Lambda} &= \frac{1}{(Th^{d+1})^2} \sum_{\substack{l=0\\ \quad \mathrlap{l\neq l'}}}^{v_T-1}\sum_{l'=0}^{v_T-1} \sum_{a=l(r_T+s_T)+1}^{l(r_T+s_T)+r_T} \sum_{b=l'(r_T+s_T)+1}^{l'(r_T+s_T)+r_T} K_{h,1}\big(\frac{t}{T} - \frac{a}{T}\big) K_{h,1}\big(\frac{t}{T} - \frac{b}{T}\big) \E[Z_{a,t,T}Z_{b,t,T}]\\
    &= \frac{1}{(Th^{d+1})^2} \sum_{\substack{l=0\\ \quad \mathrlap{l\neq l'}}}^{v_T-1}\sum_{l'=0}^{v_T-1} \sum_{a=l(r_T+s_T)+1}^{l(r_T+s_T)+r_T} \sum_{b=l'(r_T+s_T)+1}^{l'(r_T+s_T)+r_T} K_{h,1}\big(\frac{t}{T} - \frac{a}{T}\big)  K_{h,1}\big(\frac{t}{T} - \frac{b}{T}\big)\\
    &\hspace{8cm}\times \text{$\C$ov}\big(Z_{a,t,T},Z_{b,t,T}\big)\\
    &\qquad + \frac{1}{(Th^{d+1})^2} \sum_{\substack{l=0\\ \quad \mathrlap{l\neq l'}}}^{v_T-1}\sum_{l'=0}^{v_T-1} \sum_{a=l(r_T+s_T)+1}^{l(r_T+s_T)+r_T} \sum_{b=l'(r_T+s_T)+1}^{l'(r_T+s_T)+r_T} K_{h,1}\big(\frac{t}{T} - \frac{a}{T}\big) K_{h,1}\big(\frac{t}{T} - \frac{b}{T}\big)\\
    &\hspace{8cm}\times  \E[Z_{a,t,T}] \E[Z_{b,t,T}]\\
    &=: \mathlarger{\mathsf S}_{31}^{\Lambda} + \mathlarger{\mathsf S}_{32}^{\Lambda}.
\end{align*}
\noindent {\textcolor{magenta}{\underline{\it Step 1.3.1 Control of $\mathlarger{\mathsf S}_{31}^{\Lambda}$.}}} Looking at $\mathlarger{\mathsf S}_{31}^{\Lambda}$, we have
\begin{align*}
    \mathlarger{\mathsf S}_{31}^{\Lambda} &= \frac{1}{(Th^{d+1})^2} \sum_{\substack{l=0\\ \quad \mathrlap{l\neq l'}}}^{v_T-1}\sum_{l'=0}^{v_T-1} \sum_{a=l(r_T+s_T)+1}^{l(r_T+s_T)+r_T} \sum_{b=l'(r_T+s_T)+1}^{l'(r_T+s_T)+r_T} K_{h,1}\big(\frac{t}{T} - \frac{a}{T}\big)  K_{h,1}\big(\frac{t}{T} - \frac{b}{T}\big) \text{$\C$ov}\big(Z_{a,t,T},Z_{b,t,T}\big)\\
        &= \frac{1}{(Th^{d+1})^2} \sum_{\substack{l=0 \\ \quad \mathrlap{l\neq l'}}}^{v_T-1}\sum_{l'=0}^{v_T-1}\sum_{n_1 =1}^{r_T}\sum_{n_2=1}^{r_T} K_{h,1}\big(\frac{t}{T} - \frac{\lambda + n_1}{T}\big)K_{h,1}\big(\frac{t}{T} - \frac{\lambda'+n_2}{T}\big)\\
    &\hspace{8cm} \times \text{$\C$ov}\big(Z_{\lambda + n_1,t,T},Z_{\lambda'+n_2,t,T}\big),
\end{align*}
where $\lambda = l(r_T+s_T)$ and $\lambda' = l'(r_T+s_T)$, however, for $l\neq l'$, see that
\begin{align*}
    |\lambda - \lambda' + n_1 - n_2| &\geq | l(r_T+s_T) - l'(r_T+s_T) + n_1 - n_2 |\\
    &\geq |(l-l')(r_T+s_T) + n_1 - n_2|\\
        &> s_T,
\end{align*}
since $n_1,n_2\in \{1, \ldots, r_T\}$. So if we let $m = \lambda+n_1$ and $m'=\lambda'+n_2$, we have
\begin{align*}
    \mathlarger{\mathsf S}_{31}^{\Lambda} &= \frac{1}{(Th^{d+1})^2} \sum_{\substack{m=1 \\ \qquad \mathrlap{|m-m'|> s_T}}}^{v_T(r_T+s_T)-s_T}\sum_{m'= 1}^{v_T(r_T+s_T)-s_T} K_{h,1}\big(\frac{t}{T} - \frac{m}{T}\big) K_{h,1}\big(\frac{t}{T} - \frac{m'}{T}\big)  \text{$\C$ov}\big(Z_{m,t,T},Z_{m',t,T}\big)\\
    &\leq \frac{1}{(Th^{d+1})^2} \sum_{\substack{m=1 \\  \mathrlap{|m-m'|> s_T}}}^{T}\sum_{m'=1}^{T} K_{h,1}\big(\frac{t}{T} - \frac{m}{T}\big) K_{h,1}\big(\frac{t}{T} - \frac{m'}{T}\big) \big|\text{$\C$ov}\big(Z_{m,t,T},Z_{m',t,T}\big)\big|,
\end{align*}
Now, using (\ref{eqn: cov within blocks}), we have
\begin{align*}
    \lefteqn{ K_{h,1}\big(\frac{t}{T} - \frac{m}{T}\big) K_{h,1}\big(\frac{t}{T} - \frac{m'}{T}\big) \big|\text{$\C$ov}\big(Z_{m,t,T},Z_{m',t,T}\big)\big|}\\
    &\lesssim K_{h,1}\big(\frac{t}{T} - \frac{m}{T}\big) K_{h,1}\big(\frac{t}{T} - \frac{m'}{T}\big) \big(\frac{1}{T^\nu h^{\nu-1}}   + h^{d+1} + h^{d+3}\Big)^{\frac{2}{p}} \beta(|m - m'|)^{1-\frac{2}{p}}.
\end{align*}
Thus
\begin{align*}
    \mathlarger{\mathsf S}_{31}^{\Lambda} &\lesssim \frac{1}{T^2h^{2d+2}} \Big(\frac{1}{T^\nu h^{\nu-1}}   + h^{d+1} + h^{d+3}\Big)^{\frac{2}{p}} \sum_{\substack{m=1 \\ \mathrlap{|m-m'|> s_T}}}^{T}\sum_{m'=1}^{T} K_{h,1}\big(\frac{t}{T} - \frac{m}{T}\big) K_{h,1}\big(\frac{t}{T} - \frac{m'}{T}\big) \beta(|m -m'|)^{1-\frac{2}{p}}\nonumber\\
    &\leq\frac{C_1^2}{T^2h^{2d+2}} \Big(\frac{1}{T^\nu h^{\nu-1}}   + h^{d+1} + h^{d+3} \Big)^{\frac{2}{p}} \sum_{\substack{m=1 \\ \mathrlap{|m-m'|> s_T}}}^{T}\sum_{m'=1}^{T} \beta(|m - m'|)^{1-\frac{2}{p}}.
\end{align*}
Using Assumption \ref{Assumption: mixing}, $\sum_{k=1}^{\infty} k^\zeta \beta(k)^{1-\frac{2}{p}}<\infty$. Now, observe that letting $k = |m - m'|$ yields
\begin{align*}
    \sum_{\substack{m=1 \\  \mathrlap{|m-m'|> s_T}}}^{T} \sum_{m'=1}^{T} \beta(|m - m'|)^{1-\frac{2}{p}}
    &\leq C \sum_{k = s_T+1}^{T} \beta(k)^{1-\frac{2}{p}}
    \lesssim \frac{1}{k^\zeta} \sum_{k = s_T+1}^{T} k^\zeta \beta(k)^{1-\frac{2}{p}}\\
    &\leq  \frac{1}{s_T^\zeta} \sum_{k = s_T+1}^{T} k^\zeta \beta(k)^{1-\frac{2}{p}}, \quad \text{since } k > s_T,\\
    &\leq \frac{1}{s_T^\zeta} \sum_{k = s_T+1}^{\infty} k^\zeta \beta(k)^{1-\frac{2}{p}},
\end{align*}
since $\beta(k)\geq 0$ and $\big(\frac{k}{s_T}\big)^\zeta \geq 1$ for $\zeta > 1 - \frac{2}{p}$, where $p>2$. So
\begin{align}\label{eqn: big block 3_1}
    \mathlarger{\mathsf S}_{31}^{\Lambda} 
    &\leq \frac{C_1^2}{s_T^\zeta T^2h^{2(d+1)}} \Big(\frac{1}{T^\nu h^{\nu-1}}   + h^{d+1} + h^{d+3}\Big)^{\frac{2}{p}} \sum_{k=s_T+1}^{\infty} k^\zeta \beta(k)^{1-\frac{2}{p}}\nonumber\\
    &\lesssim \frac{1}{T^2 h^{2(d+1)}} \Big(\frac{1}{T^\nu h^{\nu-1}}   + h^{d+1} + h^{d+3}\Big)^{\frac{2}{p}}, \quad \text{since } \frac{1}{s_T^\zeta} \leq 1, \nonumber\\
    &\lesssim \Big(\frac{1}{T^{2p} h^{2(d+1)p}}   \Big(\frac{1}{T^\nu h^{\nu-1}}   + h^{d+1} + h^{d+3}\Big)^2 \Big)^\frac{1}{p}\nonumber\\
    &\lesssim \Big(\frac{1}{T^{2p} h^{2(d+1)p}}   \Big(\frac{1}{T^{2\nu} h^{2(\nu-1)}}   + h^{2(d+1)} \Big) \Big)^\frac{1}{p}\nonumber\\
    &\lesssim \Big(\frac{1}{T^{2(p+\nu)} h^{2(d + 1)p + 2(\nu - 1)}} + \frac{1}{T^{2p} h^{2(d + 1)p - 2(d + 1)}}\Big)^\frac{1}{p}\nonumber\\
    &\lesssim \Big(\frac{1}{T^{2p} h^{2(d + 1)p + 2(\nu - 1)}}\Big)^\frac{1}{p}\nonumber\\
    &\lesssim \frac{1}{T^2 h^{2(d + 1) -\frac{2}{p}(1-\nu) }}.
\end{align}
\noindent {\textcolor{magenta}{\underline{\it Step 1.3.2 Control of $\mathlarger{\mathsf S}_{32}^{\Lambda}$.}}} In view of $\mathlarger{\mathsf S}_{32}^{\Lambda}$, observe that
\begin{align*}
    \mathlarger{\mathsf S}_{32}^{\Lambda} &= \frac{1}{(Th^{d+1})^2} \sum_{\substack{l=0\\ \quad \mathrlap{l\neq l'}}}^{v_T-1}\sum_{l'=0}^{v_T-1} \sum_{a=l(r_T+s_T)+1}^{l(r_T+s_T)+r_T} \sum_{b=l'(r_T+s_T)+1}^{l'(r_T+s_T)+r_T} K_{h,1}\big(\frac{t}{T} - \frac{a}{T}\big) K_{h,1}\big(\frac{t}{T} - \frac{b}{T}\big)\\
    &\hspace{8cm} \times  \E[Z_{a,t,T}] \E[Z_{b,t,T}]\\
    &= \frac{1}{(Th^{d+1})^2} \sum_{\substack{l=0 \\ \quad \mathrlap{l\neq l'}}}^{v_T-1}\sum_{l'=0}^{v_T-1}\sum_{n_1 =1}^{r_T}\sum_{n_2=1}^{r_T} K_{h,1}\big(\frac{t}{T} - \frac{\lambda + n_1}{T}\big)K_{h,1}\big(\frac{t}{T} - \frac{\lambda'+n_2}{T}\big)\\
    &\hspace{8cm} \times \E[Z_{\lambda + n_1,t,T}] \E[Z_{\lambda'+n_2,t,T}].
\end{align*}
Similarly, for $l\neq l'$, $|\lambda - \lambda' + n_1 - n_2| > s_T$, then
\begin{align*}
    \mathlarger{\mathsf S}_{32}^{\Lambda} &\leq \frac{1}{(Th^{d+1})^2} \sum_{\substack{m=1\\ \mathrlap{|m-m'|> s_T}}}^{T}\sum_{m'=1}^{T} K_{h,1}\big(\frac{t}{T} - \frac{m}{T}\big)K_{h,1}\big(\frac{t}{T} - \frac{m'}{T}\big)\E[Z_{m,t,T}] \E[Z_{m',t,T}]\nonumber\\
    &= \frac{1}{(Th^{d+1})^2} \sum_{\substack{m=1\\ \mathrlap{|m-m'|> s_T}}}^{T}\sum_{m'=1}^{T} K_{h,1}\big(\frac{t}{T} - \frac{m}{T}\big)K_{h,1}\big(\frac{t}{T} - \frac{m'}{T}\big)  \nonumber\\
    &\hspace{4cm} \times \E\Big[\prod_{j=1}^d K_{h,2}(x^j - X_{m,T}^j)(\mathds{1}_{Y_{m,T}\leq y} - F_t^\star (y|\boldsymbol{x}))\Big]\nonumber\\
    &\hspace{4cm} \times  \E\Big[\prod_{j=1}^d K_{h,2}(x^j - X_{m',T}^j)(\mathds{1}_{Y_{m',T}\leq y} - F_t^\star (y|\boldsymbol{x}))\Big].
\end{align*}
Using Proposition \ref{Lemma: E of K2}.(\textit{iii}), $K_{h,1}\big(\frac{t}{T} - \frac{m}{T}\big)\E\big[\prod_{j=1}^d K_{h,2}(x^j - X_{m,T}^j)(\mathds{1}_{Y_{m,T}\leq y} - F_t^\star (y|\boldsymbol{x}))\big] \lesssim K_{h,1}\big(\frac{t}{T} - \frac{m}{T}\big) \big( \frac{1}{T^\nu h^{\nu-1}}  + h^{d+1} + h^{d+3} \big)$, then
\begin{align}\label{eqn: big block 3}
    \mathlarger{\mathsf S}_{32}^{\Lambda} 
    &\lesssim \frac{1}{(Th^{d+1})^2} \Big(\frac{1}{T^\nu h^{\nu-1}}  + h^{d+1} + h^{d+3}\Big)^2 \sum_{\substack{m=1\\ \mathrlap{|m-m'|> s_T}}}^{T}\sum_{m'=1}^{T} K_{h,1}\big(\frac{t}{T} - \frac{m}{T}\big)K_{h,1}\big(\frac{t}{T} - \frac{m'}{T}\big)  \nonumber\\
    &\leq \frac{1}{h^{2d}} \Big(\frac{1}{T^\nu h^{\nu-1}}  + h^{d+1} + h^{d+3}\Big)^2 \underbrace{\frac{1}{Th}\sum_{m=1}^{T} K_{h,1}\big(\frac{t}{T} - \frac{m}{T}\big)}_{\bigO(1)} \underbrace{\frac{1}{Th}\sum_{m'=1}^{T}  K_{h,1}\big(\frac{t}{T} - \frac{m'}{T}\big)}_{\bigO(1)}\nonumber\\
    &\lesssim \frac{1}{h^{2d}} \Big(\frac{1}{T^\nu h^{\nu-1}}  + h^{d+1} + h^{d+3}\Big)^2 
    \lesssim \frac{1}{h^{2d}} \Big( \frac{1}{T^{2\nu} h^{2(\nu - 1)}}   + h^{2(d+1)}  \Big)\nonumber\\
    &\lesssim \frac{1}{T^{2\nu} h^{2(d+\nu-1)}} + h^2, 
\end{align}
which goes to zero as $T\rightarrow \infty$ using Assumption \ref{Assumption: bandwidth}. Hence, comparing (\ref{eqn: big block 1}), (\ref{eqn: big block 2_1}), (\ref{eqn: big block 2_2}), (\ref{eqn: big block 3_1}), and (\ref{eqn: big block 3}), we have
\begin{align}\label{eqn: big blocks overall order}
    \E\big[\Lambda_{t,T}^2\big]
    &\lesssim  \frac{1}{Th^{2(d + 1) -\frac{2}{p}(1 - \nu) }} + \frac{1}{T^{2\nu} h^{2(d+\nu-1)}} + h^2.
\end{align}\\

\noindent {\textcolor{magenta}{\underline{\it Step 2. Control of the small blocks.}}} Next, we deal with the small blocks. See that

\begin{align*}
    \E\big[\Pi_{t,T}^2\big]
    &= \E \Big[\sum_{l=0}^{v_T-1} \Pi_{l,t,T}^2 + \sum_{\substack{l=0 \\ \quad \mathrlap{l\neq l'}}}^{v_T-1}\sum_{l'=0}^{v_T-1} \Pi_{l,t,T}\Pi_{l',t,T} \Big]\\
                &= \frac{1}{(Th^{d+1})^2}\sum_{l=0}^{v_T-1} \sum_{a=l(r_T+s_T)+r_T+1}^{(l+1)(r_T+s_T)}  K_{h,1}^2\big(\frac{t}{T} - \frac{a}{T}\big) \E\big[Z_{a,t,T}^2\big]\\
    &\qquad + \frac{1}{(Th^{d+1})^2} \sum_{l=0}^{v_T-1} \sum_{\substack{a=l(r_T+s_T)+r_T+1 \\ \qquad \qquad \quad \mathrlap{a \neq b}}}^{(l+1)(r_T+s_T)} \sum_{b=l(r_T+s_T)+r_T+1}^{(l+1)(r_T+s_T)} K_{h,1}\big(\frac{t}{T} - \frac{a}{T}\big) \\
    &\hspace{8cm} \times K_{h,1}\big(\frac{t}{T} - \frac{b}{T}\big)  \E\big[Z_{a,t,T} Z_{b,t,T}\big]\\
    &\qquad \quad + \frac{1}{(Th^{d+1})^2} \sum_{\substack{l=0 \\ \quad \mathrlap{l\neq l'}}}^{v_T-1}\sum_{l'=0}^{v_T-1} \sum_{a=l(r_T+s_T)+r_T+1}^{(l+1)(r_T+s_T)} \sum_{b=l'(r_T+s_T)+r_T+1}^{(l'+1)(r_T+s_T)}  K_{h,1}\big(\frac{t}{T} - \frac{a}{T}\big) \\
    &\hspace{8cm} \times K_{h,1}\big(\frac{t}{T} - \frac{b}{T}\big) \E\big[Z_{a,t,T} Z_{b,t,T}\big]\\
    &=: \mathlarger{\mathsf S}_{1}^\Pi + \mathlarger{\mathsf S}_{2}^\Pi + \mathlarger{\mathsf S}_{3}^\Pi.
\end{align*}
\noindent {\textcolor{magenta}{\underline{\it Step 2.1. Control of $\mathlarger{\mathsf S}_{1}^\Pi$}}} First, let us consider $\mathlarger{\mathsf S}_{1}^\Pi$.
\begin{align*}
    \mathlarger{\mathsf S}_{1}^\Pi     &= \frac{1}{(Th^{d+1})^2}\sum_{l=0}^{v_T-1} \sum_{a=l(r_T+s_T)+r_T+1}^{(l+1)(r_T+s_T)} K_{h,1}^2\big(\frac{t}{T} - \frac{a}{T}\big) \nonumber\\
    &\hspace{6cm} \times \E\Big[ \prod_{j=1}^d K_{h,2}^2(x^j - X_{a,T}^j)(\mathds{1}_{Y_{a,T}\leq y} - F_t^\star (y|\boldsymbol{x}))^2 \Big]\\
    &\leq \frac{2 C_2^d}{(Th^{d+1})^2}\sum_{l=0}^{v_T-1} \sum_{a=l(r_T+s_T)+r_T+1}^{(l+1)(r_T+s_T)} K_{h,1}^2\big(\frac{t}{T} - \frac{a}{T}\big) \nonumber\\
    &\hspace{6cm} \times \E\Big[ \prod_{j=1}^d K_{h,2}(x^j - X_{a,T}^j)\big| \mathds{1}_{Y_{a,T}\leq y} - F_t^\star (y|\boldsymbol{x})\big| \Big].
\end{align*}
By Proposition \ref{Lemma: E of K2}.(\textit{iii}), we get
\begin{align}\label{eqn: small block 1}
    \mathlarger{\mathsf S}_{1}^\Pi
    &\lesssim \frac{1}{T^2h^{2(d+1)}} \Big(\frac{1}{T^\nu h^{\nu-1}}  + h^{d+1} + h^{d+3}\Big)  \sum_{l=0}^{v_T-1} \sum_{a=l(r_T+s_T)+r_T+1}^{(l+1)(r_T+s_T)}  K_{h,1}^2\big(\frac{t}{T} - \frac{a}{T}\big) \nonumber\\
    &\leq \frac{C_1}{T^2h^{2d+2}} \Big(\frac{1}{T^\nu h^{\nu-1}}  + h^{d+1} + h^{d+3}\Big)  \sum_{l=0}^{v_T-1} \sum_{a=l(r_T+s_T)+r_T+1}^{(l+1)(r_T+s_T)}  K_{h,1}\big(\frac{t}{T} - \frac{a}{T}\big) \nonumber\\
    &\leq \frac{C_1}{Th^{2d+1}} \Big(\frac{1}{T^\nu h^{\nu-1}}  + h^{d+1} + h^{d+3}\Big)  \underbrace{\frac{1}{Th} \sum_{a=1}^{T} K_{h,1}\big(\frac{t}{T} - \frac{a}{T}\big)}_{\bigO(1)} \nonumber\\
    &\lesssim \frac{1}{T h^{2d+1}} \Big(\frac{1}{T^\nu h^{\nu-1}}  + h^{d+1} + h^{d+3}\Big)\nonumber\\
    &\lesssim \frac{1}{T^{1+\nu} h^{2d+\nu}} + \frac{1}{T h^{d}}\nonumber\\
    &\lesssim \frac{1}{T h^{2d+\nu}}.
\end{align}
\noindent {\textcolor{magenta}{\underline{\it Step 2.2. Control of $\mathlarger{\mathsf S}_{2}^\Pi.$}}}
On the other hand,
\begin{align*}
    \mathlarger{\mathsf S}_{2}^\Pi &= \frac{1}{(Th^{d+1})^2} \sum_{l=0}^{v_T-1} \sum_{\substack{a=l(r_T+s_T)+r_T+1 \\ \qquad \qquad \quad \mathrlap{a\neq b}}}^{(l+1)(r_T+s_T)} \sum_{b=l(r_T+s_T)+r_T+1}^{(l+1)(r_T+s_T)} K_{h,1}\big(\frac{t}{T} - \frac{a}{T}\big)K_{h,1}\big(\frac{t}{T} - \frac{b}{T}\big) \E[Z_{a,t,T}Z_{b,t,T}]\\
    &= \frac{1}{(Th^{d+1})^2} \sum_{l=0}^{v_T-1} \sum_{\substack{n_1=1 \\ \mathrlap{|n_1-n_2|>0}}}^{s_T} \sum_{n_2=1}^{s_T} K_{h,1}\big(\frac{t}{T} - \frac{\lambda+n_1}{T}\big)K_{h,1}\big(\frac{t}{T} - \frac{\lambda+n_2}{T}\big)\\
    &\quad \times \big\{ \text{$\C$ov}\big(Z_{\lambda+n_1,t,T},Z_{\lambda+n_2,t,T}\big) + \E[Z_{\lambda+n_1,t,T}]\E[Z_{\lambda+n_2,t,T}]  \big\},
\end{align*}
where $\lambda = l(r_T+s_T)+r_T$. So
\begin{align*}
    \mathlarger{\mathsf S}_{2}^\Pi
    &= \frac{1}{(Th^{d+1})^2} \sum_{l=0}^{v_T-1} \sum_{\substack{n_1=1 \\ \mathrlap{|n_1-n_2|>0}}}^{s_T} \sum_{n_2=1}^{s_T} K_{h,1}\big(\frac{t}{T} - \frac{\lambda+n_1}{T}\big) K_{h,1}\big(\frac{t}{T} - \frac{\lambda+n_2}{T}\big) \text{$\C$ov}\Big(Z_{\lambda+n_1,t,T}, Z_{\lambda+n_2,t,T}\Big)\\
    &\qquad + \frac{1}{(Th^{d+1})^2} \sum_{l=0}^{v_T-1} \sum_{\substack{n_1=1 \\ \mathrlap{|n_1-n_2|>0}}}^{s_T} \sum_{n_2=1}^{s_T} K_{h,1}\big(\frac{t}{T} - \frac{\lambda+n_1}{T}\big) K_{h,1}\big(\frac{t}{T} - \frac{\lambda+n_2}{T}\big) \\
    &\hspace{8cm} \times \E\Big[  Z_{\lambda+n_1,t,T}\Big]  \E\Big[Z_{\lambda+n_2,t,T}\Big]\\
    &=: \mathlarger{\mathsf S}_{21}^\Pi + \mathlarger{\mathsf S}_{22}^\Pi.
\end{align*}
\noindent {\textcolor{magenta}{\underline{\it Step 2.2.1. Control of $\mathlarger{\mathsf S}_{21}^\Pi$.}}} Taking $\mathlarger{\mathsf S}_{21}^\Pi$ into consideration, we have
\begin{align*}
    \mathlarger{\mathsf S}_{21}^\Pi &= \frac{1}{(Th^{d+1})^2} \sum_{l=0}^{v_T-1} \sum_{\substack{n_1=1 \\ \mathrlap{|n_1-n_2|>0}}}^{s_T} \sum_{n_2=1}^{s_T} K_{h,1}\big(\frac{t}{T} - \frac{\lambda+n_1}{T}\big) K_{h,1}\big(\frac{t}{T} - \frac{\lambda+n_2}{T}\big) \text{$\C$ov}\Big(Z_{\lambda+n_1,t,T}, Z_{\lambda+n_2,t,T}\Big).
\end{align*}
Using (\ref{eqn: cov within blocks}),
\begin{align*}
    \lefteqn{K_{h,1}\big(\frac{t}{T} - \frac{\lambda + n_1}{T}\big)K_{h,1}\big(\frac{t}{T} - \frac{\lambda + n_2}{T}\big)\big|\text{$\C$ov}\big(Z_{\lambda+n_1,t,T},Z_{\lambda+n_2,t,T}\big)\big|}\\
    &\lesssim K_{h,1}\big(\frac{t}{T} - \frac{\lambda + n_1}{T}\big)K_{h,1}\big(\frac{t}{T} - \frac{\lambda + n_2}{T}\big)\big(\frac{1}{T^\nu h^{\nu-1}}   + h^{d+1} + h^{d+3}\Big)^{\frac{2}{p}} \beta(|n_1 - n_2|)^{1-\frac{2}{p}}.
\end{align*} 
Thus
\begin{align*}
    \mathlarger{\mathsf S}_{21}^\Pi 
    &\lesssim \frac{1}{T^2 h^{2(d+1)}} \Big(\frac{1}{T^\nu h^{\nu-1}}   + h^{d+1} + h^{d+3}\Big)^{\frac{2}{p}} \sum_{l=0}^{v_T-1} \sum_{\substack{n_1=1 \\  \mathrlap{|n_1-n_2|>0}}}^{s_T} \sum_{n_2=1}^{s_T}  K_{h,1}\big(\frac{t}{T} - \frac{\lambda+n_1}{T}\big) K_{h,1}\big(\frac{t}{T} - \frac{\lambda+n_2}{T}\big)\nonumber\\
    &\hspace{8cm} \times  \beta(|n_1 - n_2|)^{1-\frac{2}{p}}\nonumber\\
    &\leq \frac{C_1^2}{T^2 h^{2(d+1)}} \Big(\frac{1}{T^\nu h^{\nu-1}}   + h^{d+1} + h^{d+3}\Big)^{\frac{2}{p}} \sum_{l=0}^{v_T-1} \sum_{\substack{n_1=1 \\  \mathrlap{|n_1-n_2|>0}}}^{s_T} \sum_{n_2=1}^{s_T} \beta(|n_1 - n_2|)^{1-\frac{2}{p}}.
\end{align*}
Using Assumption \ref{Assumption: mixing}, $\sum_{k=1}^{\infty} k^\zeta \beta(k)^{1-\frac{2}{p}}<\infty$, which can be expressed as $\sum_{k=1}^{s_T} k^\zeta \beta(k)^{1-\frac{2}{p}} + \sum_{k=s_T+1}^{\infty} k^\zeta \beta(k)^{1-\frac{2}{p}}$. In addition, letting $k = |n_1 -n_2|$ yields
\begin{align*}
    \sum_{\substack{n_1=1 \\  \mathrlap{|n_1-n_2|>0}}}^{s_T} \sum_{n_2=1}^{s_T} \beta(|n_1 - n_2|)^{1-\frac{2}{p}}
    &= \sum_{n_1=1}^{s_T} \Big( \sum_{n_2>n_1}^{s_T} \beta(n_2 - n_1)^{1-\frac{2}{p}} + \sum_{n_2<n_1}^{s_T} \beta(n_1 - n_2)^{1-\frac{2}{p}}\Big)\\
    &= \sum_{n_1=1}^{s_T}\sum_{k>0}^{s_T-n_1} \beta(k)^{1 - \frac{2}{p}} + \sum_{n_2=1}^{s_T}\sum_{k>0}^{s_T-n_2} \beta(k)^{1 - \frac{2}{p}}\\
    &= 2 \sum_{n=1}^{s_T}\sum_{k>0}^{s_T-n} \beta(k)^{1 - \frac{2}{p}}
        \leq 2s_T \sum_{k=1}^{s_T} \beta(k)^{1 - \frac{2}{p}}\\
    &\lesssim  s_T \sum_{k=1}^{s_T} k^\zeta \beta(k)^{1 - \frac{2}{p}}
    \leq s_T \sum_{k=1}^{\infty} k^\zeta \beta(k)^{1 - \frac{2}{p}},
\end{align*}
            since $\beta(k)\geq 0$ and $k^\zeta \geq 1$ for $\zeta > 1 - \frac{2}{p}$, where $p>2$. So
\begin{align}\label{eqn: small block 2_1}
    \mathlarger{\mathsf S}_{21}^\Pi 
    &\leq \frac{C_1^2 s_T}{T^2 h^{2(d+1)}} \Big(\frac{1}{T^\nu h^{\nu-1}}   + h^{d+1} + h^{d+3}\Big)^{\frac{2}{p}} \sum_{l=0}^{v_T-1} \sum_{k=1}^{\infty} k^\zeta \beta(k)^{1-\frac{2}{p}}\nonumber\\
    &\lesssim \frac{v_T s_T}{T^2 h^{2(d+1)}} \Big(\frac{1}{T^\nu h^{\nu-1}}   + h^{d+1} + h^{d+3}\Big)^{\frac{2}{p}} \sum_{k=1}^{\infty} k^\zeta \beta(k)^{1-\frac{2}{p}}\nonumber\\
    &\lesssim \frac{1}{T h^{2(d+1)}} \Big(\frac{1}{T^\nu h^{\nu-1}}   + h^{d+1} + h^{d+3}\Big)^{\frac{2}{p}}, \quad \text{since } 
    v_Ts_T \leq \frac{T}{s_T}s_T = T, \nonumber\\
    &= \Big( \frac{1}{T^{p}h^{2(d +1)p}} \Big(\frac{1}{T^\nu h^{\nu-1}}   + h^{d+1} + h^{d+3}\Big)^2 \Big)^\frac{1}{p}\\ 
    &\lesssim \Big(\frac{1}{T^{p} h^{2(d+1)p}}   \Big(\frac{1}{T^{2\nu} h^{2(\nu-1)}}   + h^{2(d+1)} \Big) \Big)^\frac{1}{p}\nonumber\\
    &\lesssim \Big(\frac{1}{T^{p+2\nu} h^{2(d + 1)p + 2(\nu - 1)}} + \frac{1}{T^p h^{2(d + 1)p - 2(d + 1)}}\Big)^\frac{1}{p}\nonumber\\
    &\lesssim \Big(\frac{1}{T^{p} h^{2(d + 1)p + 2(\nu - 1)}}\Big)^\frac{1}{p}\nonumber\\
    &\lesssim \frac{1}{T h^{2(d + 1) -\frac{2}{p}(1 - \nu) }}.
\end{align}\\

\noindent {\textcolor{magenta}{\underline{\it Step 2.2.2. Control of $\mathlarger{\mathsf S}_{22}^\Pi$.}}} Next, looking at $\mathlarger{\mathsf S}_{22}^\Pi$, we have
\begin{align*}
    \mathlarger{\mathsf S}_{22}^\Pi &= \frac{1}{(Th^{d+1})^2} \sum_{l=0}^{v_T-1} \sum_{\substack{a=l(r_T+s_T)+1\\ \qquad \quad \mathrlap{|a-b|>0}}}^{l(r_T+s_T)+r_T} \sum_{b=l(r_T+s_T)+1}^{l(r_T+s_T)+r_T} K_{h,1}\big(\frac{t}{T} - \frac{a}{T}\big) K_{h,1}\big(\frac{t}{T} - \frac{b}{T}\big) \E\big[Z_{a,t,T}\big] \E\big[Z_{b,t,T}\big]. 
\end{align*}
Now see that
\begin{align*}
    \mathlarger{\mathsf S}_{22}^\Pi 
    &= \frac{1}{(Th^{d+1})^2} \sum_{l=0}^{v_T-1} \sum_{\substack{n_1=1 \\ \mathrlap{|n_1-n_2|>0}}}^{s_T} \sum_{n_2=1}^{s_T} K_{h,1}\big(\frac{t}{T} - \frac{\lambda+n_1}{T}\big) K_{h,1}\big(\frac{t}{T} - \frac{\lambda+n_2}{T}\big) \nonumber\\
    &\hspace{6cm} \times \E\Big[\prod_{j=1}^d K_{h,2}(x^j - X_{\lambda+n_1,T}^j)(\mathds{1}_{Y_{\lambda + n_1,T}\leq y} - F_t^\star (y|\boldsymbol{x}))\Big] \nonumber\\
    &\hspace{6cm} \times \E\Big[ \prod_{j=1}^d K_{h,2}(x^j - X_{\lambda+n_2,T}^j)(\mathds{1}_{Y_{\lambda + n_2,T}\leq y} - F_t^\star (y|\boldsymbol{x}))\Big]. 
\end{align*}
By Proposition \ref{Lemma: E of K2}.(\textit{iii}), for $i=1,2$, 
\begin{eqnarray*}
    \lefteqn{K_{h,1}\big(\frac{t}{T} - \frac{\lambda+n_i}{T}\big) \E\big[\prod_{j=1}^d K_{h,2}(x^j - X_{\lambda+n_i,T}^j)(\mathds{1}_{Y_{\lambda + n_i,T}\leq y} - F_t^\star (y|\boldsymbol{x}))\big]} \\&&\lesssim K_{h,1}\big(\frac{t}{T} - \frac{\lambda+n_i}{T}\big) \big(\frac{1 }{T^\nu h^{\nu-1}}  + h^{d+1} + h^{d+3}\big),
\end{eqnarray*}
then 
\begin{align}\label{eqn: small block 2_2}
    \mathlarger{\mathsf S}_{22}^\Pi 
    &\lesssim \frac{1}{T^2 h^{2d+2}}  \Big(\frac{1}{T^\nu h^{\nu-1}}  + h^{d+1} + h^{d+3}\Big)^2 \sum_{l=0}^{v_T-1} \sum_{\substack{n_1=1 \\ \mathrlap{|n_1-n_2|>0}}}^{s_T} \sum_{n_2=1}^{s_T} K_{h,1}\big(\frac{t}{T} - \frac{\lambda+n_1}{T}\big) K_{h,1}\big(\frac{t}{T} - \frac{\lambda+n_2}{T}\big)\nonumber\\
    &\leq \frac{C_1}{T h^{2d+1}}  \Big(\frac{1}{T^\nu h^{\nu-1}}  + h^{d+1} + h^{d+3}\Big)^2 \underbrace{\frac{1}{Th} \sum_{a=1}^{T} K_{h,1}\big(\frac{t}{T} - \frac{a}{T}\big)}_{\bigO(1)} \nonumber\\
    &\lesssim  \frac{1}{Th^{2d+1}} \Big(\frac{1}{T^\nu h^{\nu-1}}  + h^{d+1} + h^{d+3}\Big)^2
    \lesssim \frac{1}{Th^{2d+1}} \Big( \frac{1}{T^{2\nu} h^{2(\nu - 1)}}  + h^{2(d+1)} \Big)\nonumber\\
    &\lesssim \frac{1}{T^{1+ 2\nu} h^{2(d+\nu)-1}} + \frac{h}{T}
    \lesssim  \frac{1}{T h^{2(d+\nu)-1}}.
\end{align}
\noindent {\textcolor{magenta}{\underline{\it Step 2.3. Control of $\mathlarger{\mathsf S}_{3}^\Pi$.}}} Now, let us deal with $\mathlarger{\mathsf S}_{3}^\Pi$.
\begin{align*}
    \mathlarger{\mathsf S}_{3}^\Pi &= \frac{1}{(Th^{d+1})^2} \sum_{\substack{l=0 \\ \quad \mathrlap{l\neq l'}}}^{v_T-1}\sum_{l'=0}^{v_T-1} \sum_{a=l(r_T+s_T)+r_T+1}^{(l+1)(r_T+s_T)} \sum_{b=l'(r_T+s_T)+r_T+1}^{(l'+1)(r_T+s_T)} K_{h,1}\big(\frac{t}{T} - \frac{a}{T}\big) K_{h,1}\big(\frac{t}{T} - \frac{b}{T}\big)\\
    &\hspace{8cm} \times  \text{$\C$ov}\big(Z_{a,t,T},Z_{b,t,T}\big)\\
    &\qquad + \frac{1}{(Th^{d+1})^2} \sum_{\substack{l=0 \\ \quad \mathrlap{l\neq l'}}}^{v_T-1}\sum_{l'=0}^{v_T-1} \sum_{a=l(r_T+s_T)+r_T+1}^{(l+1)(r_T+s_T)} \sum_{b=l'(r_T+s_T)+r_T+1}^{(l'+1)(r_T+s_T)} K_{h,1}\big(\frac{t}{T} - \frac{a}{T}\big) K_{h,1}\big(\frac{t}{T} - \frac{b}{T}\big) \\
    &\hspace{8cm} \times  \E\big[Z_{a,t,T}\big] \E\big[Z_{b,t,T}\big]\\
    &= \mathlarger{\mathsf S}_{31}^\Pi + \mathlarger{\mathsf S}_{32}^\Pi.
\end{align*}
\noindent {\textcolor{magenta}{\underline{\it Step 2.3.1 Control of $\mathlarger{\mathsf S}_{31}^\Pi$.}}} Looking at $\mathlarger{\mathsf S}_{31}^\Pi$, see that
\begin{align*}
    \mathlarger{\mathsf S}_{31}^\Pi &= \frac{1}{(Th^{d+1})^2} \sum_{\substack{l=0 \\ \quad \mathrlap{l\neq l'}}}^{v_T-1}\sum_{l'=0}^{v_T-1}\sum_{n_1 =1}^{s_T}\sum_{n_2=1}^{s_T} K_{h,1}\big(\frac{t}{T} - \frac{\lambda + n_1}{T}\big)K_{h,1}\big(\frac{t}{T} - \frac{\lambda'+n_2}{T}\big)\\
    &\hspace{8cm} \times \text{$\C$ov}\big(Z_{\lambda + n_1,t,T},Z_{\lambda'+n_2,t,T}\big),
\end{align*}
where $\lambda = l(r_T+s_T)+r_T$ and $\lambda' = l'(r_T+s_T)+r_T$, however, for $l\neq l'$, 
\begin{align*}
    |\lambda - \lambda' + n_1 - n_2| &\geq | l(r_T+s_T)+r_T - l'(r_T+s_T)-r_T + n_1 - n_2 |\\
    &\geq |(l-l')(r_T+s_T) + n_1 - n_2|
        > r_T,
\end{align*}
since $n_1,n_2\in \{1, \ldots, s_T\}$. So if we let $q = \lambda+n_1$ and $q'=\lambda'+n_2$, we have
\begin{align*}
    \mathlarger{\mathsf S}_{31}^\Pi &= \frac{1}{(Th^{d+1})^2} \sum_{\substack{q=r_T+1 \\ \quad \mathrlap{|q-q'|> r_T}}}^{v_T(r_T+s_T)}\sum_{q'=r_T+1}^{v_T(r_T+s_T)} K_{h,1}\big(\frac{t}{T} - \frac{q}{T}\big) K_{h,1}\big(\frac{t}{T} - \frac{q'}{T}\big)  \text{$\C$ov}\big(Z_{q,t,T}, Z_{q',t,T}\big)\\
    &= \frac{1}{(Th^{d+1})^2} \sum_{\substack{m=1 \\ \qquad \mathrlap{|m-m'|> r_T}}}^{v_T(r_T+s_T)-r_T}\sum_{m'=1}^{v_T(r_T+s_T)-r_T} K_{h,1}\big(\frac{t}{T} - \frac{m}{T}\big) K_{h,1}\big(\frac{t}{T} - \frac{m'}{T}\big)  \text{$\C$ov}\big(Z_{m,t,T},Z_{m',t,T}\big)\\
    &\leq \frac{1}{(Th^{d+1})^2} \sum_{\substack{m=1 \\  \mathrlap{|m-m'|> r_T}}}^{T}\sum_{m'=1}^{T} K_{h,1}\big(\frac{t}{T} - \frac{m}{T}\big) K_{h,1}\big(\frac{t}{T} - \frac{m'}{T}\big) \big|\text{$\C$ov}\big(Z_{m,t,T},Z_{m',t,T}\big)\big|,
\end{align*}
where $m = q - r_T$ and $m' = q' - r_T$. Now, using (\ref{eqn: cov within blocks}), we have

\begin{align*}
    \lefteqn{ K_{h,1}\big(\frac{t}{T} - \frac{m}{T}\big) K_{h,1}\big(\frac{t}{T} - \frac{m'}{T}\big) \big|\text{$\C$ov}\big(Z_{m,t,T},Z_{m',t,T}\big)\big|}\\
    &\lesssim K_{h,1}\big(\frac{t}{T} - \frac{m}{T}\big) K_{h,1}\big(\frac{t}{T} - \frac{m'}{T}\big) \big(\frac{1}{T^\nu h^{\nu-1}}   + h^{d+1} + h^{d+3}\Big)^{\frac{2}{p}} \beta(|m - m'|)^{1-\frac{2}{p}}.
\end{align*}
Thus
\begin{align*}
    \mathlarger{\mathsf S}_{31}^\Pi &\lesssim \frac{1}{(Th^{d+1})^2} \Big(\frac{1}{T^\nu h^{\nu-1}}   + h^{d+1} + h^{d+3}\Big)^{\frac{2}{p}} \sum_{\substack{m=1 \\ \mathrlap{|m-m'|> r_T}}}^{T}\sum_{m'=1}^{T} K_{h,1}\big(\frac{t}{T} - \frac{m}{T}\big) K_{h,1}\big(\frac{t}{T} - \frac{m'}{T}\big) \beta(|m-m'|)^{1-\frac{2}{p}}\nonumber\\
    &\leq \frac{C_1^2}{T^2 h^{2d+2}} \Big(\frac{1}{T^\nu h^{\nu-1}}   + h^{d+1} + h^{d+3}\Big)^{\frac{2}{p}} \sum_{\substack{m=1 \\ \mathrlap{|m-m'|> r_T}}}^{T}\sum_{m'=1}^{T} \beta(|m-m'|)^{1-\frac{2}{p}}.
\end{align*}
By Assumption \ref{Assumption: mixing}, $\sum_{k=1}^{\infty} k^\zeta \beta(k)^{1-\frac{2}{p}}<\infty$, which can be expressed as $\sum_{k=1}^{r_T} k^\zeta \beta(k)^{1-\frac{2}{p}} + \sum_{k=r_T+1}^{\infty} k^\zeta \beta(k)^{1-\frac{2}{p}}$. Additionally, observe that letting $k = |m - m'|$ yields
\begin{align*}
    \sum_{\substack{m=1 \\  \mathrlap{|m-m'|> r_T}}}^{T} \sum_{m'=1}^{T} \beta(|m - m'|)^{1-\frac{2}{p}}
    &\leq C \sum_{k = r_T+1}^{T} \beta(k)^{1-\frac{2}{p}}
    \lesssim \frac{1}{k^\zeta} \sum_{k = r_T+1}^{T} k^\zeta \beta(k)^{1-\frac{2}{p}}\\
    &\leq  \frac{1}{r_T^\zeta} \sum_{k = r_T+1}^{T} k^\zeta \beta(k)^{1-\frac{2}{p}}, \quad \text{since } k > r_T,\\
    &\leq \frac{1}{r_T^\zeta} \sum_{k = r_T+1}^{\infty} k^\zeta \beta(k)^{1-\frac{2}{p}},
\end{align*}
            since $\beta(k)\geq 0$ and $\big(\frac{k}{r_T}\big)^\zeta \geq 1$ for $\zeta > 1 - \frac{2}{p}$, where $p>2$. So
\begin{align}\label{eqn: small block 3_1}
    \mathlarger{\mathsf S}_{31}^\Pi 
    &\lesssim \frac{1}{r_T^\zeta T^2 h^{2(d+1)}} \Big(\frac{1}{T^\nu h^{\nu-1}}   + h^{d+1} + h^{d+3}\Big)^{\frac{2}{p}} \sum_{k=r_T+1}^{\infty} k^\zeta \beta(k)^{1-\frac{2}{p}}\nonumber\\
    &\lesssim \frac{1}{ T^2 h^{2(d+1)}} \Big(\frac{1}{T^\nu h^{\nu-1}}   + h^{d+1} + h^{d+3}\Big)^{\frac{2}{p}}, \quad \text{since } \frac{1}{r_T^\zeta} \leq 1, \nonumber\\
    &= \Big(\frac{1}{T^{2p} h^{2(d+1)p}}   \Big(\frac{1}{T^\nu h^{\nu-1}}   + h^{d+1} + h^{d+3}\Big)^2 \Big)^\frac{1}{p}\nonumber\\
    &\lesssim \Big(\frac{1}{T^{2p} h^{2(d+1)p}}   \Big(\frac{1}{T^{2\nu} h^{2(\nu-1)}}   + h^{2(d+1)} \Big) \Big)^\frac{1}{p}\nonumber\\
    &\lesssim \Big(\frac{1}{T^{2(p+\nu)} h^{2(d + 1)p + 2(\nu - 1)}} + \frac{1}{T^{2p} h^{2(d + 1)p - 2(d - 1)}}\Big)^\frac{1}{p}\nonumber\\
    &\lesssim \Big(\frac{1}{T^{2p} h^{2(d + 1)p + 2(\nu - 1)}}\Big)^\frac{1}{p}\nonumber\\
    &\lesssim \frac{1}{T^2 h^{2(d + 1) -\frac{2}{p}(1-\nu) }}.
\end{align}
\noindent {\textcolor{magenta}{\underline{\it Step 2.3.2 Control of $\mathlarger{\mathsf S}_{32}^\Pi$.}}} In dealing with $\mathlarger{\mathsf S}_{32}^\Pi$, observe that
\begin{align*}
    \mathlarger{\mathsf S}_{32}^\Pi &= \frac{1}{(Th^{d+1})^2} \sum_{\substack{l=0 \\ \quad \mathrlap{l\neq l'}}}^{v_T-1}\sum_{l'=0}^{v_T-1}\sum_{n_1 =1}^{s_T}\sum_{n_2=1}^{s_T} K_{h,1}\big(\frac{t}{T} - \frac{\lambda + n_1}{T}\big)K_{h,1}\big(\frac{t}{T} - \frac{\lambda'+n_2}{T}\big)\\
    &\hspace{8cm}\times \E\big[Z_{\lambda + n_1,t,T}\big] \E\big[Z_{\lambda'+n_2,t,T}\big]\\
    &= \frac{1}{(Th^{d+1})^2}  \sum_{\substack{l=0 \\ \quad \mathrlap{l\neq l'}}}^{v_T-1}\sum_{l'=0}^{v_T-1}\sum_{n_1 =1}^{s_T}\sum_{n_2=1}^{s_T} K_{h,1}\big(\frac{t}{T} - \frac{\lambda + n_1}{T}\big)K_{h,1}\big(\frac{t}{T} - \frac{\lambda'+n_2}{T}\big)\\
    &\hspace{6cm} \times \E\Big[\prod_{j=1}^d K_{h,2}(x^j - X_{\lambda + n_1,T}^j)(\mathds{1}_{Y_{\lambda + n_1,T}\leq y} - F_t^\star (y|\boldsymbol{x}))\Big] \\
    &\hspace{6cm} \times \E\Big[\prod_{j=1}^d K_{h,2}(x^j - X_{\lambda' + n_2,T}^j)(\mathds{1}_{Y_{\lambda' + n_2,T}\leq y} - F_t^\star (y|\boldsymbol{x}))\Big].
\end{align*}
Using Proposition \ref{Lemma: E of K2}.(\textit{iii}), 
\begin{eqnarray*}\lefteqn{
K_{h,1}\big(\frac{t}{T} - \frac{\lambda + n_1}{T}\big)\E\big[\prod_{j=1}^d K_{h,2}(x^j - X_{\lambda + n_1,T}^j)(\mathds{1}_{Y_{\lambda + n_1,T}\leq y} - F_t^\star (y|\boldsymbol{x}))\big] }\\&&\lesssim K_{h,1}\big(\frac{t}{T} - \frac{\lambda + n_1}{T}\big) \big( \frac{1}{T^\nu h^{\nu-1}}  + h^{d+1} + h^{d+3} \big),
\end{eqnarray*}
then
\begin{align*}
    \mathlarger{\mathsf S}_{32}^\Pi 
    &\lesssim \frac{1}{(Th^{d+1})^2} \sum_{\substack{l=0 \\ \quad \mathrlap{l\neq l'}}}^{v_T-1}\sum_{l'=0}^{v_T-1}\sum_{n_1 =1}^{s_T}\sum_{n_2=1}^{s_T} K_{h,1}\big(\frac{t}{T} - \frac{\lambda + n_1}{T}\big) K_{h,1}\big(\frac{t}{T} - \frac{\lambda'+n_2}{T}\big) \\
    &\hspace{8cm} \times \Big(\frac{1}{T^\nu h^{\nu-1}}  + h^{d+1} + h^{d+3}\Big)^2.
\end{align*}
Similarly, for $l\neq l'$, $|\lambda - \lambda' + n_1 - n_2| > r_T$, then
\begin{align}\label{eqn: small block 3_2}
    \mathlarger{\mathsf S}_{32}^\Pi 
    &\lesssim \frac{1}{T^2 h^{2d+2}} \Big(\frac{1}{T^\nu h^{\nu-1}}  + h^{d+1} + h^{d+3}\Big)^2 \sum_{\substack{m=1\\ \mathrlap{|m-m'|> r_T}}}^{T}\sum_{m'=1}^{T} K_{h,1}\big(\frac{t}{T} - \frac{m}{T}\big) K_{h,1}\big(\frac{t}{T} - \frac{m'}{T}\big)\nonumber\\
    &\leq \frac{1}{h^{2d}} \Big(\frac{1}{T^\nu h^{\nu-1}}  + h^{d+1} + h^{d+3}\Big)^2 \underbrace{\frac{1}{Th}\sum_{m=1}^{T} K_{h,1}\big(\frac{t}{T} - \frac{m}{T}\big)}_{\bigO(1)} \underbrace{\frac{1}{Th}\sum_{m'=1}^{T}  K_{h,1}\big(\frac{t}{T} - \frac{m'}{T}\big)}_{\bigO(1)}\nonumber\\
    &\lesssim \frac{1}{h^{2d}} \Big(\frac{1}{T^\nu h^{\nu-1}}  + h^{d+1} + h^{d+3}\Big)^2 
    \lesssim \frac{1}{h^{2d}} \Big( \frac{1}{T^{2\nu} h^{2(\nu - 1)}}   + h^{2(d+1)}  \Big)\nonumber\\
    &\lesssim \frac{1}{T^{2\nu} h^{2(d+\nu-1)}} + h^2,  
\end{align}
which goes to zero as $T\rightarrow \infty$ using Assumption \ref{Assumption: bandwidth}. Now, comparing (\ref{eqn: small block 1}), (\ref{eqn: small block 2_1}), (\ref{eqn: small block 2_2}), (\ref{eqn: small block 3_1}), and (\ref{eqn: small block 3_2}), we get
\begin{align}\label{eqn: small blocks overall order}
    \E[\Pi_{t,T}^2] 
    &\lesssim  \frac{1}{Th^{2(d + 1) - \frac{2}{p}(1 - \nu) }} + \frac{1}{T^{2\nu} h^{2(d+\nu-1)}} + h^2.
\end{align}\\

\noindent {\textcolor{magenta}{\underline{\it Step 3. Control of the remainder block.}}} 
Now, let us deal with $\E\big[\Xi_{t,T}^2\big]$. See that

\begin{align*}
    \E\big[\Xi_{t,T}^2\big]
        &= \frac{1}{(Th^{d+1})^2} \sum_{a = v_T(r_T+s_T)+1}^T K_{h,1}^2\big(\frac{t}{T} - \frac{a}{T}\big)\E\big[Z_{a,t,T}^2\big]\\
    &\qquad + \frac{1}{(Th^{d+1})^2}\sum_{\substack{a = v_T(r_T+s_T)+1\\ \qquad \qquad \quad \mathrlap{a\neq b}}}^T \sum_{b = v_T(r_T+s_T)+1}^T K_{h,1}\big(\frac{t}{T} - \frac{a}{T}\big)K_{h,1}\big(\frac{t}{T} - \frac{b}{T}\big)\\
    &\hspace{8cm} \times \E \big[Z_{a,t,T} Z_{b,t,T}\big].
\end{align*}
We can further expand this as
\begin{align*}
    \E\big[\Xi_{t,T}^2\big]
    &= \frac{1}{(Th^{d+1})^2} \sum_{a = v_T(r_T+s_T)+1}^T K_{h,1}^2\big(\frac{t}{T} - \frac{a}{T}\big)\E\big[Z_{a,t,T}^2\big]\\
    &\qquad + \frac{1}{(Th^{d+1})^2}\sum_{\substack{a = v_T(r_T+s_T)+1\\ \qquad \qquad \quad \mathrlap{a\neq b}}}^T \sum_{b = v_T(r_T+s_T)+1}^T K_{h,1}\big(\frac{t}{T} - \frac{a}{T}\big)K_{h,1}\big(\frac{t}{T} - \frac{b}{T}\big)\\
    &\hspace{8cm} \times \text{$\C$ov} \big(Z_{a,t,T}, Z_{b,t,T}\big)\\
    &\qquad \quad + \frac{1}{(Th^{d+1})^2}\sum_{\substack{a = v_T(r_T+s_T)+1\\ \qquad \qquad \quad \mathrlap{a\neq b}}}^T \sum_{b = v_T(r_T+s_T)+1}^T K_{h,1}\big(\frac{t}{T} - \frac{a}{T}\big)K_{h,1}\big(\frac{t}{T} - \frac{b}{T}\big)\\
    &\hspace{8cm} \times \E\big[Z_{a,t,T}\big] \E\big[ Z_{b,t,T}\big]\\
    &=: \mathlarger{\mathsf S}_{1}^\Xi + \mathlarger{\mathsf S}_{2}^\Xi + \mathlarger{\mathsf S}_{3}^\Xi.
\end{align*}

\noindent {\textcolor{magenta}{\underline{\it Step 3.1. Control of $\mathlarger{\mathsf S}_{1}^\Xi$.}}} Considering $\mathlarger{\mathsf S}_{1}^\Xi$, we have
\begin{align*}
    \mathlarger{\mathsf S}_{1}^\Xi     &= \frac{1}{(Th^{d+1})^2} \sum_{a = v_T(r_T+s_T)+1}^T K_{h,1}^2\big(\frac{t}{T} - \frac{a}{T}\big)\nonumber\\
    &\hspace{6cm} \times \E\Big[\prod_{j=1}^d K_{h,2}^2 (x^j - X_{a,T}^j)(\mathds{1}_{Y_{a,T}\leq y} - F_t^\star (y|\boldsymbol{x}))^2 \Big]\nonumber\\
    &\leq \frac{2 C_2^d }{(Th^{d+1})^2} \sum_{a = v_T(r_T+s_T)+1}^T K_{h,1}^2\big(\frac{t}{T} - \frac{a}{T}\big)\nonumber\\
    &\hspace{6cm} \times \E\Big[\prod_{j=1}^d K_{h,2} (x^j - X_{a,T}^j)\big| \mathds{1}_{Y_{a,T}\leq y} - F_t^\star (y|\boldsymbol{x})\big| \Big].
\end{align*}
Using Proposition \ref{Lemma: E of K2}.(\textit{iii}), we have
\begin{align}\label{eqn: remainder block 1}
    \mathlarger{\mathsf S}_{1}^\Xi 
    &\lesssim \frac{1}{T^2h^{2d+2}} \Big(\frac{1}{T^\nu h^{\nu-1}}   + h^{d+1} + h^{d+3}\Big) \sum_{a = v_T(r_T+s_T)+1}^T K_{h,1}^2\big(\frac{t}{T} - \frac{a}{T}\big) \nonumber\\
    &\leq \frac{C_1}{Th^{2d+1}} \Big(\frac{1}{T^\nu h^{\nu-1}}   + h^{d+1} + h^{d+3}\Big) \underbrace{\frac{1}{Th}\sum_{a = 1}^T K_{h,1}\big(\frac{t}{T} - \frac{a}{T}\big)}_{\bigO(1)} \nonumber\\
    &\lesssim \frac{1}{T h^{2d+1}}\Big(\frac{1}{T^\nu h^{\nu-1}}   + h^{d+1} + h^{d+3}\Big)\nonumber\\
    &\lesssim \frac{1}{T^{1+\nu} h^{2d+\nu}} + \frac{1}{T h^{d}}\nonumber\\
    &\lesssim \frac{1}{T h^{2d+\nu}}.
\end{align}
\noindent {\textcolor{magenta}{\underline{\it Step 3.2. Control of $\mathlarger{\mathsf S}_{2}^\Xi$.}}} Taking $\mathlarger{\mathsf S}_{2}^\Xi$ into account, we have
\begin{align*}
    \mathlarger{\mathsf S}_{2}^\Xi &= \frac{1}{(Th^{d+1})^2}\sum_{\substack{a = v_T(r_T+s_T)+1\\ \qquad \qquad \quad \mathrlap{a\neq b}}}^T \sum_{b = v_T(r_T+s_T)+1}^T K_{h,1}\big(\frac{t}{T} - \frac{a}{T}\big)K_{h,1}\big(\frac{t}{T} - \frac{b}{T}\big) \text{$\C$ov} \big(Z_{a,t,T}, Z_{b,t,T}\big)\\
    &= \frac{1}{(Th^{d+1})^2}\sum_{\substack{n_1 = 1\\ \qquad \quad \mathrlap{|n_1-n_2|>0}}}^{T-v_T(r_T+s_T)} \sum_{n_2 = 1}^{T-v_T(r_T+s_T)} K_{h,1}\big(\frac{t}{T} - \frac{\lambda+n_1}{T}\big)K_{h,1}\big(\frac{t}{T} - \frac{\lambda+n_2}{T}\big)\\
    &\hspace{8cm} \times \text{$\C$ov} \big(Z_{\lambda+n_1,t,T}, Z_{\lambda+n_2,t,T}\big),
\end{align*}
where $\lambda = v_T(r_T+s_T)$. Now, using (\ref{eqn: cov within blocks}), we have
\begin{align*}
    \lefteqn{K_{h,1}\big(\frac{t}{T} - \frac{\lambda+n_1}{T}\big)K_{h,1}\big(\frac{t}{T} - \frac{\lambda+n_2}{T}\big)\big|\text{$\C$ov}\big(Z_{\lambda+n_1,t,T},Z_{\lambda+n_2,t,T}\big)\big|}\\
    &\lesssim K_{h,1}\big(\frac{t}{T} - \frac{\lambda+n_1}{T}\big)K_{h,1}\big(\frac{t}{T} - \frac{\lambda+n_2}{T}\big) \Big(\frac{1}{T^\nu h^{\nu-1}}   + h^{d+1} + h^{d+3}\Big)^{\frac{2}{p}} \beta(|n_1 - n_2|)^{1-\frac{2}{p}}.
\end{align*}
Thus
\begin{align*}
    \mathlarger{\mathsf S}_{2}^\Xi &\lesssim \frac{1}{T^2h^{2(d+1)}}\Big(\frac{1}{T^\nu h^{\nu-1}}  + h^{d+1} + h^{d+3}\Big)^{\frac{2}{p}}  \sum_{\substack{n_1 = 1\\ \qquad \mathrlap{|n_1 - n_2|>0}}}^{T-v_T(r_T+s_T)} \sum_{n_2 = 1}^{T-v_T(r_T+s_T)} K_{h,1}\big(\frac{t}{T} - \frac{\lambda+n_1}{T}\big) \nonumber\\
    &\hspace{8cm} \times K_{h,1}\big(\frac{t}{T} - \frac{\lambda+n_2}{T}\big) \beta(|n_1 - n_2|)^{1-\frac{2}{p}}\nonumber\\
    &\leq \frac{C_1^2}{T^2h^{2(d+1)}}\Big(\frac{1}{T^\nu h^{\nu-1}}  + h^{d+1} + h^{d+3}\Big)^{\frac{2}{p}} \sum_{\substack{n_1 = 1\\ \qquad \mathrlap{|n_1 - n_2|>0}}}^{T-v_T(r_T+s_T)} \sum_{n_2 = 1}^{T-v_T(r_T+s_T)} \beta(|n_1 - n_2|)^{1-\frac{2}{p}}.
\end{align*}
Assumption \ref{Assumption: mixing} entails $\sum_{k=1}^{\infty} k^\zeta \beta(k)^{1-\frac{2}{p}}<\infty$. Moreover, letting $k = |n_1 - n_2|$ and $w_T = T - v_T(r_T + s_T)$ yields
\begin{align*}
    \sum_{\substack{n_1=1 \\  \mathrlap{|n_1-n_2|>0}}}^{w_T} \sum_{n_2=1}^{w_T} \beta(|n_1 - n_2|)^{1-\frac{2}{p}}
    &= \sum_{n_1=1}^{w_T} \Big( \sum_{n_2>n_1}^{w_T} \beta(n_2 - n_1)^{1-\frac{2}{p}} + \sum_{n_2<n_1}^{w_T} \beta(n_1 - n_2)^{1-\frac{2}{p}}\Big)\\
    &= \sum_{n_1=1}^{w_T}\sum_{k>0}^{w_T-n_1} \beta(k)^{1 - \frac{2}{p}} + \sum_{n_2=1}^{w_T}\sum_{k>0}^{w_T-n_2} \beta(k)^{1 - \frac{2}{p}}\\
    &= 2 \sum_{n=1}^{w_T}\sum_{k>0}^{w_T-n} \beta(k)^{1 - \frac{2}{p}}
    \leq 2w_T \sum_{k=1}^{w_T} \beta(k)^{1 - \frac{2}{p}}\\
    &\lesssim  w_T \sum_{k=1}^{w_T} k^\zeta \beta(k)^{1 - \frac{2}{p}}
    \leq w_T \sum_{k=1}^{\infty} k^\zeta \beta(k)^{1 - \frac{2}{p}},
\end{align*}
            since $\beta(k)\geq 0$ and $k^\zeta \geq 1$ for $\zeta > 1 - \frac{2}{p}$, where $p>2$. So
\begin{align}\label{eqn: remainder block 2}
    \mathlarger{\mathsf S}_{2}^\Xi 
    &\leq \frac{C_1^2 w_T }{T^2h^{2(d+1)}}\Big(\frac{1}{T^\nu h^{\nu-1}}  + h^{d+1} + h^{d+3}\Big)^{\frac{2}{p}} \sum_{k=1}^{\infty} k^\zeta \beta(k)^{1-\frac{2}{p}}\nonumber\\
    &\lesssim \frac{1}{T h^{2(d+1)}}\Big(\frac{1}{T^\nu h^{\nu-1}}  + h^{d+1} + h^{d+3}\Big)^{\frac{2}{p}}, \quad \text{since } w_T \ll T,\nonumber\\
    &=\Big( \frac{1}{T^{p} h^{2(d+1)p}}   \Big( \frac{1}{T^\nu h^{\nu-1}}  + h^{d+1} + h^{d+3}\Big)^2 \Big)^\frac{1}{p}
    \lesssim \Big(\frac{1}{T^{p} h^{2(d+1)p}}   \Big(\frac{1}{T^{2\nu} h^{2(\nu-1)}}   + h^{2(d+1)} \Big) \Big)^\frac{1}{p}\nonumber\\
    &\lesssim \Big(\frac{1}{T^{p+2\nu} h^{2(d + 1)p + 2(\nu - 1)}} + \frac{1}{T^p h^{2(d + 1)p - 2(d + 1)}}\Big)^\frac{1}{p}
    \lesssim \Big(\frac{1}{T^{p} h^{2(d + 1)p + 2(\nu - 1)}}\Big)^\frac{1}{p}\nonumber\\
    &\lesssim \frac{1}{T h^{2(d + 1) -\frac{2}{p}(1 - \nu) }}.
\end{align}
\noindent {\textcolor{magenta}{\underline{\it Step 3.3. Control of $\mathlarger{\mathsf S}_{3}^\Xi$.}}} Lastly, let us look at $\mathlarger{\mathsf S}_{3}^\Xi$. 
\begin{align*}
    \mathlarger{\mathsf S}_{3}^\Xi &= \frac{1}{(Th^{d+1})^2}\sum_{\substack{a = v_T(r_T+s_T)+1\\ \qquad \qquad \quad \mathrlap{a\neq b}}}^T \sum_{b = v_T(r_T+s_T)+1}^T K_{h,1}\big(\frac{t}{T} - \frac{a}{T}\big)K_{h,1}\big(\frac{t}{T} - \frac{b}{T}\big) \E\big[Z_{a,t,T}\big] \E\big[ Z_{b,t,T}\big]\nonumber\\
    &= \frac{1}{(Th^{d+1})^2}\sum_{\substack{n_1 = 1\\ \mathrlap{|n_1 - n_2|>0}}}^{w_T} \sum_{n_2 = 1}^{w_T} K_{h,1}\big(\frac{t}{T} - \frac{\lambda+n_1}{T}\big)K_{h,1}\big(\frac{t}{T} - \frac{\lambda+n_2}{T}\big)  \E\big[Z_{\lambda+n_1,t,T}\big] \E\big[Z_{\lambda+n_2,t,T}\big]\nonumber\\
    &= \frac{1}{(Th^{d+1})^2}\sum_{\substack{n_1 = 1\\ \mathrlap{|n_1 - n_2|>0}}}^{w_T} \sum_{n_2 = 1}^{w_T} K_{h,1}\big(\frac{t}{T} - \frac{\lambda+n_1}{T}\big)K_{h,1}\big(\frac{t}{T} - \frac{\lambda+n_2}{T}\big)\nonumber\\
    &\hspace{6cm} \times  \E\Big[\prod_{j=1}^d K_{h,2}(x^j-X_{\lambda+n_1,T}^j) (\mathds{1}_{Y_{\lambda+n_1,T}\leq y} - F_t^\star (y|\boldsymbol{x}))\Big]\nonumber\\
    &\hspace{6cm} \times \E\Big[ \prod_{j=1}^d K_{h,2}(x^j-X_{\lambda+n_2,T}^j) (\mathds{1}_{Y_{\lambda+n_2,T}\leq y} - F_t^\star (y|\boldsymbol{x}))\Big].
\end{align*}
Using Proposition \ref{Lemma: E of K2}.(\textit{iii}), for $i=1,2$, 
\begin{eqnarray*}
    \lefteqn{K_{h,1}\big(\frac{t}{T} - \frac{\lambda + n_i}{T}\big)\E\big[\prod_{j=1}^d K_{h,2}(x^j - X_{\lambda+n_i,T}^j)(\mathds{1}_{Y_{\lambda + n_i,T}\leq y} - F_t^\star (y|\boldsymbol{x}))\big]}\\&& \lesssim K_{h,1}\big(\frac{t}{T} - \frac{\lambda + n_i}{T}\big) \big( \frac{1}{T^\nu h^{\nu-1}}  + h^{d+1} + h^{d+3} \big),
\end{eqnarray*}
then
\begin{align}\label{eqn: remainder block 3}
    \mathlarger{\mathsf S}_{3}^\Xi 
    &\lesssim \frac{1}{T^2 h^{2d+2}} \Big(\frac{1}{T^{\nu} h^{\nu-1}}  + h^{d+1} + h^{d+3}\Big)^2 \sum_{\substack{n_1 = 1\\ \mathrlap{|n_1 - n_2|>0}}}^{w_T} \sum_{n_2 = 1}^{w_T} K_{h,1}\big(\frac{t}{T} - \frac{\lambda+n_1}{T}\big)K_{h,1}\big(\frac{t}{T} - \frac{\lambda+n_2}{T}\big)\nonumber\\
    &\leq \frac{C_1}{T h^{2d+1}} \Big(\frac{1}{T^{\nu} h^{\nu-1}}  + h^{d+1} + h^{d+3}\Big)^2 \underbrace{\frac{1}{Th}\sum_{a=1}^T K_{h,1}\big(\frac{t}{T} - \frac{a}{T}\big)}_{\bigO(1)}\nonumber\\
    &\lesssim  \frac{1}{Th^{2d+1}} \Big(\frac{1}{T^\nu h^{\nu-1}}  + h^{d+1} + h^{d+3}\Big)^2
    \lesssim \frac{1}{Th^{2d+1}} \Big( \frac{1}{T^{2\nu} h^{2(\nu - 1)}}  + h^{2(d+1)} \Big)\nonumber\\
    &\lesssim \frac{1}{T^{1+ 2\nu} h^{2(d+\nu)-1}} + \frac{h}{T}\nonumber\\ 
    &\lesssim  \frac{1}{T h^{2(d+\nu)-1}}.
\end{align}
Now, comparing (\ref{eqn: remainder block 1}), (\ref{eqn: remainder block 2}), and (\ref{eqn: remainder block 3}), we have
\begin{align}\label{eqn: remainder blocks overall order}
    \E[\Xi_{t,T}^2] \lesssim \frac{1}{T h^{2(d + 1) - \frac{2}{p}(1 - \nu) }}.
\end{align}
Therefore, following (\ref{eqn: big blocks overall order}), (\ref{eqn: small blocks overall order}), and (\ref{eqn: remainder blocks overall order}), we get
\begin{align*}    \E\big[Z_{t,T}^2\big]
    &=  \bigO\Big(\frac{1}{Th^{2(d + 1) -\frac{2}{p}(1 - \nu) }} + \frac{1}{T^{2\nu} h^{2(d+\nu-1)}} + h^2\Big).
\end{align*}

\end{proof}

\subsection{Proof of Theorem \ref{Theorem: convergence of EW1}}\label{appendix: proof of convergence of EW1}

Recall that $\pi_t^\star(\cdot|\boldsymbol{x})$ is the probability measure of the random variable $Y_{t,T}|\boldsymbol{X}_{t,T}=\boldsymbol{x}$ with conditional CDF $F^\star_{t}(y|\boldsymbol{x}) = \mathds{P}(Y_{t,T}\leq y|\boldsymbol{X}_{t,T}=\boldsymbol{x})$. Observe that, by the definition of $W_1$ given in (\ref{def:W1_cdf}),
\begin{align*}
\E[W_1(\hat{\pi}_t(\cdot|\boldsymbol{x}),\pi_t^\star(\cdot|\boldsymbol{x}))]
        &= \int_\R \E\big[\big|\hat{F}_t(y|\boldsymbol{x})-F^\star_{t}(y|\boldsymbol{x})\big|\big]\diff y,
\end{align*}
using Fubini's theorem. Now, using Definition \ref{def:CDF of pi-hat},
\begin{align*}    {\hat{F}_t(y|\boldsymbol{x})-F^\star_{t}(y|\boldsymbol{x})}    &= \frac{\sum_{a=1}^T K_{h,1}\big(\frac{t}{T} - \frac{a}{T}\big) \prod_{j=1}^d K_{h,2}(x^j - X_{a,T}^j) \mathds{1}_{Y_{a,T}\leq y} }{\sum_{a=1}^T K_{h,1}\big(\frac{t}{T} - \frac{a}{T}\big) \prod_{j=1}^d K_{h,2}(x^j - X_{a,T}^j)} - F^\star_{t}(y|\boldsymbol{x}).
\end{align*}
Then observe that
\begin{align}\label{eqn: Fhat - Fstar}
    {\hat{F}_t(y|\boldsymbol{x})-F^\star_{t}(y|\boldsymbol{x})}    &= \frac{\frac{1}{Th^{d+1}}\sum_{a=1}^T K_{h,1}\big(\frac{t}{T} - \frac{a}{T}\big) \prod_{j=1}^d K_{h,2}(x^j - X_{a,T}^j) \big[\mathds{1}_{Y_{a,T}\leq y} - F^\star_{t}(y|\boldsymbol{x})\big]}{\frac{1}{Th^{d+1}}\sum_{a=1}^T K_{h,1}\big(\frac{t}{T} - \frac{a}{T}\big) \prod_{j=1}^d K_{h,2}(x^j - X_{a,T}^j)}.
\end{align}
Further, by applying Cauchy-Schwarz inequality, we obtain
\begin{align}\label{eqn: EW1 cauchy_schwarz sums}
\lefteqn{\E[W_1(\hat{\pi}_t(\cdot|\boldsymbol{x}),\pi_t^\star(\cdot|\boldsymbol{x}))] = \int \E\big[\big|\hat{F}_t(y|\boldsymbol{x})-F^\star_{t}(y|\boldsymbol{x})\big|\big]\diff y}\nonumber\\ 
    &= \int \E\Big[\Big| \frac{\frac{1}{Th^{d+1}}\sum_{a=1}^T K_{h,1}\big(\frac{t}{T} - \frac{a}{T}\big) \prod_{j=1}^d K_{h,2}(x^j - X_{a,T}^j) \big[\mathds{1}_{Y_{a,T}\leq y} - F^\star_{t}(y|\boldsymbol{x})\big]}{\frac{1}{Th^{d+1}}\sum_{a=1}^T K_{h,1}\big(\frac{t}{T} - \frac{a}{T}\big) \prod_{j=1}^d K_{h,2}(x^j - X_{a,T}^j)}\Big| \Big] \diff y\nonumber\\
    &\leq \int \Big( \E\Big[\Big(\frac{1}{\frac{1}{Th^{d+1}}\sum_{a=1}^T K_{h,1}\big(\frac{t}{T} - \frac{a}{T}\big) \prod_{j=1}^d K_{h,2}(x^j - X_{a,T}^j)}\Big)^2\Big]\Big)^{\frac{1}{2}}\nonumber\\
    &\quad \times \Big(\E\Big[\Big(\frac{1}{Th^{d+1}}\sum_{a=1}^T K_{h,1}\big(\frac{t}{T} - \frac{a}{T}\big) \prod_{j=1}^d K_{h,2}(x^j - X_{a,T}^j) \big[\mathds{1}_{Y_{a,T}\leq y} - F^\star_{t}(y|\boldsymbol{x})\big]  \Big)^2\Big]\Big)^{\frac{1}{2}} \diff y.\nonumber\\
\end{align}
Let $J_{t,T}(\frac{t}{T}, \boldsymbol{x}) = \frac{1}{Th^{d+1}}\sum_{a=1}^T K_{h,1}\big(\frac{t}{T} - \frac{a}{T}\big) \prod_{j=1}^d K_{h,2}(x^j - X_{a,T}^j)$. Using Proposition \ref{lemma: J1 is Op(1)}, the first term in  (\ref{eqn: EW1 cauchy_schwarz sums}) becomes
\begin{align}\label{eqn: J inv Op1}
    \Big(\E\Big[\Big(\frac{1}{\frac{1}{Th^{d+1}}\sum_{a=1}^T K_{h,1}\big(\frac{t}{T} - \frac{a}{T}\big) \prod_{j=1}^d K_{h,2}(x^j - X_{a,T}^j)}\Big)^2\Big]\Big)^{\frac{1}{2}}
    &= \bigO(1).
\end{align} 
Additionally, using Proposition \ref{prop: control of square of sums}, the second term is of order $\bigO\big(\frac{1}{T^{\frac{1}{2}} h^{d + 1 - \frac{1}{p}(1 - \nu)}} + \frac{1}{T^\nu h^{d + \nu - 1}} + h\big)$. Therefore, from (\ref{eqn: EW1 cauchy_schwarz sums}), we have
\begin{align*}
    \E\big[W_1\big(\hat{\pi}_t(\cdot|\boldsymbol{x}), \pi_t^\star(\cdot|\boldsymbol{x})\big)\big]
    = \bigO\Big(\frac{1}{T^{\frac{1}{2}} h^{d + 1 - \frac{1}{p}(1 - \nu)}} + \frac{1}{T^\nu h^{d + \nu - 1}} + h\Big),
\end{align*}
where $\nu=\rho \wedge 1$ and $p>2$.

\subsection{Proof of Corollary \ref{Remark: bound EW_s-s}}
\label{appendix: proof of expectation of W_r}

Using the definition of $W_1$ and noting that $y\in [-M,M]$, we have
\begin{align*}
    W_r^r(\hat{\pi}_t(\cdot|\boldsymbol{x}),\pi_t^\star(\cdot|\boldsymbol{x}))
                            &\leq (2M)^{r-1} \int_{-M}^M |\hat{F}_t(y|\boldsymbol{x}) - F_t^\star (y|\boldsymbol{x})| \diff y.
\end{align*}
This gives
\begin{align*}
    \E[W_r^r(\hat{\pi}_t(\cdot|\boldsymbol{x}),\pi_t^\star(\cdot|\boldsymbol{x}))]
    &\leq (2M)^{r-1} \E\Big[ \int_{-M}^M |\hat{F}_t(y|\boldsymbol{x}) - F_t^\star (y|\boldsymbol{x})| \diff y \Big]\\
    &\leq (2M)^{r-1} \E[W_1(\hat{\pi}_t(\cdot|\boldsymbol{x}),\pi_t^\star(\cdot|\boldsymbol{x}))].
\end{align*}
By Theorem \ref{Theorem: convergence of EW1}, we get the desired result in Corollary \ref{Remark: bound EW_s-s}. 

\subsection{Proof of Corollary \ref{corollary: convergence of the 2nd moment}}\label{appendix: proof of convergence of the 2nd moment}

We use the definition of $W_1$ given by (\ref{def:W1_cdf}) and Minkowski's integral inequality given by, for any $r\geq 1$,
\begin{align*}    {\Big\|\int \big|\hat{F}_t(y|\boldsymbol{x})-F^\star_{t}(y|\boldsymbol{x})\big|\diff y\Big\|}_{L_r} \leq \int {\big\|\hat{F}_t(y|\boldsymbol{x})-F^\star_{t}(y|\boldsymbol{x})\big\|}_{L_r} \diff y.
\end{align*}
By (\ref{def:W1_cdf}),
\begin{align*}    \norm{W_1\big(\hat{\pi}_t(\cdot|\boldsymbol{x}), \pi_t^\star(\cdot|\boldsymbol{x})\big)}_{L_2}
    &= {\Big\|\int_\R \big|\hat{F}_t(y|\boldsymbol{x})-F^\star_{t}(y|\boldsymbol{x})\big|\diff y\Big\|}_{L_2}.
\end{align*}
So for $r=2$, we have
\begin{align*}    \norm{W_1\big(\hat{\pi}_t(\cdot|\boldsymbol{x}), \pi_t^\star(\cdot|\boldsymbol{x})\big)}_{L_2}   
    &\leq \int_\R {\big\|\hat{F}_t(y|\boldsymbol{x})-F^\star_{t}(y|\boldsymbol{x})\big\|}_{L_2} \diff y \nonumber\\
    &= \int_\R \big( \E\big[\big(\hat{F}_t(y|\boldsymbol{x})-F^\star_{t}(y|\boldsymbol{x}) \big)^2 \big]\big)^\frac{1}{2} \diff y \nonumber\\
                &= \int_\R \Big( \E \Big[ \Big( \frac{Z_{t,T}(y, \boldsymbol{x})}{J_{t,T}(\frac{t}{T}, \boldsymbol{x})} \Big)^2 \Big] \Big)^\frac{1}{2} \diff y,
\end{align*}
using (\ref{eqn: Fhat - Fstar}) and (\ref{eqn: Z tT}). However, using Proposition \ref{lemma: J1 is Op(1)}, $J_{t,T}^{-1}(\frac{t}{T}, \boldsymbol{x}) = \bigO(1)$. So
\begin{align*}
    \norm{W_1\big(\hat{\pi}_t(\cdot|\boldsymbol{x}), \pi_t^\star(\cdot|\boldsymbol{x})\big)}_{L_2}
    &\lesssim \int_\R \big( \E\big[ Z_{t,T}^2(y, \boldsymbol{x}) \big] \big)^\frac{1}{2} \diff y \\
    &\lesssim \int_\R \Big( \frac{1}{Th^{2(d + 1) -\frac{2}{p}(1 - \nu) }} + \frac{1}{T^{2\nu} h^{2(d+\nu-1)}} + h^2 \Big)^\frac{1}{2} \diff y,
\end{align*}
by Proposition \ref{prop: control of square of sums}. Therefore,
\begin{align*}
    \norm{W_1\big(\hat{\pi}_t(\cdot|\boldsymbol{x}), \pi_t^\star(\cdot|\boldsymbol{x})\big)}_{L_2}
    &= \bigO\Big(\frac{1}{T^{\frac{1}{2}} h^{(d + 1) - \frac{1}{p}(1 - \nu)}} + \frac{1}{T^\nu h^{d + \nu - 1}} + h\Big),
\end{align*}
where $\nu=\rho \wedge 1$ and $p>2$.

\subsection{Proof of Proposition \ref{prop: |mhat-m| leq W1}}
\label{appendix: proof of mhat leq W1}

Observe that
    \begin{align*}
        |\hat{m}(\frac{t}{T},\boldsymbol{x}) - m^\star(\frac{t}{T}, \boldsymbol{x})|
        &= |\E[\hat{Y}_{t,T}|X_{t,T} = \boldsymbol{x}] - \E[Y_{t,T}|X_{t,T}=\boldsymbol{x}]|\\
                &= \Big|\int_\R \hat{y}\diff \hat{\pi}_t(\cdot|\boldsymbol{x}) - \int_\R y \diff \pi^\star_t(\cdot|\boldsymbol{x})) \Big|\\
                &\leq \sup_{f\in\mathcal{F}} \Big|\int_\R f \diff \hat{\pi}_t(\cdot|\boldsymbol{x}) - \int_\R f \diff \pi^\star_t(\cdot|\boldsymbol{x})) \Big|\\
                &= W_1(\hat{\pi}_t(\cdot|\boldsymbol{x}), \pi^\star_t(\cdot|\boldsymbol{x})).
    \end{align*}
In the last equality, we use duality formula of Kantorovich-Rubinstein distance (see Remark 6.5 in \cite{Villani2009OToldnew}), where $\mathcal{F}$ is the set of all continuous functions satisfying Lipschitz condition $\norm{f}_{Lip}\leq 1$, i.e., $\sup_{y\neq y'} \frac{|f(y) - f(y')|}{|y - y'|} \leq 1$. Hence,
    \begin{align*}
        \E \big[|\hat{m}(\frac{t}{T},\boldsymbol{x}) - m^\star(\frac{t}{T}, \boldsymbol{x})| \big]
        \leq \E \big[ W_1(\hat{\pi}_t(\cdot|\boldsymbol{x}), \pi^\star_t(\cdot|\boldsymbol{x}))\big].
    \end{align*}
This finishes the proof.

\subsection{Proof of Proposition \ref{prop: convergence of EW1 chosen h}}\label{appendix: proof of convergence of EW1 chosen h}

If $h = \bigO(T^{-\xi})$, then directly from Theorem \ref{Theorem: convergence of EW1}, for $\nu=\rho \wedge 1$, we get
\begin{align*}
    \E\big[W_1\big(\hat{\pi}_t(\cdot|\boldsymbol{x}), \pi_t^\star(\cdot|\boldsymbol{x})\big)\big]
    &\lesssim \frac{1}{T^{\frac{1}{2}} h^{(d + 1) - \frac{1}{p}(1 - \nu)}} + \frac{1}{T^\nu h^{d + \nu - 1}} + h\\
    &\lesssim \frac{1}{T^{\frac{1}{2}} T^{-\xi((d + 1) - \frac{1}{p}(1 - \nu))}} + \frac{1}{T^\nu T^{-\xi(d + \nu - 1)}} + \frac{1}{T^\xi} \\
    &= \bigO \Big( \frac{1}{T^{\frac{1}{2} -\xi((d + 1) - \frac{1}{p}(1 - \nu))}} + \frac{1}{T^{\nu -\xi(d + \nu - 1)}} + \frac{1}{T^\xi} \Big).
\end{align*}
Note that, as $T\rightarrow \infty$, the third component goes to zero for any $\xi >0$. Additionally, the second component converges to zero when $\xi < \frac{\nu}{d+\nu-1}$, which suggests that $\xi < \frac{\nu}{d+1}$. Lastly, the first component approaches zero if $\xi < \frac{1}{2(d+1 - \frac{1}{p}(1-\nu))}$, which further implies that $\xi < \frac{1}{2(d+1)}$ since $\nu = \rho \wedge 1$ and $p>2$. Therefore, for $h = \bigO(T^{-\xi})$, $\E\big[W_1\big(\hat{\pi}_t(\cdot|\boldsymbol{x}), \pi_t^\star(\cdot|\boldsymbol{x})\big)\big]$ converges to zero if
\begin{align*}
    \xi < \begin{cases}
        \frac{\nu}{d+1} & \text{if } \nu < \frac{1}{2},\\
        \frac{1}{2(d+1)} & \text{otherwise}.
    \end{cases}
\end{align*}
As a consequence, $\xi < \frac{\frac{1}{2} \wedge \nu}{d+1}$.

\subsection{Proof of Theorem \ref{Theorem: convergence of ESW1_multivariate Y}}\label{appendix: proof of convergence of ESW1_multivariate Y}

Observe that using (\ref{def: sliced W}) and by Fubini's theorem, we have
\begin{align*}
    \E[SW_1(\hat{\boldsymbol{\pi}}_t(\cdot|\boldsymbol{x}), \boldsymbol{\pi}^\star_t(\cdot|\boldsymbol{x}))]
        &= \int_{\mathds{S}^{q-1}} \E \big[ W_1(\boldsymbol{\theta}_\#\hat{\boldsymbol{\pi}}_t(\cdot|\boldsymbol{x}), \boldsymbol{\theta}_\#\boldsymbol{\pi}^\star_t(\cdot|\boldsymbol{x}))\big] \sigma_{q-1}(\diff {\boldsymbol{\theta}}).
\end{align*}
On the other hand,
\begin{align*}
    \E \big[ W_1(\boldsymbol{\theta}_\#\hat{\boldsymbol{\pi}}_t(\cdot|\boldsymbol{x}), \boldsymbol{\theta}_\#\boldsymbol{\pi}^\star_t(\cdot|\boldsymbol{x}))\big]
    &= \E\big[\int_\R \big|\hat{F}_{t, \boldsymbol{\theta}}(y|\boldsymbol{x}) - F_{t, \boldsymbol{\theta}}^\star(y|\boldsymbol{x})\big|\diff y \big] \\
    &= \int_\R \E\big[\big|\hat{F}_{t, \boldsymbol{\theta}}(y|\boldsymbol{x}) - F_{t, \boldsymbol{\theta}}^\star(y|\boldsymbol{x})\big|\big]\diff y.
\end{align*}
Using (\ref{def: weights}) and (\ref{eqn: CDF of projected pi-hat}),
\begin{align*}
    \hat{F}_{t, \boldsymbol{\theta}}(y|\boldsymbol{x}) - F_{t, \boldsymbol{\theta}}^\star(y|\boldsymbol{x})
    &= \frac{\sum_{a=1}^T K_{h,1}\big(\frac{t}{T} - \frac{a}{T}\big) \prod_{j=1}^d K_{h,2}(x^j - X_{a,T}^j) \mathds{1}_{{\boldsymbol{\theta}}^\top \boldsymbol{Y}_{a,T} \leq y} }{\sum_{a=1}^T K_{h,1}\big(\frac{t}{T} - \frac{a}{T}\big) \prod_{j=1}^d K_{h,2}(x^j - X_{a,T}^j)} - F_{t, \boldsymbol{\theta}}^\star(y|\boldsymbol{x})\\
    &= \frac{\frac{1}{Th^{d+1}}\sum_{a=1}^T K_{h,1}\big(\frac{t}{T} - \frac{a}{T}\big) \prod_{j=1}^d K_{h,2}(x^j - X_{a,T}^j) \big[\mathds{1}_{{\boldsymbol{\theta}}^\top \boldsymbol{Y}_{a,T} \leq y} - F_{t, \boldsymbol{\theta}}^\star(y|\boldsymbol{x})\big]}{\frac{1}{Th^{d+1}}\sum_{a=1}^T K_{h,1}\big(\frac{t}{T} - \frac{a}{T}\big) \prod_{j=1}^d K_{h,2}(x^j - X_{a,T}^j)}.
\end{align*}
Further, by applying Cauchy-Schwarz inequality, we obtain
\begin{align}\label{eqn: EW1 cauchy_schwarz sums projections}
    \lefteqn{\E \big[ W_1(\boldsymbol{\theta}_\#\hat{\boldsymbol{\pi}}_t(\cdot|\boldsymbol{x}), \boldsymbol{\theta}_\#\boldsymbol{\pi}^\star_t(\cdot|\boldsymbol{x}))\big]}\nonumber\\
    &= \int_\R \E\Big[\Big| \frac{\frac{1}{Th^{d+1}}\sum_{a=1}^T K_{h,1}\big(\frac{t}{T} - \frac{a}{T}\big) \prod_{j=1}^d K_{h,2}(x^j - X_{a,T}^j) \big[\mathds{1}_{{\boldsymbol{\theta}}^\top \boldsymbol{Y}_{a,T} \leq y} - F_{t, \boldsymbol{\theta}}^\star(y|\boldsymbol{x})\big]}{\frac{1}{Th^{d+1}}\sum_{a=1}^T K_{h,1}\big(\frac{t}{T} - \frac{a}{T}\big) \prod_{j=1}^d K_{h,2}(x^j - X_{a,T}^j)} \Big| \Big] \diff y\nonumber\\
    &\leq \int_\R \Big( \E\Big[\Big(\frac{1}{\frac{1}{Th^{d+1}}\sum_{a=1}^T K_{h,1}\big(\frac{t}{T} - \frac{a}{T}\big) \prod_{j=1}^d K_{h,2}(x^j - X_{a,T}^j)}\Big)^2\Big]\Big)^{\frac{1}{2}}\nonumber\\
    &\quad \times \Big(\E\Big[\Big(\frac{1}{Th^{d+1}}\sum_{a=1}^T K_{h,1}\big(\frac{t}{T} - \frac{a}{T}\big) \prod_{j=1}^d K_{h,2}(x^j - X_{a,T}^j) \big[\mathds{1}_{{\boldsymbol{\theta}}^\top \boldsymbol{Y}_{a,T} \leq y} - F_{t, \boldsymbol{\theta}}^\star(y|\boldsymbol{x})\big]  \Big)^2\Big]\Big)^{\frac{1}{2}} \diff y.\nonumber\\
\end{align}
Note that from Proposition \ref{lemma: J1 is Op(1)}, the first term in (\ref{eqn: EW1 cauchy_schwarz sums projections}) is $\bigO(1)$. Moreover, it can be observed that inequality (\ref{eqn: EW1 cauchy_schwarz sums projections}) is similar to inequality (\ref{eqn: EW1 cauchy_schwarz sums}). Hence, using similar steps in the proof of Proposition \ref{prop: control of square of sums}, we again use Bernstein's big-block and small-block procedure and consider (\ref{eqn: bernstein blocking}) with $Z_{a,t,T} = \prod_{j=1}^d K_{h,2}(x^j - X_{a,T}^j) \big[\mathds{1}_{{\boldsymbol{\theta}}^\top \boldsymbol{Y}_{a,T} \leq y} - F_{t, \boldsymbol{\theta}}^\star(y|\boldsymbol{x})\big]$. Additionally, by Assumption \ref{assumption: CDF multivariate case} and Proposition \ref{Lemma: E of K2}.\textit{(iii)},
\begin{align*}
    K_{h,1}\big(\frac{t}{T} - \frac{a}{T}\big) &\E\Big[\prod_{j=1}^d K_{h,2}(x^j - X_{a,T}^j)(\mathds{1}_{{\boldsymbol{\theta}}^\top \boldsymbol{Y}_{a,T} \leq y} - F_{t, \boldsymbol{\theta}}^\star(y|\boldsymbol{x}))\Big]\\ 
    &\lesssim K_{h,1}\big(\frac{t}{T} - \frac{a}{T}\big) \big(\frac{1 }{T^\nu h^{\nu-1}}  + h^{d+1} + h^{d+3}\big).
\end{align*}
The rest of the proof follows directly from the proof of Theorem \ref{Theorem: convergence of EW1}. Accordingly, using Proposition \ref{prop: control of square of sums}, we have

                                \begin{align}\label{eqn: bound E of S square projection}
    \E\big[\big(Z_{t,T}\big)^2\big]
    \lesssim  \frac{1}{Th^{2(d + 1) -\frac{2}{p}(1 - \nu) }} + \frac{1}{T^{2\nu} h^{2(d+\nu-1)}} + h^2.
\end{align}
Furthermore, from (\ref{eqn: EW1 cauchy_schwarz sums projections}), and incorporating (\ref{eqn: J inv Op1}) and (\ref{eqn: bound E of S square projection}), we have
\begin{align*}
    \E \big[ W_1(\boldsymbol{\theta}_\#\hat{\boldsymbol{\pi}}_t(\cdot|\boldsymbol{x}), \boldsymbol{\theta}_\#\boldsymbol{\pi}^\star_t(\cdot|\boldsymbol{x}))\big]
    = \bigO\Big(\frac{1}{T^{\frac{1}{2}} h^{d + 1 - \frac{1}{p}(1 - \nu)}} + \frac{1}{T^\nu h^{d + \nu - 1}} + h\Big).
\end{align*}
Therefore, 
\begin{align*}
    \E[SW_1(\hat{\boldsymbol{\pi}}_t(\cdot|\boldsymbol{x}), \boldsymbol{\pi}^\star_t(\cdot|\boldsymbol{x}))] = \bigO\Big(\frac{1}{T^{\frac{1}{2}} h^{d + 1 - \frac{1}{p}(1 - \nu)}} + \frac{1}{T^\nu h^{d + \nu - 1}} + h\Big),
\end{align*}
where $\nu=\rho \wedge 1$.

\section{Useful lemmas} \label{sec:proofs}

\begin{lemma}\label{lemma: Prod of K2 leq Sum}
    Let Assumption \ref{Assumption: kernel functions} hold, then
    \begin{enumerate}[label=(\roman*)]
    \item  $ {\Big| \prod_{j=1}^d K_{h,2} (x^j - X_{a,T}^j) - \prod_{j=1}^d K_{h,2} \big(x^j - X_a^j\big(\frac{a}{T}\big)\big) \Big|}\\ 
        \leq C_2^{d-1} \sqrt{d} \sum_{j=1}^d \big|  K_{h,2} (x^j - X_{a,T}^j) -  K_{h,2} \big(x^j - X_a^j\big(\frac{a}{T}\big)\big) \big|.$
    \item  ${\Big| \prod_{j=1}^d K_{h,2}^2(x^j - X_{a,T}^j) - \prod_{j=1}^d K_{h,2}^2 \big(x^j - X_a^j\big(\frac{a}{T}\big)\big) \Big|}\\ 
        \,\,\,\,\, \leq C_2^{2d-2} \sqrt{d} \sum_{j=1}^d \big|  K_{h,2} (x^j - X_{a,T}^j) -  K_{h,2} \big(x^j - X_a^j\big(\frac{a}{T}\big)\big) \big|.$
    \item 
    for $p\geq2$, $\E\Big[\Big|\prod_{j=1}^d K_{h,2}(x^j - X_{a,T}^j)\Big|^p\Big]
        \leq C_2^{d(p-1)} \E\Big[\Big|\prod_{j=1}^d K_{h,2}(x^j - X_{a,T}^j)\Big|\Big].$
    \end{enumerate}
\end{lemma}
\begin{proof}
    For (\textit{i}), let $g^j = K_{h,2}(x^j - X_{a,T}^j)$ and $\widetilde{g}^j = K_{h,2} \big(x^j - X_a^j\big(\frac{a}{T}\big)\big)$. Let $G(g^1,\ldots, g^d) = \prod_{j=1}^d g^j$.     The gradient of $G(g^1,\ldots, g^d)$ can be written as
    \begin{align*}
        \nabla G(g^1,\ldots, g^d) = 
        \begin{bmatrix}
            \frac{\partial G(g^1,\ldots, g^d)}{\partial g^1}\\
            \frac{\partial G(g^1,\ldots, g^d)}{\partial g^2}\\
            \vdots \\ 
            \frac{\partial G(g^1,\ldots, g^d)}{\partial g^d}
        \end{bmatrix}
        =
        \begin{bmatrix}
            \prod_{j=2}^d g^j\\
            \prod_{j=1; j\neq 2}^d g^j\\
            \vdots \\
            \prod_{j=1}^{d-1} g^j
        \end{bmatrix}.
    \end{align*}
    In addition, by Assumption ({\bf A2}-i), $K_2$ is bounded by $C_2$, so
    \begin{align*}
        \norm{\nabla G(g^1,\ldots, g^d)} &= \sqrt{\Big(\prod_{j=2}^d g^j\Big)^2 + \Big(\prod_{j=1; j\neq 2}^d g^j\Big)^2 + \cdots +  \Big(\prod_{j=1}^{d-1} g^j\Big)^2} \\
        &\leq \sqrt{(C_2^{d-1})^2 + \cdots + (C_2^{d-1})^2}
        = \sqrt{d(C_2^{d-1})^2} = C_2^{d-1} \sqrt{d}.
    \end{align*}
    Now,
    \begin{align*}
        |G(g^1,\ldots,g^d) - G(\widetilde{g}^1,\ldots, \widetilde{g}^j)|
        &\leq C_2^{d-1} \sqrt{d} \norm{(g^1,\ldots,g^d) - (\widetilde{g}^1,\ldots, \widetilde{g}^j)}_2\\
        &=  C_2^{d-1} \sqrt{d} \sqrt{\sum_{j=1}^d \big( K_{h,2}(x^j - X_{a,T}^j) - K_{h,2} \big(x^j - X_a^j\big(\frac{a}{T}\big) \big) \big)^2 }\\
        &\leq C_2^{d-1} \sqrt{d} \sum_{j=1}^d \big|K_{h,2}(x^j - X_{a,T}^j) - K_{h,2} \big(x^j - X_a^j\big(\frac{a}{T}\big) \big)\big|,
    \end{align*}
    since for $d$-dimensional vector $\boldsymbol{z}$, $\norm{\boldsymbol{z}}_2 \leq \norm{\boldsymbol{z}}_1$.
    So,
    \begin{align*}
        \lefteqn{\Big| \prod_{j=1}^d K_{h,2}(x^j - X_{a,T}^j) - \prod_{j=1}^d K_{h,2} \big(x^j - X_a^j\big(\frac{a}{T}\big) \big) \Big|}\\
        &\leq C_2^{d-1} \sqrt{d}  \sum_{j=1}^d \big|K_{h,2}(x^j - X_{a,T}^j) - K_{h,2} \big(x^j - X_a^j\big(\frac{a}{T}\big) \big)\big|.
    \end{align*}
    Similarly, to show (\textit{ii}), we let $g^j = K_{h,2}^2(x^j - X_{a,T}^j)$ and $\widetilde{g}^j = K_{h,2}^2 \big(x^j - X_a^j\big(\frac{a}{T}\big)\big)$. Let $G(g^1,\ldots, g^d) = \prod_{j=1}^d g^j$.     Using the gradient of $G(g^1,\ldots, g^d)$ given in (\textit{i}) and noting that $K_2^2(\cdot)$ is bounded by $C_2^2$, so
    \begin{align*}
        \norm{\nabla G(g^1,\ldots, g^d)} &= \sqrt{\Big(\prod_{j=2}^d g^j\Big)^2 + \Big(\prod_{j=1; j\neq 2}^d g^j\Big)^2 + \cdots +  \Big(\prod_{j=1}^{d-1} g^j\Big)^2} \\
        &\leq \sqrt{(C_2^{2d-2})^2 + \cdots + (C_2^{2d-2})^2}
        = \sqrt{d(C_2^{2d-2})^2} = C_2^{2d-2} \sqrt{d}.
    \end{align*}
    Now,
    \begin{align*}
        |G(g^1,\ldots,g^d) - G(\widetilde{g}^1,\ldots, \widetilde{g}^j)|
        &\leq C_2^{2d-2} d^{\frac{1}{2}} \norm{(g^1,\ldots,g^d) - (\widetilde{g}^1,\ldots, \widetilde{g}^j)}_2\\
        &=  C_2^{2d-2} \sqrt{d} \sqrt{\sum_{j=1}^d \big( K_{h,2}(x^j - X_{a,T}^j) - K_{h,2} \big(x^j - X_a^j\big(\frac{a}{T}\big) \big) \big)^2 }\\
        &\leq C_2^{2d-2} \sqrt{d} \sum_{j=1}^d \big|K_{h,2}(x^j - X_{a,T}^j) - K_{h,2} \big(x^j - X_a^j\big(\frac{a}{T}\big) \big)\big|,
    \end{align*}
    since for $d$-dimensional vector $\boldsymbol{z}$, $\norm{\boldsymbol{z}}_2 \leq \norm{\boldsymbol{z}}_1$.
    So,
    \begin{align*}
        \lefteqn{\Big| \prod_{j=1}^d K_{h,2}(x^j - X_{a,T}^j) - \prod_{j=1}^d K_{h,2} \big(x^j - X_a^j\big(\frac{a}{T}\big) \big) \Big|}\\
        &\leq C_2^{2d-2} \sqrt{d}  \sum_{j=1}^d \big|K_{h,2}(x^j - X_{a,T}^j) - K_{h,2} \big(x^j - X_a^j\big(\frac{a}{T}\big) \big)\big|.
    \end{align*}
    To show (\textit{ii}), we again use the boundedness of $K_2$, so
    \begin{align*}
        \E\Big[\Big|\prod_{j=1}^d K_{h,2}(x^j - X_{a,T}^j)\Big|^p\Big]
        &\leq \E\Big[\max_{j=1,\ldots,d} \Big|\prod_{j=1}^d K_{h,2}(x^j - X_{a,T}^j)\Big|^{p-1} \Big|\prod_{j=1}^d K_{h,2}(x^j - X_{a,T}^j)\Big|\Big]\\
        &\leq C_2^{d(p-1)} \E\Big[\Big|\prod_{j=1}^d K_{h,2}(x^j - X_{a,T}^j)\Big|\Big].
    \end{align*}
\end{proof}

\begin{lemma}\label{lemma: beta l_l'}
    For $l\neq l'$, $\beta(\sigma(\boldsymbol{X}_{l,T}),\sigma(\boldsymbol{X}_{l',T})) \leq \beta(|l-l'|)$, where $\sigma(X)$ denotes the $\sigma$-algebra generated by $X$.
\end{lemma}
\begin{proof}
    Let us start the proof by first considering the case $l>l'$, that is
    \begin{align*}
        \beta(\sigma(\boldsymbol{X}_{l,T}),\sigma(\boldsymbol{X}_{l',T}))
        &\leq \beta(\sigma(\boldsymbol{X}_{s,T}, s\geq l), \sigma(\boldsymbol{X}_{s,T}, s\leq l'))\\
        &= \beta(\sigma(\boldsymbol{X}_{s,T}, s\leq l'), \sigma(\boldsymbol{X}_{s,T}, l\leq s))\\
        &\leq \sup_{t} \beta(\sigma(\boldsymbol{X}_{s,T}, s\leq t), \sigma(\boldsymbol{X}_{s,T}, t+l-l'\leq s\leq T)), \quad \text{by letting }t=l'\\
        &\leq \sup_{t,T: t \leq T-|l-l'|} \beta(\sigma(\boldsymbol{X}_{s,T}, s\leq t), \sigma(\boldsymbol{X}_{s,T}, t+|l-l'|\leq s\leq T))\\
        &= \beta(|l-l'|).
    \end{align*}
    The last inequality holds since $t+|l-l'|\leq T$, which implies $t\leq T - |l-l'|$. Now let us see the case $l'>l$. Observe that
    
    \begin{align*}
        \beta(\sigma(\boldsymbol{X}_{l,T}),\sigma(\boldsymbol{X}_{l',T}))
        &\leq \beta(\sigma(\boldsymbol{X}_{s,T}, s\geq l'), \sigma(\boldsymbol{X}_{s,T}, s\leq l))\\
        &= \beta(\sigma(\boldsymbol{X}_{s,T}, s\leq l), \sigma(\boldsymbol{X}_{s,T}, l'\leq s))\\
        &\leq \sup_{t} \beta(\sigma(\boldsymbol{X}_{s,T}, s\leq t), \sigma(\boldsymbol{X}_{s,T}, t+l'-l\leq s\leq T)), \quad \text{by letting }t=l\\
        &\leq \sup_{t,T: t \leq T-|l'-l|} \beta(\sigma(\boldsymbol{X}_{s,T}, s\leq t), \sigma(\boldsymbol{X}_{s,T}, t+|l'-l|\leq s\leq T))\\
        &= \beta(|l-l'|).
    \end{align*}
    Again, the last inequality holds since $t+|l'-l|\leq T$, which implies $t\leq T - |l'-l|$.
\end{proof}

\begin{lemma}[\cite{Davydov1973}] \label{lemma: Davydovs}
    Suppose that $X$ and $Y$ are random variables which are $\mathscr{G}$ and $\mathscr{H}$-measurable, respectively, and that $\E[|X|^p]<\infty$, $\E[|Y|^{p'}]<\infty$, where $p,p' >1$, $p^{-1}+{p'}^{-1}<1$. Then
    \begin{align*}
        |\C ov(X,Y)| \leq 8\norm{X}_{L_p}\norm{Y}_{L_p'}[\beta(\mathscr{G},\mathscr{H})]^{1-p^{-1}-{p'}^{-1}}.
    \end{align*}
\end{lemma}

\begin{lemma}[\cite{vogt2012}, Lemma B.2]\label{Lemma: sup K1 - g}
    Suppose $K$ fulfills Assumption \ref{Assumption: kernel functions} and let $g:[0,1]\times \R^d \rightarrow \R$, $(u,x) \mapsto g(u,x)$ be continuously differentiable wrt $u$. Then for any compact set $S\subset \R^d$,
    \begin{align*}
        \sup_{u\in I_h,x\in S} \Big|
        \frac{1}{Th}\sum_{a=1}^TK_{h,1} \Big(u-\frac{t}{T}\Big)g\Big(\frac{t}{T},x\Big)-g(u,x)\Big| = \bigO\Big(\frac{1}{Th^2}\Big) + o(h).
    \end{align*}
\end{lemma}

%\end{table}

\bibliographystyle{plain}

\begin{thebibliography}{10}

\bibitem{Ahmedetal2020}
H.I.E.S. Ahmed, B.~Salha, R, and H.O. EL-Sayed.
\newblock Adaptive weighted nadaraya-watson estimation of the conditional
  quantiles by varying bandwidth.
\newblock {\em Communications in Statistics - Simulation and Computation},
  49(5):1105--1117, 2020.

\bibitem{AhsenVidyasagar2014}
M.E. Ahsen and M.~Vidyasagar.
\newblock Mixing coefficients between discrete and real random variables:
  Computation and properties.
\newblock {\em IEEE Transactions on Automatic Control}, 59(1):34--47, 2014.

\bibitem{Amatoetal2020}
F.~Amato, M.~Laib, F.~Guignard, and M.~Kanevski.
\newblock Analysis of air pollution time series using complexity-invariant
  distance and information measures.
\newblock {\em Physica A: Statistical Mechanics and Its Applications}, 547,
  2020.

\bibitem{Aueetal2015}
A.~Aue, D.~Dubart~Nourinho, and S.~Hörmann.
\newblock On the prediction of stationary functional time series.
\newblock {\em J. Amer. Statist. Assoc.}, 110:378–392, 2015.

\bibitem{AuevanDelft2020}
A.~Aue and A.~van Delft.
\newblock Testing for stationarity of functional time series in the frequency
  domain.
\newblock {\em Ann. Statist.}, 48:2505–2547, 2020.

\bibitem{BayraktarGuo2021}
E.~Bayraktar and G.~Guo.
\newblock {Strong equivalence between metrics of Wasserstein type}.
\newblock {\em Electronic Communications in Probability}, 26:1--13, 2021.

\bibitem{Bernstein1927}
S.~N. Bernstein.
\newblock Sur l’extension du théorème limite du calcul des probabilités
  aux sommes de quantités dépendantes.
\newblock {\em Math. Ann.}, 97:1--59, 1927.

\bibitem{Birretal2017}
S.~Birr, S.~Volgushev, T.~Kley, H.~Dette, and M.~Hallin.
\newblock {Quantile spectral analysis for locally stationary time series}.
\newblock {\em Journal of the Royal Statistical Society. Series B (Statistical
  Methodology)}, 79(5):1619–1643, 2017.

\bibitem{Bonnotte2013}
N.~Bonnotte.
\newblock {\em Unidimensional and Evolution Methods for Optimal
  Transportation}.
\newblock PhD thesis, Universit{\'e} paris sud, 2013.

\bibitem{Bosq2012}
D.~Bosq.
\newblock {\em Nonparametric statistics for stochastic processes: estimation
  and prediction}.
\newblock Springer Science and Business Media, 2012.

\bibitem{MR4783436}
O.~Bouanani and S.~Bouzebda.
\newblock Limit theorems for local polynomial estimation of regression for
  functional dependent data.
\newblock {\em AIMS Math.}, 9(9):23651--23691, 2024.

\bibitem{MR4737023}
S.~Bouzebda.
\newblock Weak convergence of the conditional single index {$U $}-statistics
  for locally stationary functional time series.
\newblock {\em AIMS Math.}, 9(6):14807--14898, 2024.

\bibitem{Bradley2005}
R.~Bradley.
\newblock Basic properties of strong mixing conditions. a survey and some open
  questions.
\newblock {\em Probability Survey}, 2:107–144, 2005.

\bibitem{Bugnietal2009}
F.~A. Bugni, P.~Hall, J.~Horowitz, and G.~R. Neumann.
\newblock Goodness-of-fit tests for functional data.
\newblock {\em The Econometrics Journal}, 12:S1–S18, 2009.

\bibitem{CarrascoChen2002}
M.~Carrasco and X~Chen.
\newblock Mixing and moment properties of various {GARCH} and stochastic
  volatility models.
\newblock {\em Econometric Theory}, 18(1):17--39, 2002.

\bibitem{Chenetal2016}
S.X. Chen, L.~Lei, and Y.~Tu.
\newblock Functional coefficient moving average model with applications to
  forecasting chinese cpi.
\newblock {\em Statistica Sinica}, 26:1649–1672, 2016.

\bibitem{Dadashovaetal2021}
B.~Dadashova, X.~Li, S.~Turner, and P.~Koeneman.
\newblock Multivariate time series analysis of traffic congestion measures in
  urban areas as they relate to socioeconomic indicators.
\newblock {\em Socio-Economic Planning Sciences}, 75, 2021.

\bibitem{Dahlhaus1996}
R.~Dahlhaus.
\newblock Fitting time series models to nonstationary processes.
\newblock {\em Ann. Statist.}, 25:1--37, 1996.

\bibitem{DAHLHAUS2012351}
R.~Dahlhaus.
\newblock Locally stationary processes.
\newblock In Tata {Subba Rao}, Suhasini {Subba Rao}, and C.R. Rao, editors,
  {\em Time Series Analysis: Methods and Applications}, volume~30 of {\em
  Handbook of Statistics}, pages 351--413. Elsevier, 2012.

\bibitem{Dahlhausetal2019}
R.~Dahlhaus, S.~Richter, and W.~B. Wu.
\newblock Towards a general theory for nonlinear locally stationary processes.
\newblock {\em Bernoulli}, 25(2):1013–1044, 2019.

\bibitem{DahlhausSubbaRao2006}
R.~Dahlhaus and S.~Subba~Rao.
\newblock Statistical inference for timevarying arch processes.
\newblock {\em Ann. Statist.}, 34:1075–1114, 2006.

\bibitem{Davydov1973}
J.A. Davydov.
\newblock Mixing conditions for markov chains.
\newblock {\em Teor. Verojatnost. i Primenen.}, pages 321--338, 1973.

\bibitem{Dedeckeretal2007}
J.~Dedecker, P.~Doukhan, G.~Lang, J.R. Leon, S.~Louhichi, and C.~Prieur.
\newblock {\em Weak dependence: with examples and applications}.
\newblock Springer Science and Business Media, 2007.

\bibitem{DedeckerMerleved2017}
J.~Dedecker and F.~Merlevede.
\newblock {Behavior of the Wasserstein distance between the empirical and the
  marginal distributions of stationary $\alpha$-dependent sequences}.
\newblock {\em Bernoulli}, 23(3):2083--2127, 2017.

\bibitem{DedeckerPrieur2005}
J.~Dedecker and C.~Prieur.
\newblock New dependence coefficients. examples and applications to statistics.
\newblock {\em Probab. Theory Relat. Fields}, 132:203–236, 2005.

\bibitem{Dombry2023}
C.~Dombry, T.~Modeste, and R.~Pic.
\newblock Stone's theorem for distributional regression in {W}asserstein
  distance.
\newblock {\em J. Nonparametr. Stat.}, 0(0):1--23, 2024.

\bibitem{Doukhan1994}
P.~Doukhan.
\newblock {\em Mixing: Properties and Examples}.
\newblock Springer-Verlag New York, 1994.

\bibitem{FanMasry1992}
J.~Fan and E.~Masry.
\newblock Multivariate regression estimation with errors-in-variables:
  asymptotic normality for mixing processes.
\newblock {\em Journal of Multivariate Analysis}, 43:237--271, 1992.

\bibitem{Ferraty&Vieu2006}
F.~Ferraty and P.~Vieu.
\newblock {\em Nonparametric Functional Data Analysis}.
\newblock Springer, 233 Spring Street, New York,NY10013, USA, 2006.

\bibitem{Guerardetal2020}
J.~Guerard, D.~Thomakos, and F.~Kyriazi.
\newblock Automatic time series modeling and forecasting: A replication case
  study of forecasting real gdp, the unemployment rate and the impact of
  leading economic indicators.
\newblock {\em Cogent Economics and Finance}, 8(1), 2020.

\bibitem{Halletal1999}
P.~Hall, R.~C.~L. Wolff, and Q.~Yao.
\newblock Methods for estimating a conditional distribution function.
\newblock {\em Journal of the American Statistical Association}, 94:154--163,
  1999.

\bibitem{HallinSegers2021}
M.~Hallin, G.~Mordant, and J.~Segers.
\newblock Multivariate goodness-of-fit tests based on {W}asserstein distance.
\newblock {\em Electron. J. Stat.}, 15(1):1328--1371, 2021.

\bibitem{Hansen2008}
B.~E. Hansen.
\newblock Uniform convergence rates for kernel estimation with dependent data.
\newblock {\em Econometric Theory}, 24(3):726--748, 2008.

\bibitem{Haslbecketal2020}
J.M.B. Haslbeck, L.F. Bringmann, and L.J. Waldorp.
\newblock {A Tutorial on Estimating Time-Varying Vector Autoregressive Models}.
\newblock {\em Multivariate Behavioral Research}, 2020.

\bibitem{jaber2024analysis}
A.S. Jaber, T.~Rashid, M.~Rasheed, R.~S. Mahmood, and O.~Maalej.
\newblock Analysis of cauchy distribution and its applications.
\newblock {\em Journal of Positive Sciences}, 4(4), 2024.

\bibitem{Jiangetal2020}
Y.~Jiang, X.~Wu, and Y.~Guan.
\newblock Effect of ambient air pollutants and meteorological variables on
  covid-19 incidence.
\newblock {\em Infection Control and Hospital Epidemiology}, 41(9):1011--1015,
  2020.

\bibitem{Jingetal2023}
T.Y. Jing, N.A.A. Rahman, and I.~Mohamed.
\newblock An analysis of unemployment rate in malaysia.
\newblock {\em International Journal of Advanced Management and Finance}, 4(3),
  2023.

\bibitem{Kolluruetal2021}
S.S.R. Kolluru, A.K. Patra, T.~Nazneen, and S.~Shiva~Nagendra.
\newblock Association of air pollution and meteorological variables with
  covid-19 incidence: Evidence from five megacities in india.
\newblock {\em Infection Control and Hospital Epidemiology}, 195, 2021.

\bibitem{Kristensen2009}
D~Kristensen.
\newblock Uniform convergence rates of kernel estimators with heterogeneous
  dependent data.
\newblock {\em Econometric Theory}, 25:1433–1445, 2009.

\bibitem{Daisuke2022}
D.~Kurisu.
\newblock Nonparametric regression for locally stationary functional time
  series.
\newblock {\em Electron. J. Stat.}, 16(2):3973--3995, 2022.

\bibitem{Kurisuetal2022}
D.~Kurisu, R.~Fukami, and Y.~Koike.
\newblock Adaptive deep learning for nonlinear time series models.
\newblock {\em Bernoulli}, 31(1):240--270, 2025.

\bibitem{Lee2012}
O.~Lee.
\newblock {Exponential Ergodicity and $\beta$-Mixing Property for Generalized
  Ornstein-Uhlenbeck Processes}.
\newblock {\em Theoretical Economics Letters}, 2:21--25, 2012.

\bibitem{LI20241123}
Xixi Li and Jingsong Yuan.
\newblock Deeptvar: Deep learning for a time-varying var model with extension
  to integrated var.
\newblock {\em International Journal of Forecasting}, 40(3):1123--1133, 2024.

\bibitem{lubik2015time}
Thomas~A Lubik and Christian Matthes.
\newblock Time-varying parameter vector autoregressions: Specification,
  estimation, and an application.
\newblock {\em Estimation, and an Application}, 2015.

\bibitem{Manoleetal2022}
T.~Manole, S.~Balakrishnan, and L.~Wasserman.
\newblock Minimax confidence intervals for the sliced {W}asserstein distance.
\newblock {\em Electron. J. Stat.}, 16(1):2252--2345, 2022.

\bibitem{Masry2005}
E.~Masry.
\newblock Nonparametric regression estimation for dependent functional data:
  asymptotic normality.
\newblock {\em Stochastic Process. Appl.}, 115(1):155--177, 2005.

\bibitem{Masuda2007}
H.~Masuda.
\newblock Ergodicity and exponential {$\beta$}-mixing bounds for
  multidimensional diffusions with jumps.
\newblock {\em Stochastic Process. Appl.}, 117(1):35--56, 2007.

\bibitem{McDonaldetal2011}
D.J. McDonald, C.R. Shalizi, and M.~Schervish.
\newblock Estimating beta-mixing coefficients.
\newblock In {\em Proceedings of the 14th International Conference on
  Artificial Intelligence and Statistics (AISTATS)}, volume~15, Fort
  Lauderdale, FL, USA, 2011. JMLR: W and CP.

\bibitem{Miyamaetal2020}
T.~Miyama, H.~Matsui, K.~Azuma, C.~Minejima, Y.~Itano, N.~Takenaka, and
  M.~Ohyama.
\newblock Time series analysis of climate and air pollution factors associated
  with atmospheric nitrogen dioxide concentration in japan.
\newblock {\em International Journal of Environmental Research and Public
  Health}, 17(24), 2020.

\bibitem{Mokkadem1998}
A.~Mokkadem.
\newblock Mixing properties of {ARMA} processes.
\newblock {\em Stochastic Process. Appl.}, 29(2):309--315, 1988.

\bibitem{MR166874}
\`E.~A. Nadaraja.
\newblock On a regression estimate.
\newblock {\em Teor. Verojatnost. i Primenen.}, 9:157--159, 1964.

\bibitem{Nadjahietal2021}
K.~Nadjahi, A.~Durmus, P.E. Jacob, R.~Badeau, and U.~Şimşekli.
\newblock {Fast approximation of the sliced Wasserstein distance Using
  concentration of random projections}.
\newblock {35th Conference on Neural Information Processing Systems (NeurIPS
  2021), France. ffhal-03494781}, December 2021.

\bibitem{Nietertetal2021}
S.~Nietert, Z.~Goldfeld, and K.~Kato.
\newblock {Smooth p-Wasserstein Distance: Structure, Empirical Approximation,
  and Statistical Applications}.
\newblock In {\em Proceedings of the 38th International Conference on Machine
  Learning}, volume 139. PMLR, 2021.

\bibitem{OtneimTjøstheim2016}
H.~Otneim and D.~Tj{\o}stheim.
\newblock Conditional density estimation using the local {G}aussian
  correlation.
\newblock {\em Stat. Comput.}, 28(2):303--321, 2018.

\bibitem{Owen1986}
A.B. Owen.
\newblock {\em Nonparametric conditional estimation}.
\newblock PhD thesis, Stanford University, 1986.

\bibitem{PanaretosZemel2019}
V.~M. Panaretos and Y.~Zemel.
\newblock Statistical aspects of {W}asserstein distances.
\newblock {\em Annu. Rev. Stat. Appl.}, 6:405--431, 2019.

\bibitem{Peligrad2002}
M.~Peligrad.
\newblock Some remarks on coupling of dependent random variables.
\newblock {\em Statist. Probab. Lett.}, 60(2):201--209, 2002.

\bibitem{PeyreCuturi2020}
G.~Peyré and M.~Cuturi.
\newblock Computational optimal transport.
\newblock {\em Foundations and Trends in Machine Learning}, 11:355--607, 2020.

\bibitem{Poinas2020}
A.~Poinas.
\newblock A bound of the {$\beta$}-mixing coefficient for point processes in
  terms of their intensity functions.
\newblock {\em Statist. Probab. Lett.}, 148:88--93, 2019.

\bibitem{RichterDahlhaus2019}
S.~Richter and R.~Dahlhaus.
\newblock Cross validation for locally stationary processes.
\newblock {\em Ann. Statist.}, 47(4):2145--2173, 2019.

\bibitem{Rio2017}
E.~Rio.
\newblock {\em Asymptotic theory of weakly dependent random processes}.
\newblock Springer, Berlin Heidelberg, 2017.

\bibitem{rojo2013heavy}
Javier Rojo.
\newblock Heavy-tailed densities.
\newblock {\em Wiley Interdisciplinary Reviews: Computational Statistics},
  5(1):30--40, 2013.

\bibitem{Silverman1998}
B.W. Silverman.
\newblock {\em Density estimation for statistics and data analysis}.
\newblock Chapman and Hall/CRC, FL, 1998.

\bibitem{SoukariehBouzebda2023}
I.~Soukarieh and S.~Bouzebda.
\newblock Weak convergence of the conditional {$U$}-statistics for locally
  stationary functional time series.
\newblock {\em Stat. Inference Stoch. Process.}, 27(2):227--304, 2024.

\bibitem{Truquet2019}
L.~Truquet.
\newblock Local stationarity and time-inhomogeneous {M}arkov chains.
\newblock {\em Ann. Statist.}, 47(4):2023--2050, 2019.

\bibitem{Veraverbekeetal2014}
N.~Veraverbeke, I.~Gijbels, and M.~Omelka.
\newblock {Preadjusted non-parametric estimation of a conditional distribution
  function}.
\newblock {\em J. R. Statist. Soc. B}, 76(2):399–438, 2014.

\bibitem{Vidyasagar1997}
M.~Vidyasagar.
\newblock {\em A Theory of Learning and Generalization: With Applications to
  Neural Networks and Control Systems}.
\newblock Springer Verlag, Berlin, 1997.

\bibitem{Villani2009OToldnew}
C.~Villani.
\newblock {\em Optimal transport: old and new}.
\newblock Springer, Verlag Berlin Heidelberg, 2009.

\bibitem{vogt2012}
M.~Vogt.
\newblock Nonparametric regression for locally stationary time series.
\newblock {\em Ann. Statist.}, 40(5):2601--2633, 2012.

\bibitem{MR185765}
G.~S. Watson.
\newblock Smooth regression analysis.
\newblock {\em Sankhy\={a} Ser. A}, 26:359--372, 1964.

\bibitem{Wengetal2018}
B.~Weng, W.~Martinez, T.~Tsai, C.~Li, L.~Lu, J.R. Barth, and F.M. Megahed.
\newblock Macroeconomic indicators alone can predict the monthly closing price
  of major u.s. indices: Insights from artificial intelligence, time-series
  analysis and hybrid models.
\newblock {\em Applied Soft Computing}, 71:685--697, 2018.

\bibitem{XuHuang2022}
X.~Xu and Z.~Huang.
\newblock {Central limit theorem for the 1-Wasserstein distance and the
  max-sliced 1-Wasserstein distance }.
\newblock {arXiv preprint arXiv:2205.14624v2. Available at
  http://arxiv.org/abs/2205.14624v2}, September 2022.

\bibitem{ZangWu2015}
T.~Zhang and W.~B. Wu.
\newblock Time-varying nonlinear regression models: Nonparametric estimation
  and model selection.
\newblock {\em Ann. Statist.}, 43:741–768, 2015.

\end{thebibliography}
\end{document}